\documentclass[a4paper, oneside, 11pt]{amsart}
\usepackage{import}

\title[Categories of representations of fixed-points conformal nets]{Balanced tensor categories of representations of fixed-points conformal nets}
\author{Adrià Marín-Salvador}
\date{}

\newcommand{\Z}{\mathbb{Z}}
\newcommand{\R}{\mathbb{R}}
\newcommand{\C}{\mathbb{C}}

\newcommand{\Diff}{\text{Diff}}

\newcommand{\Mob}{\mathbf{M\ddot{o}b}}
\newcommand{\A}{\mathcal{A}}
\newcommand{\Rot}{\text{Rot}}

\newcommand{\id}{\text{id}}

\newcommand{\Hom}{\text{Hom}}
\newcommand{\Aut}{\text{Aut}}
\newcommand{\Rep}{\text{Rep}}

\newcommand{\Dcal}{\mathcal{D}}

\newcommand{\Ccal}{\mathcal{C}}
\newcommand{\End}{\text{End}}

\newcommand{\Jcal}{\mathcal{J}}
\newcommand{\ConfR}{\widetilde{\text{Conf}}_{2}(S^1)}

\newcommand{\B}{\mathbb{B}}

\newcommand{\Loc}{\text{Loc}}
\newcommand{\DHR}{\text{DHR}}
\newcommand{\1}{\mathbf{1}}
\newcommand{\M}{\text{Mod}(\Gamma)}

\usepackage{amssymb}
\usepackage{amsfonts}
\usepackage{amsmath}
\usepackage{quiver}
\usepackage{amsthm}
\usepackage{comment}
\usepackage[english]{babel}
\usepackage{quiver}
\usepackage{nicefrac}
\usepackage[a4paper]{geometry}
\usepackage{amssymb}
\usepackage{mathtools}
\usepackage{stackrel}
\usepackage[hidelinks]{hyperref}
\usepackage{mathrsfs} 
\usetikzlibrary{decorations.pathmorphing}
\usetikzlibrary{calc}
\usetikzlibrary{decorations.markings}
\usetikzlibrary{fadings,decorations.pathreplacing}
\usetikzlibrary{matrix,arrows}
\usetikzlibrary{arrows,calc,decorations.pathreplacing,decorations.markings,shapes.geometric,shadows}

\usepackage{xcolor}

\tikzset{-dot-/.style={decoration={
  markings,
  mark=at position 0.5 with {\fill circle (2pt);}},postaction={decorate}}}

\tikzset{
	Fdot/.style={circle, draw, fill, inner sep=0pt}, 
	Odot/.style={circle, draw, inner sep=0.1pt, minimum size=0.1cm}
	}

  \vfuzz \hfuzz
\usepackage{xcolor}
\usetikzlibrary{arrows}

\theoremstyle{definition}
\newtheorem{definition}{Definition}[section]
\newtheorem{proposition}[definition]{Proposition}
\newtheorem{lemma}[definition]{Lemma}
\newtheorem{theorem}[definition]{Theorem}
\newtheorem{corollary}[definition]{Corollary}

\newtheorem{remark}[definition]{Remark}

\newtheorem*{theorem*}{Theorem}
\newtheorem*{definition*}{Definition}
\newtheorem*{example*}{Example}
\newtheorem*{corollary*}{Corollary}
\newtheorem*{conjecture*}{Conjecture}

\definecolor{myviolet}{HTML}{7D3C98}
\definecolor{myorange}{HTML}{F39C12}
\definecolor{myblue}{HTML}{2E86C1}
\definecolor{mygreen}{HTML}{1E8449}
\definecolor{myred}{HTML}{C0392B}

\tikzstyle{ipe stylesheet} = [
  ipe import,
  even odd rule,
  line join=round,
  line cap=butt,
  ipe pen normal/.style={line width=0.4},
  ipe pen heavier/.style={line width=0.8},
  ipe pen fat/.style={line width=1.2},
  ipe pen ultrafat/.style={line width=2},
  ipe pen normal,
  ipe mark normal/.style={ipe mark scale=3},
  ipe mark large/.style={ipe mark scale=5},
  ipe mark small/.style={ipe mark scale=2},
  ipe mark tiny/.style={ipe mark scale=1.1},
  ipe mark normal,
  /pgf/arrow keys/.cd,
  ipe arrow normal/.style={scale=7},
  ipe arrow large/.style={scale=10},
  ipe arrow small/.style={scale=5},
  ipe arrow tiny/.style={scale=3},
  ipe arrow normal,
  /tikz/.cd,
  ipe arrows, 
  <->/.tip = ipe normal,
  ipe dash normal/.style={dash pattern=},
  ipe dash dotted/.style={dash pattern=on 1bp off 3bp},
  ipe dash dashed/.style={dash pattern=on 4bp off 4bp},
  ipe dash dash dotted/.style={dash pattern=on 4bp off 2bp on 1bp off 2bp},
  ipe dash dash dot dotted/.style={dash pattern=on 4bp off 2bp on 1bp off 2bp on 1bp off 2bp},
  ipe dash normal,
  ipe node/.append style={font=\normalsize},
  ipe stretch normal/.style={ipe node stretch=1},
  ipe stretch normal,
  ipe opacity 10/.style={opacity=0.1},
  ipe opacity 30/.style={opacity=0.3},
  ipe opacity 50/.style={opacity=0.5},
  ipe opacity 75/.style={opacity=0.75},
  ipe opacity opaque/.style={opacity=1},
  ipe opacity opaque,
]
\definecolor{red}{rgb}{1,0,0}
\definecolor{blue}{rgb}{0,0,1}
\definecolor{green}{rgb}{0,1,0}
\definecolor{yellow}{rgb}{1,1,0}
\definecolor{orange}{rgb}{1,0.647,0}
\definecolor{gold}{rgb}{1,0.843,0}
\definecolor{purple}{rgb}{0.627,0.125,0.941}
\definecolor{gray}{rgb}{0.745,0.745,0.745}
\definecolor{brown}{rgb}{0.647,0.165,0.165}
\definecolor{navy}{rgb}{0,0,0.502}
\definecolor{pink}{rgb}{1,0.753,0.796}
\definecolor{seagreen}{rgb}{0.18,0.545,0.341}
\definecolor{turquoise}{rgb}{0.251,0.878,0.816}
\definecolor{violet}{rgb}{0.933,0.51,0.933}
\definecolor{darkblue}{rgb}{0,0,0.545}
\definecolor{darkcyan}{rgb}{0,0.545,0.545}
\definecolor{darkgray}{rgb}{0.663,0.663,0.663}
\definecolor{darkgreen}{rgb}{0,0.392,0}
\definecolor{darkmagenta}{rgb}{0.545,0,0.545}
\definecolor{darkorange}{rgb}{1,0.549,0}
\definecolor{darkred}{rgb}{0.545,0,0}
\definecolor{lightblue}{rgb}{0.678,0.847,0.902}
\definecolor{lightcyan}{rgb}{0.878,1,1}
\definecolor{lightgray}{rgb}{0.827,0.827,0.827}
\definecolor{lightgreen}{rgb}{0.565,0.933,0.565}
\definecolor{lightyellow}{rgb}{1,1,0.878}
\definecolor{black}{rgb}{0,0,0}
\definecolor{white}{rgb}{1,1,1}

\usetikzlibrary{arrows.meta,patterns}
\usetikzlibrary{ipe}

\begin{document}

\begin{abstract}
    Let $\A$ be a (not necessarily rational) conformal net with a faithful action of a finite group $G$. Let $\Rep^G(\A)$ be the $G$-crossed balanced $\mathrm{W}^*$-tensor category of $G$-twisted representations of $\A$ as introduced in \cite{GcrossedbraidedRep}. We show that there is an equivalence of balanced $\mathrm{W}^*$-tensor categories $(\Rep^G(\A))^G\cong \Rep(\A^G)$ between the $G$-equivariantization of $\Rep^G(\A)$ and the category of representations of the fixed-points conformal net $\A^G$. This generalizes to the non-rational case the equivalence of braided tensor categories $(\Rep^G(\A))^G\cong \Rep(\A^G)$ for $\A$ rational appearing (in the language of localized endomorphisms) in \cite{muger05}, and it also includes the balances.
\end{abstract}

\maketitle

\tableofcontents

\section{Introduction}
\addtocontents{toc}{\protect\setcounter{tocdepth}{-1}}

Conformal nets are one of the most developed mathematical approaches to 2-dimensional chiral conformal field theory (CFT). A conformal net consists of a vacuum Hilbert space $H_0$ together with a net of von Neumann algebras of observables $I\subset S^1\mapsto \A(I)\subset B(H_0)$ parametrized by intervals of $S^1$. The vacuum $H_0$ is also equipped with a projective action of the group $\Diff^+(S^1)$ of orientation preserving diffeomorphisms of $S^1$ in a way that is compatible with the net $\A$. See Definition \ref{def: ConformalNet} for a precise definition and \cite{bmt88, bsm90, bgl93, gf93, w95, kl04, bdh15} for an introduction to the theory of conformal nets. 

Given a conformal net $\A$, one can consider actions of $\A$ on Hilbert spaces other than the vacuum $H_0$. More precisely, a representation of $\A$ consists of a Hilbert space $H$ together with $*$-actions of all the von Neumann algebras $\A(I)$ which are compatible with the inclusions $\A(J)\subset \A(I)$ for $J\subset I\subset S^1$. We write $\Rep(\A)$ for the category of representations of $\A$ and bounded linear maps intertwining the actions of all the algebras $\A(I)$. It is well known that $\Rep(\A)$ is a braided tensor category. We refer the reader to \cite[Sec. 2]{frs}, \cite[Sec. 7]{lon89} and \cite[Sec. IV.4]{gf93} for the definition of the tensor product of representations and the braiding, and \cite{bdh15, bicommutantfromnets, Gui21} for a more modern treatment. The algebras of endomorphisms of $\Rep(\A)$ are von Neumann algebras, and $\Rep(\A)$ is furthermore a braided $\mathrm{W}^*$-tensor category. See \cite{henriques2024completewcategories} for an introduction to $\mathrm{W}^*$-tensor categories. 

Most of the literature on conformal nets focuses on the subclass of \emph{rational} conformal nets (we note that \emph{rationality} and \emph{complete rationality} are equivalent by \cite{MR2100058}). One of the main reasons for this restriction is that, whenever $\A$ is a rational conformal net, the braided tensor category $\Rep(\A)$ is a \emph{modular tensor category} \cite{klm01}. In particular, $\Rep(\A)$ is a finite semisimple rigid tensor category, and it carries a \emph{balance} (or \emph{categorical twist}). Recall that a balance on a tensor category $\mathcal{C}$ with a braiding $\beta$ consists of a natural family of isomorphisms
\[
\theta_X: X\to X
\]
satisfying that $\theta_{X\otimes Y} = (\theta_X\otimes\theta_Y)\circ \beta_{Y,X}\circ\beta_{X,Y}$ for all $X,Y\in\mathcal{C}$. If $\mathcal{C}$ is furthermore~rigid and $\theta$ is compatible with the evaluation and coevaluation maps, the balance is called a \emph{ribbon~structure}.

When $\A$ is not rational, the category $\Rep(\A)$ is in general not finite, semisimple, or rigid: it may have infinitely (or continuously) many simple objects, and a generic representation of $\A$ is a \emph{direct integral} of irreducible representations as opposed to a direct sum. Whilst there is a vast literature on finite semisimple rigid tensor categories, much less is known about categories with direct integrals and continuously many simple objects. For example, the fact that $\Rep(\A)$ is a balanced tensor category even when $\A$ is not rational only appeared in the literature in \cite{marinsalvador2026balancedstructurecategoryrepresentations}. In this paper, we do not assume rationality at any point.

Even when non-rational conformal nets are considered, most of the literature approaches the representation theory of conformal nets through the language of localized endomorphisms \cite{MR258394}. Fix a point $\mathrm{p}\in S^1$ and let $\A_\infty$ be the $*$-subalgebra of $B(H_0)$ constructed as the union of the algebras $\A(I)$ for $I\subset S^1$ not containing $\mathrm{p}$. The endomorphisms of $\A_\infty$ are the objects of a category $\End\,\A_\infty$ whose morphisms are elements of $\A_\infty$ intertwining the actions of the endomorphisms. An endomorphism $\rho\in \End\,\A_\infty$ is said to be localized if there is an interval $\mathrm{p}\notin I\subset S^1$ such that $\rho$ acts like the identity outside of $I$. The subcategory $\DHR(\A)\subset \End\,\A_\infty$ (for Doplicher–Haag–Roberts) whose objects are localized endomorphisms and whose morphisms are morphisms in $\End\,\A_\infty$ contained in $\A(I)$ is a braided tensor category equivalent to $\Rep(\A)$, see \cite[Thm. 2.31]{muger05} and \cite[Sec.~6]{Gui21}.

\textbf{Fixed-points conformal nets and orbifolding.} Given a CFT and a compact subgroup $G$ of symmetries of the CFT, there exists a process, known as \emph{orbifolding}, which constructs a new theory by \emph{modding out} by the symmetries $G$. The study of orbifolding in CFTs began in \cite{MR1003430} in the context of \emph{rational} CFTs and finite subgroups of symmetries. Mathematically, if the CFT is encoded as a conformal net $\A$, one considers a compact group $G$ and a continuous group homomorphism from $G$ into the group of automorphisms of $\A$. An automorphism $\varphi$ of $\A$ is a collection $\{\varphi_I\}_{I\subset S^1}$ of automorphisms of the von Neumann algebras $\A(I)$ which is implemented in the vacuum Hilbert space in the sense that there exists a unitary $V_\varphi\in U(H_0)$ such that $\varphi(x) = V_\varphi x V_\varphi^*$ for all intervals $I\subset S^1$ and all elements $x\in \A(I)$. We write $\Aut(\A)$ for the group of automorphisms of $\A$, which inherits a topology from that of $U(H_0)$. Given such an action $G\to \Aut(\A)$, orbifolding amounts to taking the fixed-points conformal net $\A^G$, defined as follows.

Let us identify an element of $G$ with the automorphism of $\A$ it produces.  The fixed-points conformal net $\A^G$ is the net that assigns to an interval $I\subset S^1$ the von Neumann algebra $\A^G(I) := \A(I)^G:=\{x\in\A(I)\,|\,gx = x\,\text{for all $g\in G$}\}$ acting on the fixed-points Hilbert space $H_0^G:=\{\xi\in H_0\,|\,\text{$V_g\xi = \xi$ for all $g\in G$}\}$. It is natural to ask whether one can understand the category of representations of $\A^G$ in terms of representations of $\A$. It is known that $\Rep(\A)$ does not contain all the necessary information to compute $\Rep(\A^G)$, but one has to, in addition, consider \emph{twisted representations} of $\A$. 

\textbf{Twisted representations of conformal nets.} Given an automorphism $\varphi\in \Aut(\A)$, a $\varphi$-twisted representation of $\A$ is, roughly speaking, a Hilbert space $H$ together with a family of $*$-actions of the algebras $\A(I)$ compatible with the inclusions $\A(J)\subset \A(I)$ for $J\subset I$ everywhere away from a fixed point $\mathrm{p}\in S^1$, and only up to an application of $\varphi$ at the point $\mathrm{p}$, see Definition \ref{def: TwistedRep}. We write $\Rep^\varphi(\A)$ for the $\mathrm{W}^*$-category of $\varphi$-twisted representations of $\A$. Twisted representations were introduced, in the context of localized endomorphisms, in \cite{muger05} precisely to study representations of fixed-points conformal nets.

Given an automorphism $\varphi\in\Aut(\A)$, Müger constructs in \cite{muger05} a category $\varphi-\Loc_I(\A)\subset \End\,\A_\infty$ of endomorphisms which are $\varphi$-localized in some fixed interval $\mathrm{p}\notin I\subset S^1$, meaning that they act by the identity on the left of $I$ and by $\varphi$ on the right. Note that, if $\varphi = \id$ is the identity automorphism, then $\id-\Loc_I(\A) \cong \DHR(\A)$. Given a faithful action $G\to \Aut(\A)$ of a discrete group on $\A$ and identifying an element $g\in G$ with the automorphism it produces, Müger showed that the category $G-\Loc_I(\A) : = \bigoplus\limits_{g\in G}g-\Loc_I(\A)$ has the structure of a $G$-crossed braided tensor category. In particular, $G-\Loc_I(\A)$ admits a tensor product (given by composition of endomorphisms) covering the multiplication of $G$, as well as an action of $G$ by tensor automorphisms and a version of a braiding in the context of a group acting on a graded tensor category. In \cite{GcrossedbraidedRep}, we produced explicitly a $G$-crossed braided tensor structure on the $\mathrm{W}^*$-category $\Rep^G(\A):=\bigoplus\limits_{g\in G}\Rep^g(\A)$ and provided an equivalence of $G$-crossed braided $\mathrm{W}^*$-tensor categories $\Rep^G(\A)\cong G-\Loc_I(\A)$. Working in the language of representations, we also provided a $G$-crossed balance on $\Rep^G(\A)$, that is, a natural family of unitary isomorphisms $\theta_H: H\to T_g(H)$, for $g\in G$ and $H\in\Rep^g(\A)$ compatible with the $G$-crossed braiding, where $T_g(H)$ is the image of $H$ under the action of~$g$.
\begin{theorem*}(\cite{GcrossedbraidedRep})
    Let $\A$ be a conformal net being acted on faithfully by a discrete group $G$. Then, the $\mathrm{W}^*$-category $\Rep^G(\A)$ of $G$-twisted representations of $\A$ has the structure of a $G$-crossed balanced $\mathrm{W}^*$-tensor category. In addition, its underlying $G$-crossed braided $\mathrm{W}^*$-tensor category is equivalent to $G-\Loc_I(\A).$
\end{theorem*}
The $G$-equivariantization of a $G$-crossed braided (balanced) $\mathrm{W}^*$-tensor category is canonically a braided (resp. balanced) $\mathrm{W}^*$-tensor category. Hence, the $G$-equivariantization $(G-\Loc_I(\A))^G$ of $G-\Loc_I(\A)$ is a braided $\mathrm{W}^*$-tensor category, and the $G$-equivariantization $(\Rep^G(\A))^G$ of $\Rep^G(\A)$ is a balanced $\mathrm{W}^*$-tensor category.

\textbf{The category of representations of the fixed-points conformal net.} Assume now that the group $G$ is finite. Then, one of the main results in \cite{muger05} is that, if $\A$ is a \emph{rational} conformal net, $\DHR(\A^G)$ is completely controlled by $G-\Loc_I(\A)$, as follows.
\begin{theorem*}(\cite{muger05})
    Let $\A$ be a \emph{rational} conformal net and $G$ a finite group acting faithfully on $\A$. Then, there is an equivalence of braided tensor categories $\DHR(\A^G)\cong (G-\Loc_I(\A))^G$. 
\end{theorem*}
In the language of representations, we obtain an equivalence of braided $(\mathrm{W}^*-)$tensor categories $\Rep(\A^G)\cong (\Rep^G(\A))^G$ for a faithful action of a finite group $G$ on a \emph{rational} conformal net $\A$. The purpose of the current paper is to generalize the result above in two directions. Firstly, we remove the hypothesis of $\A$ being rational (while keeping $G$ finite), and secondly we upgrade the equivalence to an equivalence of balanced $\mathrm{W}^*$-tensor categories. The main result of the paper is Theorem \ref{Thm: EquivariantizationCompatibleBalance}.
\begin{theorem*}
    Let $\A$ be a (not necessarily rational) conformal net being acted on faithfully by a finite group $G$. Then, there is an equivalence of balanced $\mathrm{W}^*$-tensor categories \begin{equation}\label{eq: intro}\Rep(\A^G)\cong (\Rep^G(\A))^G.\end{equation}
\end{theorem*}
 The underlying functor of this equivalence is the functor $\mathfrak{R}: (\Rep^G(\A))^G\to \Rep(\A^G)$ that restricts a $G$-equivariant twisted representation of $\A$ on a Hilbert space $H$ to an honest $\A^G$-representation on the subspace of $G$-invariant vectors of $H$. 
 
 Our proof strategy is as follows. The vacuum $H_0$, seen as an $\A^G$-representation by restriction along $\A^G\subset\A$, has a canonical commutative $\mathrm{C}^*$-Frobenius algebra (or $Q$-system) structure in $\Rep(\A^G)$. Denoting this algebra by $\Gamma$, we argue that the category $\M$ of unitary modules over $\Gamma$ is a $G$-crossed braided $\mathrm{W}^*$-tensor category equivalent to $\Rep^G(\A)$, using similar arguments to those in \cite{Gui_2025}. This requires us to introduce $G$-crossed categorical extensions of conformal nets in Section \ref{sec: CategoricalExtensions}, following \cite{Gui21}. The fact that $\M$ and $\Rep^G(\A)$ are equivalent as $G$-crossed braided $\mathrm{W}^*$-tensor categories implies that their equivariantizations $\M^G$ and $(\Rep^G(\A))^G$ are equivalent as braided $\mathrm{W}^*$-tensor categories. Then, we construct a braided $\mathrm{W}^*$-tensor functor $\Rep(\A^G)\to \M^G$ and show that the composition $\Rep(\A^G)\to \M^G\cong (\Rep^G(\A))^G$ is an equivalence of braided $\mathrm{W}^*$-tensor categories with inverse the functor $\mathfrak{R}: (\Rep^G(\A))^G\to \Rep(\A^G).$ We borrow some of the arguments from \cite{MR4244264}, which proves a similar equivalence to \eqref{eq: intro} in the context of vertex~operator~algebras. It is then easy to show that $\mathfrak{R}: (\Rep^G(\A))^G\to \Rep(\A^G)$ is compatible with the balances on both sides.

 In \cite[App. A]{GcrossedbraidedRep}, we constructed an involutive structure on $\Rep^G(\A)$ compatible with the $G$-crossed balance. Then, both the $G$-equivariantization $(\Rep^G(\A))^G$ and the category $\Rep(A^G)$ are involutive $\mathrm{W}^*$-tensor categories. In Appendix \ref{Appendix}, we upgrade the equivalence $\Rep(\A^G)\cong (\Rep^G(\A))^G$ to also be compatible with these involutive structures.

These results are used in \cite{RepHeis} to produce the first explicit computation of a category of representations of a conformal net with irreducible representations that tensor to a direct integral of irreducible representations, as opposed to a direct sum.

\section*{Acknowledgements}
I am grateful to André Henriques and Bin Gui for their comments, and to Josep Fontana McNally for proofreading this paper. This work has been funded by the EPSRC grant EP/W524311/1.

For the purpose of Open Access, the author has applied a CC BY public copyright license to any Author Accepted Manuscript version arising from this submission.

\addtocontents{toc}{\protect\setcounter{tocdepth}{5}}

\section{Preliminaries}

\subsection{Conformal nets}
\label{Sec: ConfNets}

We denote by $S^1:=\{z\in \C\ |\ |z| = 1\}$ the standard circle, and define an \emph{interval} of $S^1$ to be an open, connected, non-empty, non-dense subset of $S^1$. We write $\Jcal$ for the collection of intervals of $S^1$. The group $\Mob$ is the group of Möbius transformations of $S^1$, that is, transformations of the form
\[
z\mapsto \frac{az+b}{\bar{b}z+\bar{a}}
\]
with $a,b\in \C$, $|a|^2-|b|^2 = 1$. Then, $\Mob$ is a Lie group isomorphic to $PSL(2, \R).$ We denote by $\Rot$ the subgroup of rotations of $\Mob$, and an anticlockwise rotation of angle $t$ by $R_t$. In this paper, all Hilbert spaces and $\mathrm{C}^*$-algebras are assumed to be separable.

\begin{definition}
    A \textit{Möbius covariant net on $S^1$} is a tuple $(H_0, \A, U, \Omega)$ where $H_0$ is a Hilbert space equipped with a nonzero vector $\Omega\in H_0$, $U$ is a strongly continuous unitary representation of $\Mob$ on $H_0$ and $\A$ is an assignment of a von Neumann algebra $\A(I)$ acting on $H_0$ for every interval $I\in\Jcal$. This data is required to satisfy, for every $I,J\in\Jcal$ and $\varphi\in\Mob$,
    \begin{enumerate}
        \item isotony: if $J\subset I$, then $\A(J)\subset \A(I)$;
        \item locality: if $I\cap J = \emptyset$, then $\A(J)$ and $\A(I)$ commute in $B(H_0)$;
        \item Möbius covariance: $U(\varphi)\A(I)U(\varphi)^* = \A(\varphi I)$;
        \item positivity of the energy: the representation $U$ is of positive energy, meaning that the conformal Hamiltonian $L_0$, defined by $U(R_t) = e^{itL_0}$ is positive;
        \item uniqueness of the vacuum: the vector $\Omega\in H_0$ is the unique vector, up to a constant, which is invariant under $U$;
        \item cyclicity of the vacuum: $\Omega$ is cyclic for the von Neumann  algebra $\A(S^1):=\bigvee\limits_{I\in \Jcal}\A(I)\subset B(H_0)$.
    \end{enumerate}
\end{definition}

Let $\Diff^+(S^1)$ denote the group of orientation-preserving diffeomorphisms of $S^1$. A strongly continuous projective unitary representation $V$ of $\Diff^+(S^1)$ on a Hilbert space $H$ is a strongly continuous homomorphism $V: \Diff^+(S^1)\to \mathcal{P}U(H)$. Note that $\Mob\subset\Diff^+(S^1)$. We say that $V$ is an extension of a unitary representation $U$ of $\Mob$ on $H$ if, for every $\varphi\in \Mob$, it holds that $V(\varphi) = [U(\varphi)]$, where $[-]: U(H)\to \mathcal{P}U(H)$ denotes the projection map.

\begin{definition}\label{def: ConformalNet}
    A Möbius covariant net $(H_0, \A ,U, \Omega)$ is said to be a \emph{conformal net} if there is an extension of $U$ to a projective unitary representation of $\Diff^+(S^1)$, which we also denote by $U$, such that for all $I\in\Jcal$ and $\varphi\in \Diff^+(S^1)$, it holds that
    \begin{enumerate}
        \item $U(\varphi)\A(I) U(\varphi)^* = \A(\varphi I)$;
        \item if $\varphi|_{I} = \id_I$, then $\text{Ad}(U(\varphi))|_{\A(I)} = \id_{\A(I)}.$
    \end{enumerate}
\end{definition}

The following are well-known properties of any conformal net $(H_0, \A, U, \Omega)$. 
\begin{enumerate}
    \item Haag duality: the commutant of $\A(I)$ in $B(H_0)$ is $\A(I^c)$, where $I^c$ denotes the interior of $S^1\setminus I$ \cite[Thm. 2.19]{gf93};
    \item The Reeh-Schlieder Theorem: the vacuum vector $\Omega$ is cyclic and separating for every von Neumann algebra $\A(I)$ \cite[Cor. 2.8]{gf93};
    \item Each von Neumann algebra $\A(I)$ is a type $\mathrm{III}_1$ factor \cite[Prop. 1.2]{MR1410566};
    \item Additivity: for any interval $I\in \Jcal$ and any collection of intervals $\{I_\alpha\}_{\alpha \in A}$ with $I_\alpha\in \Jcal$ and $\bigcup\limits_{\alpha\in A}I_\alpha = I$, it holds that $\A(I) = \bigvee\limits_{\alpha\in A}\A(I_\alpha)$ \cite[p. 545]{MR1376431}.
\end{enumerate}

Fix a conformal net $(H_0,\A,U, \Omega)$, which we will simply denote by $\A$ from now on.

\begin{definition}\label{def: Rep}
    A \emph{representation} of a conformal net $\A$ consists of a Hilbert space $H$ and a collection of $*$-homomorphisms $\pi_I:\A(I)\to B(H)$ for every interval $I\in\Jcal$ such that
\[
\pi_I|_{\A(J)} = \pi_J
\]
whenever $J\in\Jcal$ is an interval such that $J\subset I$. We write $\Rep(\A)$ for the category whose objects are representations of $\A$ and whose morphisms are bounded linear maps intertwining the actions of all the algebras $\A(I)$.
\end{definition}

The conformal net $\A$ has a canonical representation on its vacuum Hilbert space $H_0$ by definition, which we call the \emph{vacuum representation of} $\A$. Given $I\in\Jcal$ and $x\in \A(I)$, we write $\pi_{0,I}(x)$ or $\pi_0(x)$ for the action of $x$ on the vacuum representation $H_0$.

The category $\Rep(\A)$ is a braided tensor category, the structure of which can be produced in different equivalent ways \cite{gf93, was98, bdh15,Gui21}.

We wish to consider more general representations of $\A$, those which are \emph{twisted} by an automorphism of the conformal net. An \emph{automorphism} $\varphi$ of $\A$ consists of a von Neumann algebra automorphism $\varphi_I: \A(I)\to\A(I)$ for every $I\in\Jcal$ such that for every inclusion $J\subset  I$ of intervals, the following diagram commutes
\[\begin{tikzcd}
	{\A(I)} & {\A(I)} \\
	{\A(J)} & {\A(J)}.
	\arrow["{\varphi_I}", from=1-1, to=1-2]
	\arrow[hook, from=2-1, to=1-1]
	\arrow["{\varphi_J}"', from=2-1, to=2-2]
	\arrow[hook, from=2-2, to=1-2]
\end{tikzcd}\]
Automorphisms are furthermore required to be unitarily implemented in the vacuum representation in the sense that they come equipped with a unitary $V_\varphi: H_0\to H_0$ such that, for every $I\in \Jcal$, it holds that
\[
\pi_{0,I}\circ\varphi_I(-) = V_\varphi\circ \pi_{0,I}(-)\circ V_\varphi^{-1},
\]
and $V_\varphi(\Omega) = \Omega$. We denote by $\Aut(\A)$ the group of automorphisms of $\A$. Given $\varphi\in \Aut(\A)$, an interval $I\in \Jcal$ and $x\in\A(I)$, we will write $\varphi x:=\varphi_I x$.

Automorphisms of $\A$ allow us to define twisted representations of the conformal net, as follows. Let us pick $\mathrm{p}:=1\in S^1$, and fix it for the rest of the paper.

\begin{definition}\label{def: TwistedRep}
Let $\varphi\in\Aut(\A)$ be an automorphism of $\A$ and recall that we have fixed $\mathrm{p} = 1\in S^1$. A $\varphi$\emph{-twisted representation} of $\A$ consists of a Hilbert space $H$ equipped with a collection of $*$-actions $\pi^H_I$ of $\A(I)$ on $H$, such that for every pair of intervals $I,J\in\Jcal$ with $J\subset I$, the diagram
\[\begin{tikzcd}
	{\A(I)} & {\A(I)} & {B(H)} \\
	& {\A(J)}
	\arrow["{\varphi_I^{-1}}", from=1-1, to=1-2]
	\arrow["{\pi^H_I}", from=1-2, to=1-3]
	\arrow[hook, from=2-2, to=1-1]
	\arrow["{\pi^H_J}"', from=2-2, to=1-3]
\end{tikzcd}\]
commutes if $\mathrm{p}\in \mathrm{cl}({I})$, $\mathrm{p}\notin{\mathrm{cl}({J})}$, and $J$ is counter-clockwise to $\mathrm{p}$ ($\mathrm{cl}(L)$ denotes the closure of an interval $L$); and 
\[\begin{tikzcd}
	{\A(I)} & {B(H)} \\
	{\A(J)}
	\arrow["{\pi^H_I}", from=1-1, to=1-2]
	\arrow[hook, from=2-1, to=1-1]
	\arrow["{\pi_J^H}"', from=2-1, to=1-2]
\end{tikzcd}\]
commutes otherwise. We write $\Rep^\varphi(\A)$ for the category of $\varphi$-twisted representations of $\A$, where morphisms are bounded linear maps commuting with the actions of all the von Neumann algebras $\A(I)$. We define the $\mathrm{W}^*$-category $\Rep^{\Aut(\A)}(\A) :=\bigoplus\limits_{\varphi\in \Aut(\A)}\Rep^\varphi(\A)$, whose objects are countable orthogonal direct sums of twisted $\A$-representations.
\end{definition}
\begin{remark}
    Note that, although $\Aut(\A)$ may be naturally a topological group, we are treating it as a discrete group, and an object of $\Rep^{\Aut(\A)}(\A)$ decomposes as a direct sum of objects living in countably many of the subcategories $\Rep^\varphi(\A)$. If $\Aut(\A)$ is a finite group, or we are only interested in a finite subgroup of automorphisms, then $\Rep^{\Aut(\A)}(\A)$ provides a complete picture of the problem at hand.
\end{remark}

\begin{remark}
    Given distinct automorphisms $\varphi,\mu\in \Aut(\A)$, and twisted representations $H^\varphi\in\Rep^\varphi(\A)$ and $K^\mu\in\Rep^\mu(\A)$, the only bounded linear map $H^\varphi\to K^\mu$ commuting with the actions of all the algebras $\A(I)$ is the zero map.
\end{remark}

The condition in Definition \ref{def: TwistedRep} can also be formulated as follows. Let $q: \R\to S^1$ be the map $t\mapsto e^{2\pi i t}$. Given an interval $I\in \Jcal$, we shall write $\hat{I}\subset \R_{>0}$ for the lift of $I$ to $\R_{> 0}$ as close to zero as possible so that $\mathrm{cl}(\hat{I})\subset\R_{>0}$. We call an inclusion $J\subset I$ in $S^1$ \emph{standard} if $\hat{J}\subset \hat{I}$, and \emph{special} if it is not standard. For an inclusion $J\subset I$, we shall write $\delta^{\varphi^{-1}}_{J\subset I}: \A(J)\to \A(I)$ for the inclusion $\A(J)\hookrightarrow \A(I)$ if $J\subset I$ is standard, and for $\A(J)\hookrightarrow\A(I)\xrightarrow{\varphi^{-1}_I}\A(I)$ if $J\subset I$ is special. Then, the condition in Definition \ref{def: TwistedRep} is equivalent to the diagram
\[\begin{tikzcd}
	{\A(I)} & {B(H)} \\
	{\A(J)}
	\arrow["{\pi^H_I}", from=1-1, to=1-2]
	\arrow["{\delta_{J\subset I}^{\varphi^{-1}}}", hook, from=2-1, to=1-1]
	\arrow["{\pi^H_J}"', from=2-1, to=1-2]
\end{tikzcd}\]
commuting for all inclusions $J\subset I$ of intervals in $\Jcal$. 

In what follows, it will be useful to have the following characterization of special inclusions. Let $\Jcal_\R := \{\tilde{I}\subset \R\ |\ \tilde{I} \text{ is an open interval such that $q(\tilde{I}) \in \Jcal$}\}$. For an interval $\tilde{I}\in \Jcal_\R$, we denote $I : =q(\tilde{I})\in\Jcal$. Also, let 
\[
\epsilon(\tilde{I}):=\begin{cases}
    |(0, \partial_-\tilde{I})\cap \Z|,
    & \text{if $\partial_{-} \tilde{I}$} > 0\\
    -|[\partial_-\tilde{I}, 0]\cap \Z|, & \text{if $\partial_-\tilde{I}\leq0$}.
\end{cases}
\] 
Given an inclusion $\tilde{J}\subset \tilde{I}$ of intervals in $\Jcal_\R$, the inclusion $J\subset I$ is standard if and only if $\epsilon(\tilde{J}) = \epsilon(\tilde{I})$. In addition, if the inclusion $J\subset  I $ is special, then $\epsilon(\tilde{I}) = \epsilon(\tilde{J}) - 1$, see \cite[Lem. 2.7]{GcrossedbraidedRep}.

Given a conformal net $\A$, we can produce a partially defined net $\A_\R$ on the real line by assigning to every interval $\tilde{I}\in \Jcal_\R$ the von Neumann algebra $\A_\R(\tilde{I}) := \A(I)$. The net $\A_\R$ satisfies the weaker version of locality that $\A_\R(\tilde{I})$ and $\A_\R(\tilde{J})$ commute if $I\cap J = \emptyset$. Given a $\varphi$-twisted representation $(H^\varphi, \pi^H)$ of $\A$, we obtain an action $\pi^H_{\tilde{I}}$ of every von Neumann algebra $\A_\R(\tilde{I})$ on $H^\varphi$ by setting
\[
    \pi^H_{\tilde{I}}(x) := \pi^H_I(\varphi^{\epsilon(\tilde{I})}x)
\]
for every $\tilde{I}\in \Jcal_\R$ and $x\in A_\R(\tilde{I}) = \A(I)$. Note that it holds that, for every inclusion $\tilde{J}\subset \tilde{I}$ in $\Jcal_\R$ and $x\in \A_\R(\tilde{J}) = \A(J)$,
\[
\pi^H_{\tilde{J}}(x) = \pi^H_J(\varphi^{\epsilon(\tilde{J})}x) = \pi^H_I(\delta^{\varphi^{-1}}_{J\subset I}\varphi^{\epsilon(\tilde{J})}x) = \pi^H_I(\varphi^{\epsilon(\tilde{I})}x) = \pi_{\tilde{I}}^H(x).
\]
In addition, for all $x\in \A_\R(\tilde{I}) = \A(I) = \A_\R(\tilde{I} + 1)$, we have
\begin{equation}\label{eq: TwistedRepOnR}
\pi_{\tilde{I}}^H(\varphi x ) = \pi_{\tilde{I} + 1}^H( x)
\end{equation}
coming from the fact that $\epsilon(\tilde{I} + 1) = \epsilon(\tilde{I}) + 1$. Actually, given a Hilbert space $K$ with a $*$-action $\pi_{\tilde{I}}$ of the von Neumann algebra $\A_\R(\tilde{I})$ for every $\tilde{I}\in\Jcal_\R$, compatible with the inclusion of intervals in $\Jcal_\R$ and satisfying \eqref{eq: TwistedRepOnR}, we obtain a $\varphi$-twisted representation $\pi$ of $\A$ on $K$ by setting
\[
\pi_I(x) := \pi_{\hat{I}}(x)
\]
for every $I\in\Jcal$ and $x\in \A(I)$. We will freely move between the notion of a $\varphi$-twisted representation of $\A$ on a Hilbert space $H$ in the sense of Definition \ref{def: TwistedRep} and in the sense of a compatible family of $*$-actions of the von Neumann algebras $\A_\R(\tilde{I})$ on $H$ satisfying \eqref{eq: TwistedRepOnR}.

\subsection{Equivariantization of $G$-crossed balanced categories}\label{Sec: G-X-balanced}

In this section, we discuss equivariantizations of braided and balanced $\mathrm{W}^*$-tensor categories. Let us briefly recall the notions of crossed-braided and crossed-balanced tensor categories. For a more extensive introduction the reader may read \cite[Sec. 2.2]{GcrossedbraidedRep}. Let $\mathcal{C}$ be a $\mathrm{W}^*$-tensor category (which we assume admits all countable orthogonal direct sums). We write $\underline{\text{Aut}}_\otimes(\mathcal{C})$ for the monoidal category of $\mathrm{W}^*$-tensor automorphisms and unitary natural transformations~of~$\mathcal{C}$.

Let $G$ be a discrete group and write $\underline{G}$ for the tensor category whose objects are elements of $G$ and whose morphisms are only identities. A \emph{$G$-crossed braided $\mathrm{W}^*$-tensor category} structure on $\mathcal{C}$ consists of a grading $\mathcal{C} = \bigoplus\limits_{g\in G}\mathcal{C}_g$ such that the tensor product covers the group multiplication, an action $T: \underline{G}\to \underline{\text{Aut}}_\otimes(\mathcal{C})$ of $G$ on $\mathcal{C}$ covering the conjugation action of $G$ on itself, and a family of unitary isomorphisms $\beta_{X, Y}: X\otimes Y\to T_g(Y)\otimes X$ for $g\in G$, $X\in \mathcal{C}_g$ and $Y\in \mathcal{C}$. The family $\beta$ is called the $G$-crossed braiding and is required to satisfy the obvious versions of the hexagon diagrams. Note that the action $T$ comes, by definition, with unitary equivalences $i_g: \1\to T_g(\1)$, for $\1\in\mathcal{C}$ the unit object, as well as unitary natural isomorphisms $s_g:T_g(-)\otimes T_g(-)\to T_g(-\otimes -)$ and $n_{g,h}: T_g\circ T_h\to T_{gh}$, for $g,h\in G$.

 Given $(\mathcal{C}, \otimes^\Ccal, T^\Ccal, \beta^\Ccal)$ and $(\mathcal{D}, \otimes^\Dcal, T^\Dcal, \beta^\Dcal )$ two $G$-crossed braided $\mathrm{W}^*$-tensor categories, a functor between them consists of a $\mathrm{W}^*$-tensor functor $(F,\Psi): (\Ccal,\otimes^\Ccal)\to (\Dcal, \otimes^\Dcal)$ preserving the $G$-grading, together with unitary natural isomorphisms 
\[
\Phi_g(X): T_g^\Dcal(F(X)) \xrightarrow{\cong }   F(T^\Ccal_g(X))
\]
for all $g\in G$ and $X\in \Ccal$. These are required to be compatible with the $G$-action and the $G$-crossed braiding in the obvious way, see \cite[Sec. 2.2]{GcrossedbraidedRep}. We shall call such a triple $(F,\Psi,\Phi)$ a $G$\emph{-crossed braided} $\mathrm{W}^*$\emph{-tensor functor}. A $G$-crossed braided $\mathrm{W}^*$-tensor functor is an equivalence of $G$-crossed braided $\mathrm{W}^*$-tensor categories if its underlying $\mathrm{W}^*$-functor is an equivalence of $\mathrm{W}^*$-categories.

Given a $G$-crossed braided $\mathrm{W}^*$-tensor category $\mathcal{C}$, a $G$\emph{-crossed balance} on $\mathcal{C}$ is a family of natural unitary isomorphisms $\theta_X: X\to T_g(X)$ for every $g\in G$ and $X\in \mathcal{C}_g$ such that $T_h(\theta_X) = \theta_{T_h(X)}$ for every $g,h\in G$ and $X\in \mathcal{C}_g$ and the following diagram commutes
\begin{equation}\label{eq: Xtwist}
\begin{tikzcd}
	{X\otimes Y} && {T_{gh}(X\otimes Y)} \\
	&& {T_{gh}(X)\otimes T_{gh}(Y)} \\
	&& {T_{ghgh^{-1}g^{-1}}T_{ghg^{-1}}(X)\otimes T_{ghg^{-1}}T_g(Y)} \\
	{T_g(Y)\otimes X} && {T_{ghg^{-1}}(X)\otimes T_g(Y)}
	\arrow["{\theta_{X\otimes Y}}", from=1-1, to=1-3]
	\arrow["{\beta_{X,Y}}"', from=1-1, to=4-1]
	\arrow["\cong"', from=2-3, to=1-3]
	\arrow["\cong"', from=3-3, to=2-3]
	\arrow["{\beta_{T_g(Y), X}}"', from=4-1, to=4-3]
	\arrow["{\theta_{T_{ghg^{-1}}(X)}\otimes\theta_{T_g(Y)}}"', from=4-3, to=3-3]
\end{tikzcd}\end{equation}
for all $g,h\in G$, $X\in \mathcal{C}_g$, $Y\in\mathcal{C}_h$.
    A $G$\emph{-crossed balanced $\mathrm{W}^*$-tensor category} is a $G$-crossed braided $\mathrm{W}^*$-tensor category with a $G$-crossed balance.

Given $\Ccal$ and $\Dcal$ two $G$-crossed braided $\mathrm{W}^*$-tensor categories with $G$-crossed balances $\theta^\Ccal$ and $\theta^\Dcal$ respectively, a $G$\emph{-crossed balanced $\mathrm{W}^*$-tensor functor} $(F,\Psi,\Phi):\Ccal\to \Dcal$ is a functor of $G$-crossed braided $\mathrm{W}^*$-tensor categories such that $F(\theta^\mathcal{C}_X) = \Phi_g(X)\circ \theta^\mathcal{D}_{F(X)}$ for all $g\in G$ and $X\in\Ccal_g$.

It is well-known that, given a $G$-crossed braided tensor category $\mathcal{C}$ one can canonically obtain a braided category $\mathcal{C}^G$ by means of an equivariantization procedure \cite{braidedFC}. The equivariantization construction extends to an assignment from $G$-crossed balanced $\mathrm{W}^*$-tensor categories to balanced $\mathrm{W}^*$-tensor categories. We spell it out here for completeness.

The equivariantized category $\mathcal{C}^G$ has objects pairs $(X,u)$ where $X\in \mathcal{C}$ and $u = \{u_g\}_{g\in G}$ is a family of unitary isomorphisms $u_g: T_g(X)\cong X$ such that the diagram
\[\begin{tikzcd}
	{T_gT_h(X)} && {T_g(X)} \\
	{T_{gh}(X)} && X
	\arrow["{T_g(u_h)}", from=1-1, to=1-3]
	\arrow["\cong"', from=1-1, to=2-1]
	\arrow["{u_g}", from=1-3, to=2-3]
	\arrow["{u_{gh}}"', from=2-1, to=2-3]
\end{tikzcd}\]
commutes for every $g,h\in G$. We also require that $u_e$ is induced by the trivialization $\text{Id}_\mathcal{C}\cong T_e$. A morphism between two objects $(X,u)$ and $(Y,v)$ in $\mathcal{C}^G$ is a morphism from $X$ to $Y$ in $\mathcal{C}$ compatible with $u_g$ and $v_g$ for all $g\in G$. Given $(X, u)\in \mathcal{C}^G$, its algebra of endomorphisms is the subalgebra of $\End_\mathcal{C}(X)$ of morphisms $f:X\to X$ such that $u_g\circ T_g(f)\circ u_g^* = f$ for all $g\in G$. Since $T_g$ is a $\mathrm{W}^*$-functor and $u_g$ is unitary for each $g\in G$, the map $f\mapsto u_g\circ T_g(f)\circ u_g^*$ is a normal isometry from $\End_{\mathcal{C}}(X)$ to itself. Hence, $\End_{\mathcal{C}^G}((X, u)) = \bigcap_{g\in G}\ker(u_gT_g(-)u_g^*-\id)$ is a von Neumann algebra. Therefore, $\mathcal{C}^G$ is a $\mathrm{W}^*$-category. The $\mathrm{W}^*$-category $\mathcal{C}^G$ inherits a tensor structure
\[
(X,u)\otimes (Y,v):=(X\otimes Y,u\otimes v)
\]
from that of $\mathcal{C}$, where $u\otimes v = \{(u_g\otimes v_g)\circ s_g^{-1}(X,Y)\}_{g\in G}$. The unit is the object $(\1, \{i^{-1}_g\}_{g\in G})\in\mathcal{C}^G$. The fact that the action of $G$ on $\mathcal{C}$ is by tensor automorphisms implies that the associators and unitors of $\mathcal{C}$ descend to data on $\mathcal{C}^G$, providing its tensor structure. The $G$-crossed braiding $\beta$ on $\mathcal{C}$ induces an honest braiding on $\mathcal{C}^G$ as follows. Let $(X,u), (Y,v)\in \mathcal{C}^G$. Then, we can write $X = \bigoplus\limits_{g\in G} X_g$ with $X_g\in \Ccal_g$, and similarly for $Y = \bigoplus\limits_{h\in G}Y_h$, so that $X\otimes Y = \bigoplus\limits_{(g,h)\in G^2}X_g\otimes Y_h$. Note that for every $g\in G$, the isomorphism $v_g: T_g(Y)\cong Y$ breaks as a direct sum $\bigoplus_{h\in G}v_g^h$ with $v_g^h: T_g(Y_h)\xrightarrow{\cong} Y_{ghg^{-1}}$. On the $(g,h)$-component of $\bigoplus\limits_{(g,h)\in G^2}X_g\otimes Y_h$, we define $\beta^G_{(X,u), (Y,v)}$~as
\[
X_g\otimes Y_h\xrightarrow{\beta_{X_g, Y_h}} T_g(Y_h)\otimes X_g\xrightarrow{v^h_g\otimes \id} Y_{ghg^{-1}}\otimes X_g\hookrightarrow Y\otimes X.
\]
This map is extended linearly to produce $\beta^G_{(X,u), (Y,v)}$. The hexagon axioms follow from the definition of a $G$-crossed braiding. Finally, the $G$-crossed balance $\theta$ on $\mathcal{C}$ provides an honest balance on $\mathcal{C}^G$. Let us write $u_g = \bigoplus\limits_{h\in G} u_g^h$ with $u_g^h: T_g(X_h)\xrightarrow{\cong } X_{ghg^{-1}}$ for all $g,h\in G$. Then, we define the balance $\theta^G$ on $(X,u)$ to be the direct sum of the maps
\[
X_g\xrightarrow{\theta_{X_g}} T_g(X_g)\xrightarrow{u_{g}^g} X_g\hookrightarrow X
\]
for all $g\in G$. Diagram \eqref{eq: Xtwist} implies that ${\theta}^G$ is a balance on the braided $\mathrm{W}^*$-tensor category $(\mathcal{C}^G,{\beta}^G)$ as follows. The balance condition on $(X,u)$ and $(Y, v)$ reads
\begin{equation}\label{eq: BalanceEquivariantization}
\theta^G_{(X,u)\otimes(Y,v)} = (\theta_{(X,u)}^G\otimes \theta^G_{(Y,v)})\circ \beta^G_{(Y,v), (X,u)}\circ \beta^G_{(X,u), (Y, v)}.
\end{equation}
Given $g,h\in G$, the equality \eqref{eq: BalanceEquivariantization} on the component $X_g\otimes Y_h$ of $X\otimes Y$ is equivalent to the commutativity of the outer diagram of

\[\hspace{-1.6cm}\begin{tikzcd}
	{X_g\otimes Y_h} & {T_{gh}(X_g\otimes Y_h)} & {T_{gh}(X_g)\otimes T_{gh}(Y_h)} \\
	{T_g(Y_h)\otimes X_g} & {T_{ghg^{-1}}(X_g)\otimes T_g(Y_h)} \\
	{Y_{ghg^{-1}}\otimes X_g} && {X_{ghgh^{-1}g^{-1}}\otimes Y_{ghg^{-1}}} \\
	{T_{ghg^{-1}}(X_g)\otimes Y_{ghg^{-1}}} \\
	& {T_{ghgh^{-1}g^{-1}}T_{ghg^{-1}}(X_g)\otimes T_{ghg^{-1}}T_g(Y_h)} \\
	{X_{ghgh^{-1}g^{-1}}\otimes Y_{ghg^{-1}}} && {T_{ghgh^{-1}g^{-1}}(X_{ghgh^{-1}g^{-1}})\otimes T_{ghg^{-1}}(Y_{ghg^{-1}})}.
	\arrow["{\theta_{X_g\otimes Y_h}}", from=1-1, to=1-2]
	\arrow["{\beta_{X_g, Y_h}}"', from=1-1, to=2-1]
	\arrow["\cong", from=1-2, to=1-3]
	\arrow["{u_{gh}\otimes v_{gh}}", from=1-3, to=3-3]
	\arrow["{\beta_{T_g(Y_h), X_g}}", from=2-1, to=2-2]
	\arrow["{v_g^h\otimes\id}"', from=2-1, to=3-1]
	\arrow["{\id\otimes v_{g}^h}"', from=2-2, to=4-1]
	\arrow["{\theta_{T_{ghg^{-1}}(X_g)}\otimes\theta_{T_g(Y_h)}}"{pos=0.2}, from=2-2, to=5-2]
	\arrow["{u^{g}_{ghg^{-1}}\otimes v^h_g}"{pos=0.4}, from=2-2, to=6-1]
	\arrow["{\beta_{Y_{ghg^{-1}}, X_g}}"', from=3-1, to=4-1]
	\arrow["{u^g_{ghg^{-1}}\otimes\id}"', from=4-1, to=6-1]
	\arrow["\cong"', from=5-2, to=1-3]
	\arrow["\begin{array}{c} T_{ghgh^{-1}g^{-1}}(u^g_{ghg^{-1}})\otimes\\ T_{ghg^{-1}}(v^h_g) \end{array}", shift right, from=5-2, to=6-3]
	\arrow["{\theta_{X_{ghgh^{-1}g^{-1}}}\otimes \theta_{Y_{ghg^{-1}}}}"', from=6-1, to=6-3]
	\arrow["{u^{ghgh^{-1}g^{-1}}_{ghgh^{-1}g^{-1}}\otimes v_{ghg^{-1}}^{ghg^{-1}}}"', from=6-3, to=3-3]
\end{tikzcd}\]
The top diagram commutes by definition of a $G$-crossed balance. The uppermost left triangle commutes by naturality of the $G$-crossed braiding, and the one immediately below commutes trivially. The bottom diagram commutes by naturality of the $G$-crossed balance. The right-most inner diagram commutes by hypothesis on the structure morphisms $u$ and $v$. Hence, the outer diagram commutes and, since $g,h\in G$ were arbitrary, $\theta^G$ is indeed a balance on the braided tensor category $\mathcal{C}^G$.

Given two $G$-crossed braided $\mathrm{W}^*$-tensor categories $\Ccal$ and $\Dcal$, it is straightforward to check that a $G$-crossed braided $\mathrm{W}^*$-functor from $\Ccal$ to $\Dcal$ defines a braided $\mathrm{W}^*$-tensor functor between the equivariantizations $\Ccal^G$ and $\Dcal^G$. If $\Ccal$ and $\Dcal$ are furthermore $G$-crossed balanced, and the functor is compatible with the crossed balance, then the induced functor $\Ccal^G\to \Dcal^G$ is compatible with the balances.

\subsection{The $G$-crossed balanced category of representations of a conformal net}

Let $G$ be a discrete group acting faithfully on a conformal net $\A$ by an injective group homomorphism $G\to \Aut(\A)$. Throughout, we identify an element $g\in G$ with the automorphism of $\A$ it induces. In this section, we recall from \cite{GcrossedbraidedRep} the construction of the $G$-crossed balanced $\mathrm{W}^*$-tensor structure of
\[
\Rep^G(\A):=\bigoplus\limits_{g\in G}\Rep^g(\A).
\]
Actually, we shall produce an $\Aut(\A)$-crossed balanced $\mathrm{W}^*$-tensor structure on $\Rep^{\Aut(\A)}(\A)$. The $G$-crossed braided $\mathrm{W}^*$-tensor structure on $\Rep^G(\A)$ can be obtained by pulling the latter back along $\Phi: G\to \Aut(\A)$.

\label{Sec: CatOfTwistedReps}
\subsubsection{Connes fusion of twisted representations}

Recall that we write $q:\R\to S^1$, $t\mapsto e^{2\pi i t}$ for the exponential map, and we denote by $\Jcal_\R$ the collection of connected open intervals in $\R$ whose image under $q$ is an interval of $S^1$. Given $\tilde{I}\in \Jcal_\R$, we write $\tilde{I}^{c+} := (\partial_+\tilde{I}, \partial_-\tilde{I} + 1)$ and $\tilde{I}^{c-}:= (\partial_+\tilde{I} - 1, \partial_-\tilde{I})$ for the \emph{positive} and the \emph{negative complements} of $\tilde{I}$ respectively, which are intervals in $\Jcal_\R$. Note that $q(\tilde{I}^{c\pm}) = I^{c}$. Fix an automorphism $\varphi\in \Aut(\A)$ and let $H^\varphi = (H^\varphi, \pi^H)\in\Rep^\varphi(\A)$ be a $\varphi$-twisted representation. We denote by $\Hom_{\A_\R(\tilde{I})}(H_0, H^\varphi)$ the space of bounded linear operators $T:H_0\to H^\varphi$ intertwining the actions of $\A_\R(\tilde{I})$. Since $\A_\R(\tilde{I})$ is a type $\mathrm{III}$ factor, there are unitary operators in $\Hom_{\A_\R(\tilde{I})}(H_0, H^\varphi)$.

We say that a vector $\xi\in H^\varphi$ is $(\tilde{I}, +)$\emph{-bounded} if there exists $T\in \Hom_{\A_\R(\tilde{I}^{c+})}(H_0, H^\varphi)$ such that $\xi = T\Omega$. We write $H_+^\varphi(\tilde{I})$ for the subspace of $H^\varphi$ of $(\tilde{I},+)$-bounded vectors. Given $\xi\in H_+^\varphi(\tilde{I})$, there exists a unique operator $T\in \Hom_{\A_\R(\tilde{I}^{c+})}(H_0, H^\varphi)$ such that $\xi = T\Omega$, as $\A(I^{c})\Omega$ is dense in $H_0$ by the Reeh-Schlieder Theorem. We denote this operator by $T = Z^+(\xi, \tilde{I})$. We define similarly the space of $(\tilde{I},-)$-bounded vectors $H^\varphi_-(\tilde{I})$, using the negative complement $\tilde{I}^{c-}$ of $\tilde{I}$, and $Z^-(\xi, \tilde{I}): H_0\to H^\varphi$ for every $\xi\in H^\varphi_-(\tilde{I})$. 

For an untwisted representation $K\in \Rep(\A)$, there is no distinction between $(\tilde{I},+)$- and $(\tilde{I}, -)$-bounded vectors of $K$, and these furthermore only depend on $I$. Also, given $\xi\in K_{\pm}(\tilde{I})$, we have that $Z^+(\xi, \tilde{I}) = Z^-(\xi,\tilde{I})$, and these also only depend on $I$. We may therefore write $K_{\pm}(\tilde{I})$ as $K(I)$. In particular, for the vacuum representation we have that $H_0(I) = \A(I)\Omega$, by Haag duality. This implies that $H_0(I)\subset H_0$ is dense for all $I\in\Jcal$. Since there exist unitary operators in $\Hom_{\A_\R(\tilde{I}^{c\pm})}(H_0, H^\varphi)$, we obtain that $H_\pm^\varphi(\tilde{I})\subset H^\varphi$ are dense inclusions for all $\tilde{I}\in\Jcal_\R$. In addition, if $\tilde{I_1}\subset \tilde{I}$ is an inclusion of intervals in $\Jcal_\R$, we have inclusions $H_\pm^\varphi(\tilde{I_1})\subset H_\pm^\varphi(\tilde{I})$, which are also dense. 

Let $\mu\in \Aut(\A)$ be another automorphism and $K^\mu = (K^\mu, \pi^K)\in \Rep^\mu(\A)$ be a $\mu$-twisted representation. Fix intervals $\tilde{I}, \tilde{J}\in\Jcal_\R $ satisfying $\tilde{J}\subset\tilde{I}^{c+}$.

\begin{definition}(\cite[Def. 3.2]{GcrossedbraidedRep})
    The \emph{Connes fusion of $H^\varphi$ and $K^\mu$ over $\tilde{I}$ and $\tilde{J}$}, denoted $H^\varphi_+(\tilde{I})\boxtimes K^\mu_-(\tilde{J})$, is the completion of $H^\varphi_+(\tilde{I})\otimes K^\mu_-(\tilde{J})$ with respect to the positive sesquilinear~form
    \begin{equation}\label{eq: SesquilinearConnesForm}
    \langle \xi\otimes\eta, \xi'\otimes\eta'\rangle :=\langle Z^-(\eta', \tilde{J})^*Z^-(\eta, \tilde{J})Z^+(\xi', \tilde{I})^*Z^+(\xi, \tilde{I})\Omega, \Omega\rangle
    \end{equation}
    for $\xi, \xi'\in H^\varphi_+(\tilde{I})$ and $\eta, \eta'\in K^\mu_-(\tilde{J})$.
\end{definition}

Note that $Z^+(\xi', \tilde{I})^*Z^+(\xi, \tilde{I})\in \Hom_{\A(I^c)}(H_0, H_0) \cong \A(I) = \A_\R(\tilde{I})$ and $Z^-(\eta', \tilde{J})^*Z^-(\eta, \tilde{J})\in \Hom_{\A(J^c)}(H_0, H_0) \cong \A(J) = \A_\R(\tilde{J})$, hence the inner product \eqref{eq: SesquilinearConnesForm} can be expressed in the following equivalent ways
\begin{align}\label{eq: DiffConnesFroms}
\langle \xi\otimes\eta|\xi'\otimes\eta'\rangle&=\langle Z^-(\eta', \tilde{J})^*Z^-(\eta, \tilde{J})Z^+(\xi', \tilde{I})^*Z^+(\xi, \tilde{I})\Omega, \Omega\rangle\\\nonumber&=\langle Z^+(\xi', \tilde{I})^*Z^+(\xi, \tilde{I})Z^-(\eta', \tilde{J})^*Z^-(\eta, \tilde{J})\Omega|\Omega\rangle\\\nonumber
 & = \langle \pi^K_{\tilde{I}}(Z^+(\xi', \tilde{I})^*Z^+(\xi, \tilde{I}))\eta|\eta'\rangle\\\nonumber &  =\langle \pi^H_{\tilde{J}}(Z^-(\eta', \tilde{J})^*Z^-(\eta, \tilde{J}))\xi|\xi'\rangle.
\end{align}

Given twisted representations $\hat{H}^\varphi\in \Rep^\varphi(\A)$ and $\hat{K}^\mu\in \Rep^\mu(\A)$, and morphisms $F: H^\varphi\to \hat{H}^\varphi$ and $G: K^\mu\to \hat{K}^\mu$, the map $F\otimes G: H^\varphi_+(\tilde{I})\otimes K^\mu_-(\tilde{J})\to \hat{H}^\varphi_+(\tilde{I})\otimes \hat{K}^\mu_-(\tilde{J})$ induces a bounded linear map $F\boxtimes G: H^\varphi_+(\tilde{I})\boxtimes K^\mu_-(\tilde{J})\to \hat{H}^\varphi_+(\tilde{I})\boxtimes \hat{K}^\mu_-(\tilde{J})$, which we call the \emph{Connes fusion of $F$ and $G$}.

Given $\tilde{I_1},\tilde{J_1}\in\Jcal_\R$ such that $\tilde{J_1}\subset\tilde{I_1}^{c+}$ and $\tilde{I_1}\subset\tilde{I}$, $\tilde{J_1}\subset \tilde{J}$, there is a canonical equivalence
\begin{equation}\label{eq: CanonicalEquivalenceInclusion}
H^\varphi_+(\tilde{I_1})\boxtimes K_-^\mu(\tilde{J_1})\xrightarrow{\cong} H_+^\varphi(\tilde{I})\boxtimes K_-^\mu(\tilde{J})
\end{equation}
induced by the inclusion $H_+^\varphi(\tilde{I_1})\otimes K_-^\mu(\tilde{J_1})\hookrightarrow H_+^\varphi(\tilde{I})\otimes K_-^\mu(\tilde{J})$.

Now, given $z, \zeta\in \R$ such that $\zeta\in (z,z+1)$, we define
\[
H_+^\varphi(z)\boxtimes K_-^\mu(\zeta):=\varinjlim\limits_{(z,\zeta)\in\tilde{I}\times \tilde{J}}H_+^\varphi(\tilde{I})\boxtimes K_-^\mu(\tilde{J}) = \Bigg(\bigsqcup\limits_{(z,\zeta)\in\tilde{I}\times \tilde{J}}H_+^\varphi(\tilde{I})\boxtimes K_-^\mu(\tilde{J})\Bigg)\Big/\cong,
\]
where the limit runs over pairs $\tilde{I}, \tilde{J}\in \Jcal_\R$ containing $z$ and $\zeta$ respectively and satisfying $\tilde{J}\subset\tilde{I}^{c+}$, and $\cong$ denotes the equivalence obtained by identifying $H^\varphi_+(\tilde{I_1})\boxtimes K^\mu_-(\tilde{J_1})$ with $H^\varphi_+(\tilde{I})\boxtimes K^\mu_-(\tilde{J})$ by the unitary \eqref{eq: CanonicalEquivalenceInclusion} whenever $\tilde{I_1}\subset \tilde{I}$ and $\tilde{J_1}\subset \tilde{J}$. Given intervals $\tilde{I}, \tilde{J}\in \Jcal_\R$ containing $z$ and $\zeta$ respectively and such that $\tilde{J}\subset\tilde{I}^{c+}$, the canonical map 
\[
H_+^\varphi(\tilde{I})\boxtimes K_-^\mu(\tilde{J})\to H_+^\varphi(z)\boxtimes K_-^\mu(\zeta)
\]
is a unitary equivalence. 

We next relate the Connes fusion over different pairs of intervals. Consider the topological space $(\R\times\R)\setminus (\R\times_{S^1}\R) = \{(z,\zeta)\in \R\times\R\ |\  qz\neq q\zeta\}$, which is a disconnected space with contractible connected components. We write $\ConfR$ for the connected component of $(\R\times\R)\setminus (\R\times_{S^1}\R)$ given by the points $(z,\zeta)$ with $\zeta\in (z,z+1).$ Then, $\ConfR$ is the universal cover of $\text{Conf}_2(S^1)$, the space of configurations of two points in $S^1$. Let $\gamma = (\alpha, \beta):[0,1]\to \ConfR$ be a path in $\ConfR$ from $\gamma(0) =: (z_0, \zeta_0)$ to $\gamma(1) =: (z_1, \zeta_1)$. Assume that $\gamma$ is small enough so that $\gamma([0,1])\subset \tilde{I}\times \tilde{J}$, for some $\tilde{I}, \tilde{J}\in \Jcal_\R$ satisfying $\tilde{J}\subset \tilde{I}^{c+}$. We define the unitary equivalence $\gamma^\bullet: H_+^\varphi(z_0)\boxtimes K_-^\mu(\zeta_0)\xrightarrow{\cong}H_+^\varphi(z_1)\boxtimes K_-^\mu(\zeta_1)$ as the composition
\[
H_+^\varphi(z_0)\boxtimes K_-^\mu(\zeta_0)\xrightarrow{\cong}H_+^\varphi(\tilde{I})\boxtimes K_-^\mu(\tilde{J})\xrightarrow{\cong}H_+^\varphi(z_1)\boxtimes K_-^\mu(\zeta_1).
\]
For a general path $\gamma$, we choose $0 = t_0<t_1<\ldots <t_k = 1$ such that each $\gamma_n:=\gamma|_{[t_n, t_{n+1}]}$ is small enough in the sense above, and define $\gamma^\bullet: H_+^\varphi(z_0)\boxtimes K_-^\mu(\zeta_0)\xrightarrow{\cong}H_+^\varphi(z_1)\boxtimes K_-^\mu(\zeta_1)$ as 
\[
\gamma^\bullet = \gamma_{k-1}^\bullet\ldots\gamma_1^\bullet\gamma_0^\bullet.
\]
Since finer partitions give the same result, the map $\gamma^\bullet$ is independent of the choice of partition. We call $\gamma^\bullet$ the \emph{path-continuation} induced by $\gamma$.

Let $\tilde{I_0}, \tilde{I_1}, \tilde{J_0}, \tilde{J_1}\in \Jcal_\R$ be intervals such that $\tilde{J_0}\subset\tilde{I_0}^{c+}$ and $\tilde{J_1}\subset\tilde{I_1}^{c+}$. Let $\gamma$ be a path in $\ConfR$ from $\gamma(0) =: (z_0, \zeta_0)\in \tilde{I_0}\times\tilde{J_0}$ to $\gamma(1) =: (z_1, \zeta_1)\in \tilde{I_1}\times \tilde{J_1}$. We define the path continuation $\gamma^\bullet : H_+^\varphi(\tilde{I_0})\boxtimes K_-^\mu(\tilde{J_0})\xrightarrow{\cong} H_+^\varphi(\tilde{I_1})\boxtimes K_-^\mu(\tilde{J_1})$ as the composition
\[
H_+^\varphi(\tilde{I_0})\boxtimes K_-^\mu(\tilde{J_0})\xrightarrow{\cong}H_+^\varphi(z_0)\boxtimes K_-^\mu(\zeta_0)\xrightarrow{\gamma^\bullet} H_+^\varphi(z_1)\boxtimes K_-^\mu(\zeta_1)\xrightarrow{\cong}H_+^\varphi(\tilde{I_1})\boxtimes K_-^\mu(\tilde{J_1}).
\]
Homotopic paths produce the same path continuation.

We next introduce Connes fusion over a single interval. By the equalities in Equation \eqref{eq: DiffConnesFroms}, the Connes fusion $H^\varphi_+(\tilde{I})\boxtimes K^\mu_-(\tilde{J})$ can also be obtained as a completion of the vector space $H^\varphi_+(\tilde{I})\otimes K^\mu$ with respect to
\[
\langle \xi\otimes\eta\,|\,\xi'\otimes \eta'\rangle := \langle\pi^K_{\tilde{I}}\big(Z^+(\xi', \tilde{I})^*Z^+(\xi, \tilde{I})\big)(\eta)\,|\, \eta'\rangle
\]
for $\xi, \xi'\in H_+^\varphi(\tilde{I})$ and $\eta, \eta'\in K^\mu$. We call this completion the \emph{Connes fusion of $H^\varphi$ and $K^\mu$ over $\tilde{I}$ on the left} and denote it $H^\varphi_+(\tilde{I})\boxtimes K^\mu$. Then, if $\tilde{J}\in \Jcal_\R$ is an interval such that $\tilde{J}\subset \tilde{I}^{c+}$, the inclusion $H_+^\varphi(\tilde{I})\otimes K_-^\mu(\tilde{J})\hookrightarrow H^\varphi_+(\tilde{I})\otimes K^\mu$ induces a canonical equivalence $H_+^\varphi(\tilde{I})\boxtimes K_-^\mu(\tilde{J})\xrightarrow{\cong} H_+^\varphi(\tilde{I})\boxtimes K^\mu$. For any $z\in \R$, we define $H^\varphi_+(z)\boxtimes K^\mu := \varinjlim\limits_{z\in\tilde{I}}\, H^\varphi_+(\tilde{I})\boxtimes K^\mu$. Given a path $\alpha$ in $\R$ from $z_0$ to $z_1$, we also obtain a path continuation $\alpha^\bullet: H^\varphi_+(z_0)\boxtimes K^\mu\xrightarrow{\cong} H^\varphi_+(z_1)\boxtimes K^\mu$ as follows. If $\alpha$ is small enough so that $\alpha([0,1])$ can be covered by an interval $\tilde{I}\in \Jcal_\R$, we define
\[
\alpha^\bullet: H^\varphi_+(z_0)\boxtimes K^\mu\xrightarrow{\cong} H^\varphi_+(\tilde{I})\boxtimes K^\mu\xrightarrow{\cong} H^\varphi_+(z_1)\boxtimes K^\mu.
\]
For a generic path $\alpha$, we pick a partition $0 = t_0<t_1<\ldots<t_k = 1$ of $[0,1]$ so that each $\alpha_n:=\alpha|_{[t_n, t_{n+1}]}$ is small enough in the sense above, and define $\alpha^\bullet = \alpha_{k-1}^\bullet\ldots\alpha_0^\bullet$. Given two intervals $\tilde{I_0},\tilde{I_1}\in \Jcal_{\R}$ and a path $\alpha$ in $\R$ from a point in $\tilde{I_0}$ to a point in $\tilde{I_1}$, we define
\[
\alpha^\bullet:H^\varphi_+(\tilde{I_0})\boxtimes K^\mu\xrightarrow{\cong } H_+^\varphi(\alpha(0))\boxtimes K^\mu\xrightarrow{\alpha^\bullet} H^\varphi_+(\alpha(1))\boxtimes K^\mu\xrightarrow{\cong} H^\varphi_+(\tilde{I_1})\boxtimes K^\mu.
\]
Homotopic paths induce the same unitary equivalence. Similar properties hold for  $H^\varphi\boxtimes K^\mu_-(\tilde{J})$, the \emph{Connes fusion of $H^\varphi$ and $K^\mu$ over $\tilde{J}$ on the right}, defined in the obvious analogous~way.

\subsubsection{Endowing the Connes fusion with a twisted action}

From now on, we fix intervals $\tilde{I},\tilde{J}\in \Jcal_\R$ such that $\tilde{J}\subset \tilde{I}^{c+}$, as well as automorphisms $\varphi,\mu\in\Aut(\A)$ and twisted representations $H^\varphi = (H^\varphi, \pi^H)\in \Rep^\varphi(\A)$ and $K^\mu = (K^\mu, \pi^K)\in \Rep^\mu(\A)$. We recall how to endow the Hilbert space $H^\varphi(\tilde{I})\boxtimes K^\mu(\tilde{J})$ with a $(\varphi\circ\mu)$-twisted action of $\A$. We will do this by endowing $H^\varphi(\tilde{I})\boxtimes K^\mu(\tilde{J})$ with compatible actions $\pi^{H\boxtimes K}_{\tilde{L}}$ of the $*$-algebras $\A_\R(\tilde{L})$ such that $\pi^{H\boxtimes K}_{\tilde{L}} \circ (\varphi\circ\mu) = \pi^{H\boxtimes K}_{\tilde{L}+1}$ for all $\tilde{L}\in \Jcal_\R$. Note that we have natural actions of $\A_\R(\tilde{I} )$ and $\A_\R(\tilde{J})$ on $H^\varphi(\tilde{I})\boxtimes K^\mu(\tilde{J})$. Indeed, given $x\in \A(I) = \A_\R(\tilde{I})$ and $y\in \A(J) = \A_\R(\tilde{J})$, we define
\[
\pi^{H \boxtimes K}_{\tilde{I}}(x)(\xi\otimes \eta) = \Big(\pi_{\tilde{I}}^H(x)(\xi)\Big)\otimes \eta\hspace{3cm}
\pi^{H\boxtimes K}_{\tilde{J} }(y)(\xi\otimes\eta) = \xi\otimes \Big(\pi_{\tilde{J}}^K(y)(\eta)\Big)
\]
for $\xi\otimes\eta \in H_+^\varphi(\tilde{I})\otimes K_-^\mu(\tilde{J})\subset H_+^\varphi(\tilde{I})\boxtimes K_-^\mu(\tilde{J}) $. This also provides canonical actions of all $\A_\R(\tilde{L})$ for $\tilde{L}\subset \tilde{I}$ or $\tilde{L}\subset\tilde{J}$.

\begin{theorem}(\cite[Thm. 3.7]{GcrossedbraidedRep})\label{thm: RepresentationOnFusion}
    Let $\varphi,\mu\in \Aut(\A)$ be automorphisms and $H^\varphi$ and $K^\mu$ be $\varphi$- and $\mu$-twisted $\A$-representations respectively. Fix $\tilde{I},\tilde{J}\in\Jcal_\R$ with $\tilde{J}\subset\tilde{I}^{c+}$. Then,
    \begin{enumerate}
        \item there is a unique compatible collection $\pi^l$ of $*$-actions of the algebras $\A_\R(\tilde{L})$ on $H^\varphi_+(\tilde{I})\boxtimes K^\mu_-(\tilde{J})$ such that the following condition holds: for every $\tilde{L},\tilde{M}\in \Jcal_\R$ with $\tilde{M}\subset \tilde{L}^{c+}$, any path $\gamma$ in $\ConfR$ from $\tilde{I}\times\tilde{J}$ to $\tilde{L}\times\tilde{M}$, it holds that
        \begin{equation}\label{eq: leftRep}
        \pi_{\tilde{L}}^l(x) = (\gamma^\bullet)^{-1}\circ \pi^{H\boxtimes K}_{\tilde{L}}(x)\circ\gamma^\bullet
        \end{equation}
        for all $x\in \A_\R(\tilde{L})$;
        \item  there is a unique compatible collection $\pi^r$ of $*$-actions of the algebras $\A_\R(\tilde{L})$ on $H^\varphi_+(\tilde{I})\boxtimes K^\mu_-(\tilde{J})$ such that the following condition holds: for every $\tilde{L},\tilde{M}\in \Jcal_\R$ with $\tilde{L}\subset \tilde{M}^{c+}$, any path $\vartheta$ in $\ConfR$ from $\tilde{I}\times\tilde{J}$ to $\tilde{M}\times\tilde{L}$, it holds that
        \begin{equation}\label{eq: rightRep}
        \pi^r_{\tilde{L}}(x) = (\vartheta^\bullet)^{-1}\circ  \pi^{H\boxtimes K}_{\tilde{L}}(x)\circ \vartheta^\bullet
        \end{equation}
        for all $x\in \A_\R(\tilde{L})$;
        \item the actions $\pi^l$ and $\pi^r$ agree, and hence there is a natural compatible collection $\pi^{H\boxtimes K}$ of actions of the algebras $\A_\R(\tilde{L})$ on $H^\varphi_+(\tilde{I})\boxtimes K^\mu_-(\tilde{J})$;
    \item the unitary maps induced by inclusions of intervals, restrictions of intervals, and path continuations between Connes fusions intertwine the actions $\pi^{H\boxtimes K}$;
    \item the actions $\pi^{H\boxtimes K}$ satisfy that, for all $\tilde{L}\in \Jcal_\R$ and $x\in \A(L)$, 
    \[
    \pi^{H\boxtimes K}_{\tilde{L}}(\varphi\mu x) = \pi^{H\boxtimes K}_{\tilde{L}+1}(x),
    \]
    hence providing a $\varphi\mu$-twisted representation $\pi^{H\boxtimes K}$ of $\A$ on $H^\varphi_+(\tilde{I})\boxtimes K_-^\mu(\tilde{J})$.
    \end{enumerate}
\end{theorem}

The Connes fusion of two morphisms of twisted $\A$-representations is again a morphism of twisted $\A$-representations.

We can define similarly twisted actions of $\A$ on fusions over single intervals. Fix $\tilde{I}\in \Jcal_\R$. If $\tilde{L}\in \Jcal_\R$ is a subinterval $\tilde{L}\subset \tilde{I}$, we define the action of $x\in\A_\R(\tilde{L})$ on $H_+^\varphi(\tilde{I})\boxtimes K^\mu$ by $\pi_{\tilde{L}}^{H\boxtimes K}(\xi\otimes \eta) = \pi_{\tilde{L}}^H(x)(\xi)\otimes \eta$ for all $\xi\in H_+^\varphi(\tilde{I})$ and $\eta\in K^\mu$. For a generic interval $\tilde{L}\in\Jcal_\R$, we choose a path $\alpha$ in $\R$ from $\tilde{I}$ to $\tilde{L}$, and define $\pi_{\tilde{L}}^{ H\boxtimes K}(x) = (\alpha^\bullet)^{-1}\circ \pi^{H\boxtimes K}_{\tilde{L}}(x)\circ \alpha^\bullet$. This is independent of the path chosen. Alternatively, one can define the action of $\A_\R$ on $H^\varphi_+(\tilde{I})\boxtimes K^\mu$ via the canonical equivalence $H^\varphi_+(\tilde{I})\boxtimes K^\mu_-(\tilde{J})\cong H^\varphi_+(\tilde{I})\boxtimes K^\mu$ for some $\tilde{J}\in \Jcal_\R$ with $\tilde{J}\subset\tilde{I}^{c+}$ and the action of $\A_\R$ on $H^\varphi_+(\tilde{I})\boxtimes K^\mu_-(\tilde{J})$ from Theorem \ref{thm: RepresentationOnFusion}. These two actions agree.

\subsubsection{The crossed balanced $\mathrm{W}^*$-structure}

In this section, we recall how to equip the category $$\Rep^{\Aut(\A)}(\A):= \bigoplus\limits_{\varphi\in \Aut(\A)}\Rep^\varphi(\A)$$ with the structure of an $\Aut(\A)$-crossed balanced $\mathrm{W}^*$-tensor category.

Let $\widetilde{S_-^1}:=(-\nicefrac{1}{2}, 0)$ and $\widetilde{S_+^1}:=(0,\nicefrac{1}{2})$. Given automorphisms $\varphi,\mu\in \Aut(A)$, for any twisted representations $H^\varphi\in \Rep^\varphi(\A)$ and $K^\mu\in \Rep^\mu(\A)$, we define their tensor product as $$H^\varphi\boxtimes K^\mu:=H^\varphi_+(\widetilde{S^1_-})\boxtimes K^\mu_-(\widetilde{S^1_+}).$$ We also identify $H^\varphi\boxtimes K^\mu$ with $H_+^\varphi(\widetilde{S^1_-})\boxtimes K^\mu$ and $H^\varphi\boxtimes K_-^\mu(\widetilde{S^1_+})$ using the canonical equivalences from the previous section. We extend this tensor product linearly to obtain a functor
\[
\boxtimes: \Rep^{\Aut(\A)}(\A)\times \Rep^{\Aut(\A)}(\A)\to \Rep^{\Aut(\A)}(\A).
\]
Let us denote $\widetilde{S^1_-}$ and $\widetilde{S^1_+}$ by $-$ and $+$ in the Connes fusions. Given three automorphisms $\varphi,\mu,\nu\in \Aut(\A)$ and representations $H^\varphi,K^\mu,R^\nu$ of $\A$ twisted by $\varphi,\mu$ and $\nu$ respectively, the map $(H^\varphi\otimes K^\mu)\otimes R^\nu\xrightarrow{\cong}H^\varphi\otimes (K^\mu\otimes R^\nu)$ given by $(\xi\otimes\eta)\otimes \chi\mapsto \xi\otimes(\eta\otimes\chi)$ induces a unitary equivalence of twisted $\A$-representations $(H^\varphi\boxtimes K^\mu)\boxtimes R^\nu\xrightarrow{\cong}H^\varphi\boxtimes (K^\mu\boxtimes R^\nu)$ which produces an associator for $\Rep^{\Aut(\A)}(\A)$. The unit object is $H_0$, for the unitors see \cite[Sec. 3.6]{GcrossedbraidedRep}.

The action of $\Aut(\A)$ on $\Rep^{\Aut(\A)}(\A)$ is given as follows. Given $\varphi,\nu\in \Aut(\A)$ and a twisted representation $H^\varphi\in\Rep^\varphi(\A)$, we define $T_\nu(H^\varphi)$ to be the $(\nu\varphi\nu^{-1})$-twisted representation of $\A$ on the underlying Hilbert space $H$ of $H^\varphi$ given by
\[
\pi^{\nu\ast H}_{\tilde{I}}(x)(\xi) = \pi_{\tilde{I}}^H(\nu^{-1}x)(\xi)
\]
for all $\tilde{I}\in\Jcal_\R$, $x\in\A_\R(\tilde{I})$ and $\xi\in H$. On morphisms, the action of $\nu$ is defined to be trivial. Let us write $\Gamma_\nu: H^\varphi\to T_\nu(H^\varphi)$ for the identity map of Hilbert spaces, which is not a morphism of twisted $\A$-representations. It is easy to see that $\Gamma_\nu$ induces equivalences $H^\varphi_-(\tilde{I}) \cong T_\nu(H^\varphi)_-(\tilde{I})$ and $H^\varphi_+(\tilde{I}) \cong T_\nu(H^\varphi)_+(\tilde{I})$. In addition, given an automorphism $\mu\in \Aut(\A)$, a twisted representation $K^\mu\in\Rep^\mu(\A)$ and an interval $\tilde{J}\in\Jcal_\R$ with $\tilde{J}\subset\tilde{I}^{c+}$, the map $H^\varphi_+(\tilde{I})\otimes K^\mu_-(\tilde{J})\to T_\nu(H^\varphi)_+(\tilde{I})\otimes T_\nu(K^\mu)_-(\tilde{J})$ given by $\xi\otimes\eta\to \Gamma_\nu\xi\otimes\Gamma_\nu\eta$ induces a unitary equivalence of twisted $\A$-representations 
\begin{equation}\label{eq: MonoidalT}
T_\nu(H^\varphi_+(\tilde{I})\boxtimes K^\mu_-(\tilde{J}))\cong T_\nu(H^\varphi)_+(\tilde{I})\boxtimes T_\nu(K^\mu)_-(\tilde{J}).
\end{equation}

The functors $T_\nu$ for $\nu\in\Aut(\A)$ provide an action $$T: \underline{\Aut(\A)}\to \underline{\Aut}_\otimes(\Rep^{\Aut(\A)}(\A)).$$ Given $\nu\in \Aut(\A)$, the monoidal structure of $T_\nu$ is given by
\[
T_\nu(H^\varphi)\boxtimes T_\nu(K^\mu) =  T_\nu\big(H^\varphi)_+(-)\boxtimes T_\nu\big(K^\mu\big)_-(+)\cong     T_\nu\big(H^\varphi_+(-)\boxtimes K^\mu_-(+)\big) =   T_\nu(H^\varphi\boxtimes K^\mu),
\]
using the unitary equivalence from Equation \eqref{eq: MonoidalT}. The monoidal structure of the functor $T$ is given, on $H^\varphi\in\Rep^\varphi(\A)$, by the identity map 
\(
T_\mu\circ T_\nu(H^\varphi) = T_{\mu\nu}(H^\varphi).
\)

We define next the $\Aut(\A)$-crossed braiding on $\Rep^{\Aut(\A)}(\A)$. Let $\varrho:[0,1]\to \R$ be the path $\varrho(t) = -\nicefrac{1}{4} + \nicefrac{t}{2}$, which goes from $\widetilde{S_-^1}$ to $\widetilde{S_+^1}$. Given twisted representations $H^\varphi\in\Rep^\varphi(\A)$ and $K^\mu\in \Rep^\mu(\A)$, we define the crossed-braid operator $\B_{H,K}: H^\varphi\boxtimes K^\mu\xrightarrow{\cong} T_\varphi(K^\mu)\boxtimes H^\varphi$~by
\[
\B_{H,
K}: H^\varphi\boxtimes K^\mu = H^\varphi_+(-)\boxtimes K^\mu\xrightarrow{\varrho^\bullet} H^\varphi_+(+)\boxtimes K^\mu\cong T_\varphi(K^\mu)\boxtimes H^\varphi_-(+) = T_\varphi(K^\mu)\boxtimes H^\varphi,
\]
where the equivalence in the middle is induced by the map $H^\varphi_+(+)\otimes K^\mu\cong T_\varphi(K^\mu)\otimes H^\varphi_-(+)$ given by $\xi\otimes\eta\mapsto \Gamma_\varphi\eta\otimes \xi$, see \cite[Prop. 3.13]{GcrossedbraidedRep}. 

Let us finally describe the $\Aut(\A)$-crossed balance on $\Rep^{\Aut(\A)}(\A)$. By \cite[Sec. 3.5]{GcrossedbraidedRep}, any twisted representation $H^\varphi\in\Rep^\varphi(\A)$ has an honest strongly continuous action of the universal cover $\widetilde{\Mob}$ of the Möbius group $\Mob$, constructed from the projective unitary action of $\Diff^+(S^1)$ on $H_0$. The group $\widetilde{\Mob}$ sits inside the universal cover $\widetilde{\Diff^+}(S^1)$ of $\Diff^+(S^1)$, which is an infinite dimensional Lie group whose Lie algebra is the Lie algebra $\text{Vec}(S^1)$ of vector fields on $S^1$ with the negative of the Lie bracket of vector fields. Let $\text{Vec}_\mathbb{C}(S^1)$ denote the complexification of $\text{Vec}(S^1)$ and write $\mathscr{W}\subset \text{Vec}_\mathbb{C}(S^1)$ for the Witt algebra, which is generated by the complex vector fields $L_n(e^{i\theta}):=-e^{in\theta}\frac{d}{d\theta}$ for all $n\in \Z$. The group $\widetilde{\Mob}$ is generated by the exponential of vectors of the form $i(\overline{a_1}L_{-1}+a_0L_0+a_1L_1)\in\text{Vec}(S^1)$ for $a_0\in\R$ and $a_1\in \C$. We write $\exp(iX)\in \widetilde{\Mob}$ for the exponential of $X = \overline{a_1}L_{-1}+a_0L_0+a_1L_1\in \text{Vec}_{\mathbb{C}}(S^1)$, and denote the action of $\exp(iX)$ on $H^\varphi$ by $e^{iX}$. Then, the balance on $H^\varphi$ is defined to be
\[
\theta_H:=e^{-2\pi i L_0}: H^\varphi\to T_\varphi(H^\varphi).
\]
The fact that $\theta_H$ is a well-defined unitary isomorphism of twisted $\A$-representations is the content of \cite[Prop. 3.18]{GcrossedbraidedRep}.

\begin{theorem}(\cite[Thm. 3.30]{GcrossedbraidedRep})\label{Thm: RepAutAIsCrossedBalanced}
    Let $\A$ be a conformal net. The $\mathrm{W}^*$-category $\Rep^{\Aut(\A)}(\A)$ of twisted representations of $\A$, equipped with the Connes fusion of twisted representations, the action $T$, and the families $\B$ and $\theta$ becomes an $\Aut(\A)$-crossed balanced $\mathrm{W}^*$-tensor category.
\end{theorem}

If we have a discrete group $G$ acting on $\A$ by a group homomorphism $\Phi:G\to \Aut(\A)$, we can pull back the $\Aut(\A)-$crossed balanced structure on $\Rep^{\Aut(\A)}(\A)$ along $\Phi$ to obtain a $G$-crossed balanced structure on $\Rep^G(\A)$.

\begin{theorem}\label{Thm: BigCorollary}
    Let $G$ be a discrete group acting faithfully on a conformal net $\A$ by a group homomorphism $\Phi: G\to \Aut(\A)$. Denoting $\Phi(g)$ by $g$, the $G$-crossed braided $\mathrm{W}^*$-tensor category
    \[
    \Rep^G(\A):=\bigoplus_{g\in G}\Rep^g(\A)
    \]
    admits a structure of a $G$-crossed balanced $\mathrm{W}^*$-tensor category.
\end{theorem}

The $G$-crossed braided $\mathrm{W}^*$-tensor category $\Rep^G(\A)$ can be equivalently obtained via localized endomorphisms, as introduced in \cite{muger05}. We recall Müger's construction here and its relationship to $\Rep^G(\A)$ as it will be used in Section \ref{sec: FixedPoints}. Recall that we have fixed the point $\mathrm{p} = 1\in S^1$. We write $\mathcal{J}_\mathrm{p} : =\{I\in\Jcal\,|\,\mathrm{p}\notin\text{cl}(I)\}$ for the poset of intervals in $S^1$ not containing $\mathrm{p}$ in their closure, and $\A_\infty = \bigcup\limits_{I\in\Jcal_\mathrm{p}}\A(I)\subset B(H_0)$ for the sub $*$-algebra of $B(H_0)$ which is the union of all the algebras $\A(I)$ with $I\in \Jcal_\mathrm{p}$. We define $\End\,\A_\infty$ as the $\mathbb{C}$-linear category whose objects are endomorphisms of the $*$-algebra $\A_\infty$ and such that $\Hom_{\End\,\A_\infty}(\rho, \sigma) =\{s\in \A_\infty\,|\, s\rho(x) = \sigma(x)s\,\text{ for all $x\in \A_\infty$}\}$ for $\rho, \sigma\in \End\,\A_\infty$. Composition is given by multiplication in $\A_\infty$. Furthermore, $\End\,\A_\infty$ is a strict tensor category when equipped with the tensor product $\rho\otimes\sigma = 
\rho\circ\sigma$ and $s\otimes t:=s\rho(t) = \rho'(t)s$, where $\rho,\rho', \sigma, \sigma'\in \End\,\A_\infty$ and $s\in\Hom_{\End\,\A_\infty}(\rho, \rho')$ and $t\in  \Hom_{\End\,\A_\infty}(\sigma, \sigma')$. 

Given an interval $I\in \Jcal_\mathrm{p}$ and an element $g\in G$, we say that an endomorphism $\rho\in\End\,\A_\infty$ is \emph{$g$-localized in} $I$ if $\rho(x) = x$ for all $x\in \A(J)$ with $J\in \Jcal_\mathrm{p}$ an interval clockwise to $I$ and $\rho(x) = gx$ for all $x\in \A(J)$ with $J\in \Jcal_\mathrm{p}$ an interval counter-clockwise to $I$. By diffeomorphism covariance of $\A$, any such endomorphism $\rho$ is \emph{transportable}, meaning that, for every interval $J\in\Jcal_\mathrm{p}$, there is an endomorphism $\rho'$, unitarily isomorphic to $\rho$, which is $g$-localized in $J$.

Fix an interval $I\in\Jcal_\mathrm{p}$. The subcategory of $\End\,\A_\infty$ whose objects are the $g$-localized-in-$I$ endomorphisms of $\A$ and whose morphisms are morphisms in $\End\,\A_\infty$ contained in $\A(I)$ is denoted $g-\Loc_I(\A)$. We define the $\mathrm{W}^*$-category $G-\Loc_I(\A): = \bigoplus\limits_{g\in G}g-\Loc_I(\A)$, whose objects are countable direct sums of $g$-localized-in-$I$ endomorphisms of $\A$, for possibly different elements $g\in G$, and whose morphisms are morphisms in $\End\,\A_\infty$ contained in $\A(I)$. In \cite{muger05} it is shown that $G-\Loc_I(\A)$ admits the structure of a $G$-crossed braided tensor category, and in \cite{GcrossedbraidedRep} we showed that the $G$-crossed braided tensor category $G-\Loc_I(\A)$ is equivalent to $\Rep^G(\A)$. We recall both of these results~here.

The category $G-\Loc_I(\A)$ is a tensor subcategory of $\End\,\A_\infty$, since the tensor product induces morphisms
\[
g-\Loc_I(\A)\times h-\Loc_I(\A)\to gh-\Loc_I(\A)
\]
for all $g,h\in G$. The tensor product turns $G-\Loc_I(\A)$ into a $\mathrm{W}^*$-tensor category, as $\End_{\End\,\A_\infty}(\rho)$ is a von Neumann algebra for any $\rho\in G-\Loc_I(\A)$ and the tensor product is a $\mathrm{W}^*$-bifunctor. In addition, the group $G$ acts on $G-\Loc_I(\A)$ by sending $\rho\in G-\Loc_I(\A)$ to $\gamma_g(\rho):=g\rho g^{-1}$ for all $g\in G$. On morphisms, the element $g$ acts on $s\in\Hom_{\End\,\A_\infty}(\rho, \rho')$ by $s\mapsto \gamma_g(s):=g(s)$. It is easy to see that $g\to \gamma_g$ can be upgraded to an action of $G$ on $G-\Loc_I(\A)$ by tensor automorphisms which covers the conjugation action of $G$ on itself. Finally, the $G$-crossed braiding is defined as follows. Let $\rho,\sigma\in G-\Loc_I(\A)$ be endomorphisms with $\rho$ $g$-localized in $I_1\subset I$ and $\sigma$ $G$-localized in $I_2\subset I$, where $(\partial_-I_1, \partial_+I_2)$ can be covered by $I$. Here, $\partial_-$ and $\partial_+$ are defined using the counter-clockwise orientation of $S^1$. Then, it holds that $\rho\otimes \sigma = \gamma_g(\sigma)\otimes \rho$. If $\rho$ is $g$-localized in $I$ and $\sigma$ is $G$-localized in $I$, we may find, by transportability, $\rho'$ and $\sigma'$ localized as above, and unitary isomorphisms $u: \rho\cong \rho'$  and $v: \sigma\cong \sigma'$, and define the crossed-braiding $c_{\rho, \sigma}$ to be given by
\[
c_{\rho, \sigma}:\rho\otimes \sigma\xrightarrow{u\otimes v}\rho'\otimes \sigma' = \gamma_g(\sigma')\otimes \rho'\xrightarrow{\gamma_g(v^{-1})\otimes u^{-1}} \gamma_g(\sigma)\otimes \rho.
\]
This morphism is independent of the particular choice of $\rho', \sigma', u$ and $v$, and provides a $G$-crossed braiding on $G-\Loc_I(\A)$. The following is the statement of \cite[Thm. 2.21]{muger05}.

\begin{theorem}
    Fix an interval $I\in \Jcal_\mathrm{p}$. The $\mathrm{W}^*$-tensor category $G-\Loc_I(\A)$, with the $G$-action $\gamma$ and the crossed-braiding $c$, is a $G$-crossed braided $\mathrm{W}^*$-tensor category.
\end{theorem}

The $G$-crossed braided $\mathrm{W}^*$-tensor category $G-\Loc_I(\A)$ is canonically equivalent to $\Rep^G(\A)$, as it was discussed in \cite[Sec. 4]{GcrossedbraidedRep}. The underlying functor $\mathfrak{E}: G-\Loc_I(\A)\to \Rep^G(\A)$ of this equivalence is defined as follows. Given $g\in G$ and $\rho\in g-\Loc_I(\A)$, we define $\mathfrak{E}(\rho)$ to be the $g$-twisted representation of $\A$ on $H_0$ given by
\[
\pi_J(x) = \pi_0(\rho(x))
\]
for all $J\in\Jcal_\mathrm{p}$ and $x\in \A(J)$. Since $\rho$ is $g$-localized in $I$, the formula above extends to a unique $g$-twisted representation of $\A$ on $H_0.$ We denote this representation by $(H_0, \pi_0\circ\rho)$. On morphisms, it holds by \cite[Lem. 2.13]{muger05} that $\Hom_{g-\Loc_I(\A)}(\rho, \sigma)\subset \A(I)$ for all $\rho, \sigma\in g-\Loc_I(\A)$ and hence we may define $\mathfrak{E}(s) = \pi_0(s)\in B(H_0)$ for all $s\in \Hom_{g-\Loc_I(\A)}(\rho, \sigma)$. Then, $\mathfrak{E}$ can be extended linearly to provide a $\mathrm{W}^*$-functor
\[
\mathfrak{E}: G-\Loc_I(\A)\to \Rep^G(\A).
\]
The following result is \cite[Thm. 4.9]{GcrossedbraidedRep}.
\begin{theorem}\label{Thm: RepIsMugerForG}
    The $\mathrm{W}^*$-functor $\mathfrak{E}: G-\Loc_I(\A)\to \Rep^G(\A)$ can be canonically upgraded to an equivalence of $G$-crossed braided $\mathrm{W}^*$-tensor categories.
\end{theorem}

\section{$G$-crossed categorical extensions of conformal nets}
\label{sec: CategoricalExtensions}

In this section, we extend the definition of a categorical extension of a conformal net $\A$ from \cite[Sec. 3]{Gui21} to the context of a group $G$ acting on $\A$. We call these $G$-crossed categorical extensions of $\A$. Actually, we will show that, up to equivalence, there is a unique $G$-crossed categorical extension of $\A$, induced by the Connes fusion of representations and the braiding on $\Rep^G(\A)$ described in Section \ref{Sec: CatOfTwistedReps}. This uniqueness result will be relevant in the next section when discussing the category of representations of the fixed points conformal net $\A^G$ (see Proposition \ref{prop: MisEquivalenceOFG-X-BraidedCats}).

Before defining $G$-crossed categorical extensions, we need the following two pieces of notation. The first is the notion of adjoint commutativity. Let $H_1, H_2, H_3, H_4$ be Hilbert spaces and let $A: H_1\to H_2$, $B:H_2\to H_4$, $C: H_1\to H_3$ and $D: H_3\to H_4$ be bounded linear maps. Then, we say that the diagram
\begin{equation}\label{eq: Adjoint1}\begin{tikzcd}
	{H_1} & {H_3} \\
	{H_2} & {H_4}
	\arrow["C", from=1-1, to=1-2]
	\arrow["A"', from=1-1, to=2-1]
	\arrow["D", from=1-2, to=2-2]
	\arrow["B"', from=2-1, to=2-2]
\end{tikzcd}\end{equation}
\emph{commutes adjointly} if both the diagram and 
\begin{equation}\label{eq: Adjoint2}\begin{tikzcd}
	{H_1} & {H_3} \\
	{H_2} & {H_4}
	\arrow["C", from=1-1, to=1-2]
	\arrow["{A^*}", from=2-1, to=1-1]
	\arrow["B"', from=2-1, to=2-2]
	\arrow["{D^*}"', from=2-2, to=1-2]
\end{tikzcd}\end{equation}
commute. Note that the commutativity of \eqref{eq: Adjoint2} is equivalent to the commutativity of
\[\begin{tikzcd}
	{H_1} & {H_3} \\
	{H_2} & {H_4}.
	\arrow["A"', from=1-1, to=2-1]
	\arrow["{C^*}"', from=1-2, to=1-1]
	\arrow["D", from=1-2, to=2-2]
	\arrow["{B^*}", from=2-2, to=2-1]
\end{tikzcd}\]
In addition, if either of $A,B$ or $C,D$ are unitary, the adjoint commutativity of \eqref{eq: Adjoint1} is equivalent to its commutativity in the usual sense.

 Let $\A$ be a conformal net. In \cite[Eqs. (16) and (17)]{GcrossedbraidedRep}, we introduced two families of operators, denoted $L$ and $R$, in order to prove the hexagon conditions for the crossed braiding on $\Rep^{\Aut(\A)}(\A)$, analogously to \cite{Gui21}. We recall their definition here. Fix an automorphism $\varphi\in\Aut(\A)$ and let $H^\varphi$ be a $\varphi$-twisted representation of $\A$ and $\tilde{I}\in \Jcal_\R$ be an interval. For any $K^\mu\in \Rep^\mu(\A)$ and $\xi\in H_+^\varphi(\tilde{I})$, we denote by $Z^+(\xi, \tilde{I})$ the bounded linear operator $K^\mu\to H^\varphi_+(\tilde{I})\boxtimes K^\mu$ induced by
\[
\begin{array}{cccc}
     Z^+(\xi, \tilde{I}):& K^\mu &\to &H_+^\varphi(\tilde{I})\otimes K^\mu \\
     & \eta&\mapsto &\xi\otimes \eta.
\end{array}
\]
It is clear that $Z^+(\xi, \tilde{I})$ commutes with the actions of $\A_\R(\tilde{I}^{c+})$. We will similarly write $Z^-(\xi',\tilde{I}):K^\mu\to K^\mu\boxtimes H^\varphi_-(\tilde{I})$ for the map induced by $\eta\mapsto \eta\otimes\xi'$ for $\xi'\in H_-^\varphi(\tilde{I})$. Let $\alpha_{\tilde{I}}:[0,1]\to \R$ be a path in $\R$ from $\tilde{I}$ to $-\nicefrac{1}{4}\in \widetilde{S^1_-}$. We write $L(\xi, \tilde{I}): K^\mu\to H^\varphi\boxtimes K^\mu$ for the map
\begin{equation}\label{eq: OperatorL}
L(\xi,\tilde{I}): K^\mu\xrightarrow{Z^+(\xi, \tilde{I})} H_+^\varphi(\tilde{I})\boxtimes K^\mu\xrightarrow{\alpha_{\tilde{I}}^\bullet} H^\varphi_+(-)\boxtimes K^\mu = H^\varphi\boxtimes K^\mu.
\end{equation}
Given $\xi'\in H^\varphi_-(\tilde{I})$, we write $R(\xi', \tilde{I}): K^\mu\to K^\mu\boxtimes H^\varphi$ for the map
\begin{equation}
\label{eq: OperatorR}
R(\xi', \tilde{I}): K^\mu\xrightarrow{Z^-(\xi', \tilde{I})}K^\mu\boxtimes H_-^\varphi(\tilde{I})\xrightarrow{(\varrho*\alpha_{\tilde{I}})^\bullet} K^\mu\boxtimes H_-^\varphi(+) = K^\mu\boxtimes H^\varphi.
\end{equation}
Since path continuations intertwine the actions of $\A_\R$, we have that $L(\xi, \tilde{I})\in \Hom_{\A_\R(\tilde{I}^{c+})}(K^\mu, H^\varphi\boxtimes K^\mu)$ and $R(\xi', \tilde{I})\in \Hom_{\A_\R(\tilde{I}^{c-})}(K^\mu, K^\mu\boxtimes H^\varphi)$.

Let $G$ be a discrete group acting faithfully on $\A$, and recall that we write $\Rep^G(\A)$ for the $\mathrm{W}^*$-category of $G$-twisted representations of $\A$. In Section \ref{Sec: CatOfTwistedReps} we have constructed an action $T$ of $G$ on the $\mathrm{W}^*$-category $\Rep^G(\A)$ compatible with the $G$-grading. 

\begin{definition}\label{def: CategoricalExtension}
    Let $\A$ be a conformal net being acted on faithfully by a discrete group $G$. A $G$-crossed categorical extension of $\A$ consists of the following data:
    \begin{enumerate}
        \item a $\mathrm{W}^*$-tensor structure $\boxdot: \Rep^G(\A)\times\Rep^G(\A)\to \Rep^G(\A)$ with unit the vacuum representation $H_0\in\Rep^G(\A)$ (and unitary associators and unitors);
        \item an upgrade of the action $T$ on the $\mathrm{W}^*$-category $\Rep^G(\A)$ to an action by $\mathrm{W}^*$-tensor automorphisms on $(\Rep^G(\A), \boxdot)$, which in particular includes compatibility isomorphisms $T_g(-\boxdot -)\cong T_g(-)\boxdot T_g(-)$ for all $g\in G$;
        \item a natural family of unitary isomorphisms $\beta_{H, K}: H^g\boxdot K^h\xrightarrow{\cong} T_g(K^h)\boxdot H^g$ indexed by $g,h\in G$ and $H^g\in \Rep^g(\A)$ and $K^h\in \Rep^h(\A)$;
        \item for every $\tilde{I}\in \Jcal_\R$, $g,h\in G$, $H^g\in \Rep^g(\A)$, $\xi\in H^g_+(\tilde{I})$ and $K^h\in \Rep^h(\A)$, a bounded linear operator
        \[
        L(\xi, \tilde{I})\in \Hom_{\A_\R(\tilde{I}^{c+})}(K^h, H^g\boxdot K^h);
        \]
\item for every $\tilde{I}\in \Jcal_\R$, $g,h\in G$, $H^g\in \Rep^g(\A)$, $\xi\in H^g_-(\tilde{I})$ and $K^h\in \Rep^h(\A)$, a bounded linear operator
        \[
        R(\xi, \tilde{I})\in \Hom_{\A_\R(\tilde{I}^{c-})}(K^h, K^h \boxdot H^g)
        \]
\end{enumerate}
such that the following conditions are satisfied, where $g,h, k\in G$ and $H^g\in \Rep^g(\A)$, $K^h,\hat{K}^h\in \Rep^h(\A)$ and $R^k\in \Rep^k(\A)$,
\begin{enumerate}
    \item compatibility with the $G$-grading: the representation $H^g\boxdot K^h$ is $gh$-twisted;
    \item isotony: if $\tilde{I_1},\tilde{I_2}\in\Jcal_\R$ are intervals with $\tilde{I_1}\subset \tilde{I_2}$, for all $\xi\in H^g_+(\tilde{I_1})$ and $\eta\in H^g_-(\tilde{I_1})$, it holds that $L(\xi,\tilde{I_1}) = L(\xi, \tilde{I_2})$ and $R(\eta, \tilde{I_1}) = R(\eta, \tilde{I_2})$ when acting on $K^h$;
\item naturality: for any $F\in \Hom_{\Rep^G(\A)}(K^h, \hat{K}^h)$, the following diagrams commute for every $\tilde{I}\in\Jcal_\R$ and $\xi\in H^g_+(\tilde{I})$ and $\eta\in H^g_-(\tilde{I})$,
\[\begin{tikzcd}
	{K^h} && {\hat{K}^h} && {K^h} && {K^h\boxdot H^g} \\
	{H^g\boxdot K^h} && {H^g\boxdot \hat{K}^h} && {\hat{K}^h} && {\hat{K}^h\boxdot H^g};
	\arrow["F", from=1-1, to=1-3]
	\arrow["{L(\xi, \tilde{I})}"', from=1-1, to=2-1]
	\arrow["{L(\xi, \tilde{I})}", from=1-3, to=2-3]
	\arrow["{R(\eta, \tilde{I})}", from=1-5, to=1-7]
	\arrow["F"', from=1-5, to=2-5]
	\arrow["{F\boxdot \id}", from=1-7, to=2-7]
	\arrow["{\id\boxdot F}"', from=2-1, to=2-3]
	\arrow["{R(\eta, \tilde{I})}"', from=2-5, to=2-7]
\end{tikzcd}\]
\item unitality: for every $\tilde{I}\in \Jcal_\R$ and every $\xi\in H^g_+(\tilde{I})$ and $\eta\in H^g_-(\tilde{I})$, it holds that
\[
L(\xi,\tilde{I})\Omega = \xi\hspace{1cm} R(\eta,\tilde{I})\Omega = \eta
\]
under the identifications $H^g\cong H_0\boxdot H^g\cong H^g\boxdot H_0$;
\item Reeh-Schlieder property: for any $\tilde{I}\in \Jcal_\R$, the set $\{L(\xi, \tilde{I})\Omega\,|\, \xi\in H^g_+(\tilde{I})\}$ spans a dense subset of $H^g\cong H^g\boxdot H_0$, and the set $\{R(\eta, \tilde{I})\Omega\,|\, \eta\in H^g_-(\tilde{I})\}$ spans a dense subset of $H^g\cong H_0\boxdot H^g$;

\item density of fusion products: for any $\tilde{I}\in\Jcal_\R$, the set $\{L(\xi, \tilde{I})\chi\ |\ \xi\in H^g_+(\tilde{I}),\ \chi\in K^h\}$ spans a dense subspace of $H^g\boxdot K^h$, and the set $\{R(\eta, \tilde{I})\chi\ |\ \eta\in H^g_-(\tilde{I}),\ \chi\in K^h\}$ spans a dense subset of $K^h\boxdot H^g$;

\item locality: for any intervals $\tilde{I}, \tilde{J}\in \Jcal_\R$ with $\tilde{J}\subset \tilde{I}^{c+}$ and any $\xi\in H^g_+(\tilde{I})$ and $\eta\in K^h_-(\tilde{J})$, the following diagram commutes adjointly
\[\begin{tikzcd}
	{R^k} && {R^k\boxdot K^h} \\
	{H^g\boxdot R^k} && {H^g\boxdot R^k\boxdot K^h};
	\arrow["{R(\eta, \tilde{J})}", from=1-1, to=1-3]
	\arrow["{L(\xi, \tilde{I})}"', from=1-1, to=2-1]
	\arrow["{L(\xi, \tilde{I})}", from=1-3, to=2-3]
	\arrow["{R(\eta, \tilde{J})}"', from=2-1, to=2-3]
\end{tikzcd}\]
\item compatibility with the braiding: the following diagram commutes for all $\tilde{I}\in\Jcal_\R$ and $\xi\in H^g_-(\tilde{I}),$
\[\begin{tikzcd}
	{K^h} && {T_g(H^g)\boxdot K^h} \\
	{K^h\boxdot H^g} && {T_g(K^h)\boxdot T_g(H^g)} \\
	& {T_g(K^h\boxdot H^g)}.
	\arrow["{L(\Gamma_g\xi, \tilde{I})}", from=1-1, to=1-3]
	\arrow["{R(\xi, \tilde{I})}"', from=1-1, to=2-1]
	\arrow["{\beta_{T_g(H), K}}", from=1-3, to=2-3]
	\arrow["{\Gamma_g}"', from=2-1, to=3-2]
	\arrow["\cong", from=2-3, to=3-2]
\end{tikzcd}\]
\end{enumerate} 
\end{definition}

\begin{remark}
    What we have defined here as a $G$-crossed categorical extension is a generalization to the $G$-twisted setting of what \cite{Gui21} calls a \emph{vector-labelled} categorical extension. Furthermore, Gui also considers categorical extensions of subcategories of representations. We are interested in the whole category $\Rep^G(\A)$ and hence we do not include such freedom.  
\end{remark}

\begin{remark}\label{rk: FormulasLB}
    Given a $G$-crossed categorical extension $(\boxdot, \beta, L, R)$ of $\A$, we have the following extra compatibility of the braiding with the operator $L$. Let $H^g\in \Rep^g(\A)$ and $K^h\in \Rep^h(\A)$ be representations of $\A$ twisted by $g,h\in G$ respectively. For any intervals $\tilde{I},\tilde{J}\in\Jcal_\R$ such that $\tilde{J}\subset \tilde{I}^{c+}$, and all $\xi\in H_-^g(\tilde{I})$ and $\eta\in K^h_+(\tilde{J})$, it holds that
    \[
L(\Gamma_h\xi,\tilde{I})\eta = \beta_{K, H}\circ L(\eta, \tilde{J})\xi.
    \]
    Similarly, if $\tilde{I}\subset \tilde{J}^{c+}$ and $\xi\in H_+^g(\tilde{I})$ and $\eta\in K^h_-(\tilde{J})$, it holds that
    \[
    L(\xi, \tilde{I})\eta = \beta_{H, K}^{-1}\circ L(\Gamma_g\eta, \tilde{J})\xi.
    \]
    For proofs of these statements, see below \cite[Eq. (19)]{GcrossedbraidedRep}.
    In addition, if $\tilde{I}, \tilde{J},\tilde{O}\in\Jcal_\R$ are intervals such that $\tilde{I},\tilde{J}\subset\tilde{O}$ and $\xi\in H_+^\varphi(\tilde{I}),\ \eta\in K^\mu(\tilde{J})$, it holds that
    \begin{equation}\label{eq: JoinLs}
        L(\xi, \tilde{I})L(\eta, \tilde{J}) = L(L(\xi, \tilde{I})\eta, \tilde{O})
    \end{equation}
    when acting on any $R^\nu\in \Rep^\nu(\A)$, see \cite[Prop. 3.28]{GcrossedbraidedRep} and \cite[Prop. 3.6]{Gui21}.
\end{remark}

The results in \cite[Sec. 3]{GcrossedbraidedRep} show, except for the locality axiom, that the tensor structure $\boxtimes$ on $\Rep^G(\A)$, together with the families of operators $L$ and $R$ in Equations \eqref{eq: OperatorL} and \eqref{eq: OperatorR} is a $G$-crossed categorical extension~of~$\A$.

\begin{theorem}\label{Thm: ConnesCategoricalExtension}
    Let $\A$ be a conformal net and $G$ be a discrete group acting faithfully on $\A$. The Connes fusion tensor structure on $\Rep^G(\A)$, together with the compatibility with the $G$-action and the braiding $\mathbb{B}$ defined in Theorem \ref{Thm: BigCorollary}, and the families of operators $L$ and $R$ from Equations \eqref{eq: OperatorL} and \eqref{eq: OperatorR}, is a $G$-crossed categorical extension of $\A$.
\end{theorem}
\begin{proof}
    All the properties except for locality are easy to extract from the results in \cite[Sec. 3]{GcrossedbraidedRep}. In particular, compatibility with the braiding follows from \cite[Prop. 3.26]{GcrossedbraidedRep}. Let us show locality. Fix $\tilde{I}, \tilde{J}\in \Jcal_\R$ intervals with $\tilde{J}\subset \tilde{I}^{c+}$ and let $H^g\in \Rep^g(\A)$ and $K^h\in \Rep^h(\A)$. Given $\xi\in H^g_+(\tilde{I})$ and $\eta\in K^h_-(\tilde{J})$, the commutativity of
\[\begin{tikzcd}
	{R^k} && {R^k\boxdot K^h} \\
	{H^g\boxdot R^k} && {H^g\boxdot R^k\boxdot K^h};
	\arrow["{R(\eta, \tilde{J})}", from=1-1, to=1-3]
	\arrow["{L(\xi, \tilde{I})}"', from=1-1, to=2-1]
	\arrow["{L(\xi, \tilde{I})}", from=1-3, to=2-3]
	\arrow["{R(\eta, \tilde{J})}"', from=2-1, to=2-3]
\end{tikzcd}\]
is proved in \cite[Prop. 3.27]{GcrossedbraidedRep}. Let us first show the adjoint commutativity of the diagram 
\[\begin{tikzcd}
	{R^k} && {R^k\boxtimes K^h_-(\tilde{J})} \\
	{H^g_+(\tilde{I})\boxtimes R^k} && {H^g_+(\tilde{I})\boxtimes R^k\boxtimes K^h_-(\tilde{J})}
	\arrow["{Z^-(\eta, \tilde{J})}", from=1-1, to=1-3]
	\arrow["{Z^+(\xi, \tilde{I})}"', from=1-1, to=2-1]
	\arrow["{Z^+(\xi, \tilde{I})}", from=1-3, to=2-3]
	\arrow["{Z^-(\eta, \tilde{J})}"', from=2-1, to=2-3]
\end{tikzcd}\]
Fix $\xi'\in H^g_+(\tilde{I})$ and $\chi'\in R^k$. Then, it is straightforward to show that
\begin{equation}\label{eq: Z*Z}
Z^+(\xi, \tilde{I})^*(\xi'\otimes \chi') = \pi_{\tilde{I}}^R(Z^+(\xi,\tilde{I})^*Z^+(\xi',\tilde{I}))\chi'
\end{equation}
by pairing both sides against an arbitrary vector $\chi\in R^k$. Similarly, we have
\[
Z^+(\xi, \tilde{I})^*Z^-(\eta, \tilde{J})(\xi'\otimes\chi') = Z^+(\xi, \tilde{I})^*(\xi'\otimes \chi'\otimes\eta) = \pi^{R\boxtimes K}_{\tilde{I}}(Z^+(\xi, \tilde{I})^*Z^+(\xi', \tilde{I}))(\chi'\otimes\eta).
\]
Since $Z^+(\xi, \tilde{I})^*Z^+(\xi', \tilde{I})\in \A_\R(\tilde{I})$, the right-hand side of the expression above equals
\[
\pi^R_{\tilde{I}}(Z^+(\xi, \tilde{I})^*Z^+(\xi',\tilde{I}))\chi'\otimes\eta = Z^-(\eta, \tilde{J})\circ \pi^R_{\tilde{I}}(Z^+(\xi, \tilde{I})^*Z^+(\xi', \tilde{I}))\chi',
\]
which, by Equation \eqref{eq: Z*Z}, equals $Z^-(\eta, \tilde{J})Z^+(\xi, \tilde{I})^*(\xi'\otimes\chi')$. Hence, we have that $Z^-(\eta, \tilde{J})Z^+(\xi, \tilde{I})^* = Z^+(\xi, \tilde{I})^*Z^-(\eta, \tilde{J})$, as needed.

We can now show the adjoint commutativity of the relevant diagram. Let $\alpha_{\tilde{I}}$ be a path in $\R$ from $\tilde{I}$ to $-\nicefrac{1}{4}\in \widetilde{S^1_-}$ and $\beta_{\tilde{J}}$ be a path in $\R$ from $\tilde{J}$ to $\nicefrac{1}{4}\in \widetilde{S^1_+}$. We can assume, by taking homotopic paths if needed, that $\gamma:=(\alpha_{\tilde{I}},\beta_{\tilde{J}})$ is a path in $\ConfR$. Note that here we use that $\tilde{J}\subset \tilde{I}^{c+}$. We have to prove the adjoint commutativity of the outer diagram of 
\[\begin{tikzcd}
	{R^k} && {R^k\boxtimes K^h_-(\tilde{J})} && {R^k\boxtimes K^h_-({+})} \\
	{H^g_+(\tilde{I})\boxtimes R^k} && {H^g_+(\tilde{I})\boxtimes R^k\boxtimes K^h_-(\tilde{J})} && {H^g_+(\tilde{I})\boxtimes (R^k\boxtimes K^h_-({+}))} \\
	{H^g_+(-)\boxtimes R^k} && {(H^g_+(-)\boxtimes R^k)\boxtimes K^h_-(\tilde{J})} && {H^g_+(-)\boxtimes R^k\boxtimes K^h_-({+})}
	\arrow["{Z^-(\eta, \tilde{J})}", from=1-1, to=1-3]
	\arrow["{Z^+(\xi, \tilde{I})}"', from=1-1, to=2-1]
	\arrow["{\beta_{\tilde{J}}^\bullet}", from=1-3, to=1-5]
	\arrow["{Z^+(\xi, \tilde{I})}", from=1-3, to=2-3]
	\arrow["{Z^+(\xi, \tilde{I})}", from=1-5, to=2-5]
	\arrow["{Z^-(\eta, \tilde{J})}", from=2-1, to=2-3]
	\arrow["{\alpha_{\tilde{I}}^\bullet}"', from=2-1, to=3-1]
	\arrow["{\id\boxtimes \beta_{\tilde{J}}^\bullet}", from=2-3, to=2-5]
	\arrow["{\alpha_{\tilde{I}}^\bullet\boxtimes \id}", from=2-3, to=3-3]
	\arrow["{\alpha_{\tilde{I}}^\bullet}", from=2-5, to=3-5]
	\arrow["{Z^-(\eta, \tilde{J})}", from=3-1, to=3-3]
	\arrow["{\beta_{\tilde{J}}^\bullet}", from=3-3, to=3-5]
\end{tikzcd}\]
The top-left diagram commutes adjointly by the discussion above. The top-right diagram commutes trivially, and since $\beta_{\tilde{J}}^\bullet$ and $\id\boxtimes \beta_{\tilde{J}}^\bullet$ are unitary, it commutes adjointly. The same arguments work for the bottom-left diagram. For the bottom-right diagram, commutativity, and hence adjoint commutativity, follows from \cite[Prop. 3.10]{GcrossedbraidedRep}.
\end{proof}
 
We call the categorical extension from Theorem \ref{Thm: ConnesCategoricalExtension} the \emph{Connes $G$-crossed categorical extension of $\A$}. The main reason to introduce $G$-crossed categorical extensions is the following uniqueness theorem, analogous to \cite[Thm. 3.10]{Gui21}. Let us denote by $L^{\boxtimes}$ and $R^\boxtimes$ the operators in the Connes $G$-crossed categorical extension of $\A$.

\begin{theorem}\label{thm: UniquenessCategoricalExtensions}
    Let $\A$ be a conformal net and $G$ a discrete group acting faithfully on $\A$. Let $(\boxdot, L^\boxdot, R^\boxdot, \beta)$ be a $G$-crossed categorical extension of $\A$. Given $g,h\in G$ and $H^g\in \Rep^g(\A)$ and $K^h\in \Rep^h(\A)$, there exists a unique unitary isomorphism
     \[
    \Phi_{H, K}: H^g\boxtimes K^h\xrightarrow{\cong} H^g\boxdot K^h 
    \]
    such that for all $\tilde{I}\in\Jcal_\R$ and $\xi\in H^g_+(\tilde{I})$, $\eta\in H^g_-(\tilde{I})$ and $\chi\in K^h$,
    \begin{align}\label{eq: RelationsPhiLR}
    \Phi_{H, K}L^\boxtimes(\xi, \tilde{I})\chi  &= L^\boxdot(\xi, \tilde{I})\chi\\
   \Phi_{H, K} R^\boxtimes(\eta, \tilde{I})\chi  &= R^\boxdot(\eta, \tilde{I})\chi\nonumber.
    \end{align}
    In addition, the family $\Phi$ is natural in $H^g$ and $K^h$ and makes the following diagrams commute:
\[\begin{tikzcd}
	{H^g\boxtimes R^k\boxtimes K^h} && {H^g\boxtimes(R^k\boxdot K^h)} && {H^g\boxtimes H_0} && {H^g\boxdot H_0} \\
	&&&&& {H^g} \\
	{(H^g\boxdot R^k)\boxtimes K^h} && {H^g\boxdot R^k\boxdot K^h} && {H_0\boxtimes H^g} && {H_0\boxdot H^g} \\
	&& {H^g\boxtimes K^h} && {T_g(K^h)\boxtimes H^g} \\
	&& {H^g\boxdot K^h} && {T_g(K^h)\boxdot H^g}
	\arrow["{\id\boxtimes \Phi_{R, K}}", from=1-1, to=1-3]
	\arrow["{\Phi_{H, R}\boxtimes\id}"', from=1-1, to=3-1]
	\arrow["{\Phi_{H, R\boxdot K}}", from=1-3, to=3-3]
	\arrow["{\Phi_{H, H_0}}", from=1-5, to=1-7]
	\arrow["\cong"', from=1-5, to=2-6]
	\arrow["\cong", from=1-7, to=2-6]
	\arrow["{\Phi_{H\boxdot R, K}}"', from=3-1, to=3-3]
	\arrow["\cong", from=3-5, to=2-6]
	\arrow["{\Phi_{H_0, H}}", from=3-5, to=3-7]
	\arrow["\cong"', from=3-7, to=2-6]
	\arrow["{\B_{H, K}}", from=4-3, to=4-5]
	\arrow["{\Phi_{H, K}}"', from=4-3, to=5-3]
	\arrow["{\Phi_{T_g(K), H}}", from=4-5, to=5-5]
	\arrow["{\beta_{H, K}}"', from=5-3, to=5-5]
\end{tikzcd}\]
where $k\in G$ and $R^k\in \Rep^k(\A)$ is a twisted representation. Moreover, $\Rep^G(\A)$, together with $\boxdot$, $T$ and $\beta$, becomes a $G$-crossed braided $\mathrm{W}^*$-tensor category equivalent to $(\Rep^G(\A), \boxtimes, T, \B)$ under the equivalence $(\text{Id}_{\Rep^G(\A)}, \Phi)$.   
\end{theorem}

\begin{proof}[Proof of Theorem \ref{thm: UniquenessCategoricalExtensions}]
    We will first produce the natural family of unitary isomorphisms $\Phi$ and prove the relations \eqref{eq: RelationsPhiLR}. Fix $g,h\in  G$ and let $H^g\in \Rep^g(\A)$ and $K^h\in \Rep^h(\A)$ be twisted representations. Let $\tilde{I},\tilde{J}\in\Jcal_\R$ be intervals such that $\tilde{J}\subset \tilde{I}^{c+}$. Then, for every $\xi\in H^g_+(\tilde{I})$ and every $x\in \A(J)$, it holds that
    \[\hspace{-.2cm}
    L^\boxdot(\xi, \tilde{I})\circ\pi_0(x)(\Omega) = \pi_{\tilde{I}^{c+}}^{H}(x)\circ L^\boxdot(\xi, \tilde{I})\Omega = \pi^H_{\tilde{I}^{c+}}(x)(\xi) = \pi^H_{\tilde{I}^{c+}}(x)\circ L^\boxtimes(\xi, \tilde{I})\Omega  = L^\boxtimes(\xi, \tilde{I})\circ \pi_0(x)(\Omega),
    \]
    which implies that $L^\boxdot(\xi, \tilde{I}) = L^\boxtimes(\xi, \tilde{I})$ when acting on $H_0$. We can similarly show that, for every $\eta\in H^g_-(\tilde{I})$, the operators $R^\boxdot(\eta, \tilde{I})$ and $R^\boxtimes(\eta, \tilde{I})$ agree when acting on $H_0$. Using locality, we compute, for every $x,x'\in \A(I)$, $\xi, \xi'\in H^g_+(\tilde{I})$ and $\eta, \eta'\in K^h_-(\tilde{J})$,
    \begin{align*}
        \langle\pi^{H\boxdot K}_{\tilde{I}}(x) L^\boxdot(\xi, \tilde{I})\eta,& \pi^{H\boxdot K}_{\tilde{I}}(x') L^\boxdot(\xi', \tilde{I})\eta'\rangle \\& =  \langle\pi^{H\boxdot K}_{\tilde{I}}(x) L^\boxdot(\xi, \tilde{I})R^{\boxdot}(\eta, \tilde{J})\Omega, \pi^{H\boxdot K}_{\tilde{I}}(x') L^\boxdot(\xi', \tilde{I})R^{\boxdot}(\eta', \tilde{J})\Omega\rangle\\ & = \langle R^{\boxdot}(\eta', \tilde{J})^*L^{\boxdot}(\xi', \tilde{I})^*\pi_{\tilde{I}}^{H\boxdot K}(x'^*x)L^\boxdot(\xi, \tilde{I})R^{\boxdot}(\eta, \tilde{J})\Omega, \Omega\rangle\\& = \langle R^{\boxdot}(\eta', \tilde{J})^*R^{\boxdot}(\eta, \tilde{J})L^{\boxdot}(\xi', \tilde{I})^*\pi_{\tilde{I}}^{H\boxdot K}(x'^*x)L^\boxdot(\xi, \tilde{I})\Omega, \Omega\rangle.
    \end{align*}
  Note that, on the right-hand side of the equation above, all of $L^\boxdot(\xi, \tilde{I}), L^\boxdot(\xi', \tilde{I}), R^\boxdot(\eta, \tilde{I}), R^\boxdot(\eta', \tilde{I})$ act on $H_0$. Similarly, 
  \[\hspace{-.2cm}
  \langle\pi^{H\boxtimes K}_{\tilde{I}}(x) L^\boxtimes(\xi, \tilde{I})\eta, \pi^{H\boxtimes K}_{\tilde{I}}(x') L^\boxtimes(\xi', \tilde{I})\eta'\rangle = \langle R^{\boxtimes}(\eta', \tilde{J})^*R^{\boxtimes}(\eta, \tilde{J})L^{\boxtimes}(\xi', \tilde{I})^*\pi_{\tilde{I}}^{H\boxtimes K}(x'^*x)L^\boxtimes(\xi, \tilde{I})\Omega, \Omega\rangle,
  \]
  and on the right-hand side all the operators $L^\boxtimes$ and $R^\boxtimes$ are acting on $H_0$. Hence, 
  \[
  \langle\pi^{H\boxdot K}_{\tilde{I}}(x) L^\boxdot(\xi, \tilde{I})\eta, \pi^{H\boxdot K}_{\tilde{I}}(x') L^\boxdot(\xi', \tilde{I})\eta'\rangle = \langle\pi^{H\boxtimes K}_{\tilde{I}}(x) L^\boxtimes(\xi, \tilde{I})\eta, \pi^{H\boxtimes K}_{\tilde{I}}(x') L^\boxtimes(\xi', \tilde{I})\eta'\rangle.
  \]
  By density of fusion products and the Reeh-Schlieder property, there is a unique unitary map $\Phi^{\tilde{I},\tilde{J}}_{H,K}: H^g\boxtimes K^h\to H^g\boxdot K^h$ such that
  \[
  \Phi^{\tilde{I}, \tilde{J}}_{H, K}\circ \pi_{\tilde{I}}^{H\boxtimes K}(x)\circ L^{\boxtimes}(\xi, \tilde{I})\eta = \pi_{\tilde{I}}^{H\boxdot K}(x)\circ L^\boxdot(\xi, \tilde{I})\eta
  \]
  for all $x\in \A(I)$, $\xi\in H^g_+(\tilde{I})$ and $\eta\in K^h_-(\tilde{J})$. In particular, $\Phi^{\tilde{I}, \tilde{J}}_{H, K}$ intertwines the actions of $\A_\R(\tilde{I})$ on $H^g\boxtimes K^h$ and $H^g\boxdot K^h$. It is clear that, if $\tilde{I_0},\tilde{J_0}\in \Jcal_\R$ are intervals with $\tilde{I_0}\subset \tilde{I}$ and $\tilde{J_0}\subset \tilde{J}$, then $\Phi^{\tilde{I_0},\tilde{J_0}}_{H,K} = \Phi^{\tilde{I},\tilde{J}}_{H,K}.$ Hence, it follows that, for any $\tilde{I_1},\tilde{J_1}\in \Jcal_\R$ such that $\tilde{J_1}\subset \tilde{I_1}^{c+}$, then $\Phi^{\tilde{I_1},\tilde{J_1}}_{H,K} = \Phi^{\tilde{I}, \tilde{J}}_{H,K}$. We write $\Phi_{H,K} = \Phi^{\tilde{I}, \tilde{J}}_{H,K}$, which intertwines the actions of $\A_\R(\tilde{L})$ on $H^g\boxtimes K^h$ and $H^g\boxdot K^h$ for all $\tilde{L}\in \Jcal_\R$.

  By the Reeh-Schlieder property, it holds that $\Phi^{\tilde{I}, \tilde{J}}_{H, K} L^{\boxtimes}(\xi, \tilde{I})\eta = L^\boxdot(\xi, \tilde{I})\eta$ for all $\xi\in H^g_+(\tilde{I})$ and all $\eta\in K^h$. Hence, by naturality of $L^\boxtimes(\xi, \tilde{I})$ and $L^\boxdot(\xi, \tilde{I})$, it also holds that $\Phi_{H, \hat{K}}(\id \boxtimes F) = (\id\boxdot F)\Phi_{H,K}$ for any $\hat{H}^k\in \Rep^k(\A)$ and $F\in \Hom_{\Rep^k(\A)}(H^k, \hat{H}^k).$ Therefore, $\Phi_{H,K}$ is natural in $K$. Given $\xi\in H^g_+(\tilde{I})$ and $\eta\in K^h_-(\tilde{J})$, we also have that
  \[
  \Phi_{H, K}R^\boxtimes(\eta,\tilde{J})\xi = \Phi_{H,K}L^\boxtimes(\xi, \tilde{I})\eta = L^{\boxdot}(\xi, \tilde{I})\eta = R^{\boxdot}(\eta, \tilde{J})\xi.
  \]
  Since $H^g_+(\tilde{I})\subset H^g$ is dense, the equality $\Phi_{H, K}R^\boxtimes(\eta,\tilde{J})\xi = R^{\boxdot}(\eta, \tilde{J})\xi$ also holds when $\xi\in H^g.$ A similar argument to the one above shows that $\Phi_{H,K}$ is natural in $H$. Hence, we have proved the first part of the statement.

  We next prove the commutativity of the three diagrams in the statement. Fix $g,h,k\in G$ and let $H^g\in\Rep^g(\A)$, $K^h\in\Rep^h(\A)$ and $R^k\in\Rep^k(\A)$ be twisted representations. We first show the commutativity of the top-left diagram. Given $\tilde{I},\tilde{J}\in\Jcal_\R$ with $\tilde{J}\subset \tilde{I}^{c+}$, and $\xi\in H^g_+(\tilde{I})$ and $\eta\in K^h_+(\tilde{J})$, we have, for every $\chi\in R^k$,
  \begin{align*}
  L^\boxtimes(\xi, \tilde{I})R^\boxtimes(\eta, \tilde{J})\chi&\xrightarrow{\id\boxtimes \Phi_{R, K}}L^\boxtimes (\xi, \tilde{I})\Phi_{R, K}R^\boxtimes(\eta,\tilde{J})\chi  = L^\boxtimes(\xi, \tilde{I})R^\boxdot(\eta, \tilde{J})\chi\\ &\xrightarrow{\Phi_{H, R\boxdot K}} L^\boxdot(\xi, \tilde{I})R^{\boxdot}(\eta, \tilde{J})\chi.
  \end{align*}
  On the other hand, we have
  \begin{align*}
      L^\boxtimes(\xi, \tilde{I})R^\boxtimes(\eta, \tilde{J})\chi = R^{\boxtimes}(\eta, \tilde{J})L^\boxtimes(\xi, \tilde{I})\chi &\xrightarrow{\Phi_{H, R}\boxtimes \id}R^{\boxtimes}(\eta, \tilde{J})\Phi_{H, R}L^\boxtimes(\xi, \tilde{I})\chi = R^{\boxtimes}(\eta, \tilde{J})L^\boxdot(\xi, \tilde{I})\chi\\& \xrightarrow{\Phi_{H\boxdot R, K} } R^\boxdot(\eta, \tilde{J})L^\boxdot(\xi, \tilde{I})\chi = L^\boxdot(\xi,\tilde{I})R^\boxdot(\eta, \tilde{J})\chi,
  \end{align*}
  hence the top-left diagram commutes. For the diagrams involving the unitors, let $\tilde{I}\in \Jcal_\R$ and $\xi\in H^g_+(\tilde{I})$. We write $r^\boxtimes$ and $r^{\boxdot}$ for the right unitors of $\boxtimes$ and $\boxdot$ respectively. Then,
  \[
  r^{\boxdot}_H\circ \Phi_{H, H_0}\circ L^\boxtimes(\xi, \tilde{I})\Omega = r^{\boxdot}_H\circ L^\boxdot(\xi, \tilde{I})\Omega = \xi = r^{\boxtimes}\circ L^\boxtimes(\xi, \tilde{I})\Omega,
  \]
  hence the top triangle in the top-right diagram commutes. The other triangle commutes analogously. Thus, we have shown that $\Phi$ (together with the identity functor) produces an equivalence of $\mathrm{W}^*$-tensor categories between $(\Rep^G(\A), \boxtimes)$ and $(\Rep^G(\A), \boxdot).$ We will finally argue compatibility with the braidings, that is, the commutativity of the bottom diagram in the statement. Note that it is equivalent to show that, for every $\tilde{I}\in\Jcal_\R$ and every $\xi\in H^g_-(\tilde{I})$, the diagram
\[\begin{tikzcd}
	{T_g(H^g)\boxtimes K^h} && {T_g(K^h)\boxtimes T_g(H^g)} \\
	{T_g(H^g)\boxdot K^h} && {T_g(K^h)\boxdot T_g(H^g)}
	\arrow["{\B_{T_g(H), K}}", from=1-1, to=1-3]
	\arrow["{\Phi_{T_g(H), K}}"', from=1-1, to=2-1]
	\arrow["{\Phi_{T_g(K), T_g(H)}}", from=1-3, to=2-3]
	\arrow["{\beta_{T_g(H), K}}"', from=2-1, to=2-3]
\end{tikzcd}\]
commutes. Let $\tilde{I}\in \Jcal_\R$, $\xi\in H^g_-(\tilde{I})$ and $\eta\in K^h$. Then, the left-bottom leg of the diagram applied to $L^{\boxtimes}(\Gamma_g\xi, \tilde{I})\eta$ reads
  \[
  L^{\boxtimes}(\Gamma_g\xi, \tilde{I})\eta\xrightarrow{\Phi_{T_g(H),K}}L^{\boxdot}(\Gamma_g\xi, \tilde{I})\eta\xrightarrow{\beta_{T_g(H),K}}\beta_{T_g(H),K}L^\boxdot(\Gamma_g\xi, \tilde{I})\eta  = \Gamma_gR^\boxdot(\xi, \tilde{I})\eta,
  \]
  where we have dropped the isomorphism $T_g(K^h)\boxdot T_g(H^g)\cong T_g(K^h\boxdot H^g)$ for readability. Similarly, the top-right leg of the diagram applied to $L^{\boxtimes}(\Gamma_g\xi, \tilde{I})\eta$ reads
  \begin{align*}
  L^{\boxtimes}(\Gamma_g\xi, \tilde{I})\eta&\xrightarrow{\B_{T_g(H), K}}\B_{T_g(H), K}L^{\boxtimes}(\Gamma_g\xi, \tilde{I})\eta =\Gamma_g R^\boxtimes(\xi, \tilde{I})\eta\\& \xrightarrow{\Phi_{T_g(K), T_g(H)}}  \Phi_{T_g(K), T_g(H)}  R^\boxtimes(\Gamma_g\xi, \tilde{I})\Gamma_g\eta = \Gamma_g \Phi_{K, H} R^\boxtimes(\xi, \tilde{I})\eta = \Gamma_g R^\boxdot(\xi, \tilde{I})\eta.
  \end{align*}
  Hence, the last diagram commutes.

  The fact that $(\Rep^G(\A), \boxdot, T, \beta)$ is a $G$-crossed braided $\mathrm{W}^*$-tensor category equivalent to $(\Rep^G(\A), \boxtimes, T, \B)$ is clear by the commutativity of the three diagrams.
    \end{proof}

\section{The category of representations of the fixed-points conformal net}
\label{sec: FixedPoints}

Given a conformal net $\A$, acted on faithfully by a finite group $G$, the category $G-\Loc_I(\A)$ of $G$-localized endomorphisms of $\A$ on some fixed interval $I\in\Jcal_\mathrm{p}$ was introduced in \cite{muger05} to study the categories of localized endomorphisms of the fixed-points conformal net $\A^G$, constructed as follows. The category $\Rep^G(\A)$ serves the same purpose in this article.

Let $(\A,H_0, U, \Omega)$ be a conformal net and let $G\to \Aut(\A)$ be an injective group homomorphism defining a faithful action of a finite group $G$ on $\A$. In order to avoid cluttering the notation, we will continue denoting by $g$ the image of an element $g\in G$ in $\Aut(\A)$. Each element $g\in G$ comes equipped with a unitary $V_g\in U(H_0)$ implementing the automorphism $g\in \Aut(\A)$. Hence, the group $G$ acts on the vacuum Hilbert space $H_0$ of $\A$ via $g\mapsto V_g$. We denote by $H_0^G$ the subspace of $H_0$ of $G$-invariant vectors. The vacuum vector $\Omega$ is in $H_0^G$ by hypothesis. Given an interval $I\in \Jcal$, we define $\A^G(I):=\A(I)^G|_{H_0^G}$, which is a von Neumann algebra acting on $H_0^G$. The projective action $U$ of $\Diff^+(S^1)$ on $H_0$ restricts to a projective action $U^G$ on $H_0^G$, see \cite[Rk. 3.14]{GcrossedbraidedRep} or \cite{gf93} and \cite[Thm. 6.1.2]{weiner2007conformalcovariancerelatedproperties}. The tuple $(\A^G, H_0^G,  U^G, \Omega)$ is a conformal net, see for example \cite[Prop. 2.1]{MR1867572}.

\begin{definition}\label{Def: FixedPointsNet}
    The \emph{fixed-points conformal net} of $\A$ with respect to the action $G\to \Aut(\A)$ of a finite group $G$ is $\A^G:=(A^G, H_0^G, U^G, \Omega)$.
\end{definition}

The main goal of this section is to obtain an equivalence of balanced $\mathrm{W}^*$-tensor categories
\[
\Rep(\A^G)\cong (\Rep^G(\A))^G.
\]
We outline the proof strategy here. The main part of the proof is producing an equivalence of braided $\mathrm{W}^*$-tensor categories
\(
\Rep(\A^G)\cong (\Rep^G(\A))^G,
\) which is done in Theorem \ref{thm: RepFixedPointsIsEquivTwistedReps}. Checking that the equivalence is compatible with the balance is done in Theorem \ref{Thm: EquivariantizationCompatibleBalance}. In order to produce a braided $\mathrm{W}^*$-tensor equivalence $\Rep(\A^G)\cong (\Rep^G(\A))^G$ we introduce an auxiliary category $\M$, as follows. The vacuum representation $H_0\in\Rep(\A)$ can be restricted to an honest representation $r(H_0)\in \Rep(\A^G)$ of $\A^G$ under the inclusion $\A^G\subset \A$. Then, $r(H_0)$ has a canonical structure of a commutative separable $\mathrm{C}^*$-Frobenius algebra object in $\Rep(\A^G)$ in the sense of \cite[Def. 2.1]{Gui_2025}. We denote this algebra object by $\Gamma$ and write $\M$ for its $\mathrm{W}^*$-tensor category of unitary modules, with tensor structure $\boxtimes_\Gamma$. We define on $\M$ a $G$-grading, a $G$-action $\mathrm{T}$ and a natural family of unitary $G$-crossed braiding isomorphisms $\mathbb{B}^\Gamma$, without arguing the hexagon conditions yet, see Proposition \ref{Prop: ModGammaIsGGraded} and Lemma \ref{lemm: braidingModGamma}. The category $\M$ is the de-equivariantization of $\Rep(\A^G)$, but we will not argue this statement, see Remark \ref{rk: CommentDeequivariantization}. Instead, we obtain in Proposition \ref{prop: MEquivalenceOfGCats} an equivalence $\mathfrak{M}:\M\to \Rep^G(\A)$ of $G$-graded $\mathrm{W}^*$-categories with a $G$-action. Identifying $\M$ and $\Rep^G(\A)$ via $\mathfrak{M}$, we extend $(\boxtimes_\Gamma, \mathbb{B}^\Gamma)$ to a $G$-crossed categorical extension of $\A$ in the sense of Definition \ref{def: CategoricalExtension}. Then, by uniqueness of categorical extensions (see Theorem \ref{thm: UniquenessCategoricalExtensions}) we find that $(\M, \mathrm{T}, \boxtimes_\Gamma, \mathbb{B}^\Gamma)$ is a $G$-crossed braided $\mathrm{W}^*$-tensor category and that $\mathfrak{M}$ can be upgraded to an equivalence of $G$-crossed braided $\mathrm{W}^*$-tensor categories $\mathfrak{M}: \M\to \Rep^G(\A)$, see Proposition \ref{prop: MisEquivalenceOFG-X-BraidedCats}. Taking $G$-equivariantizations, we obtain an equivalence of braided $\mathrm{W}^*$-tensor categories $\mathfrak{M}^G: \M^G\to (\Rep^G(\A))^G$. The proof is concluded by defining a functor of braided $\mathrm{W}^*$-tensor categories $\mathfrak{D}: \Rep(\A^G)\to \M^G$ and finding an inverse functor $\mathfrak{R}$ to the composition
\[
\Rep(\A^G)\xrightarrow{\mathfrak{D}}\M^G\xrightarrow{\mathfrak{M}^G}(\Rep^G(\A))^G.
\]
The functor $\mathfrak{R}: (\Rep^G(\A))^G\to \Rep(\A^G)$ restricts a $G$-equivariant, $G$-twisted representation of $\A$ on a Hilbert space $H$ to the honest representation of $\A^G$ on the $G$-fixed points of $H$, see above Remark \ref{rk: NonHomogeneous}. In order to upgrade the equivalence
\[\begin{tikzcd}
	{\Rep(\A^G)} && {(\Rep^G(\A))^G}
	\arrow["{\mathfrak{M}^G\circ\mathfrak{D}}", shift left, from=1-1, to=1-3]
	\arrow["{\mathfrak{R}}", shift left, from=1-3, to=1-1]
\end{tikzcd}\]
of braided $\mathrm{W}^*$-tensor categories to an equivalence of balanced $\mathrm{W}^*$-tensor categories, we argue in Theorem \ref{Thm: EquivariantizationCompatibleBalance} that $\mathfrak{R}$ preserves the balance.

\subsection{The equivalence of braided $\mathrm{W}^*$-tensor categories}

In this section, we follow \cite{Gui_2025} to produce an equivalence of braided $\mathrm{W}^*$-tensor categories 
\[
\Rep(\A^G)\cong (\Rep^G(\A))^G.
\]
In order to differentiate the Connes fusions of representations in $\Rep(\A^G)$ and $\Rep^G(\A)$, we denote the former by $-\boxtimes ^G-$, and keep the notation $-\boxtimes-$ for the Connes fusion of twisted representations in $\Rep^G(\A)$. We also write $\beta$ for the braiding on $\Rep(\A^G)$, and keep the notation $\B$ for the $G$-crossed braiding on $\Rep^G(\A)$ and $\B^G$ for the braiding on its equivariantization $(\Rep^G(\A))^G$. We write $L^G$ and $R^G$ for the families of operators of the (non-crossed) Connes categorical extension of $\A^G$.

Let us recall the notion of a $\mathrm{C}^*$-Frobenius algebra object in $\Rep(\A^G)$, following \cite[Def. 2.1]{Gui_2025}.

\begin{definition}
    A \emph{$\mathrm{C}^*$-Frobenius algebra object} in $\Rep(\A^G)$ consists of an object $H\in \Rep(\A^G)$ together with a unit $\iota: H_0^G\to H$ and a multiplication $\mu: H\boxtimes^G H\to H$ such that $\mu\circ(\iota\boxtimes^G\id_H) = \id_H = \mu\circ(\id_H\boxtimes^G\iota)$ and the following diagram commutes adjointly,
\begin{equation}\label{eq: AssociativityFrobeniusCondition}\begin{tikzcd}
	{H\boxtimes^G H\boxtimes^G H} && {H\boxtimes^G H} \\
	{H\boxtimes^G H} && H.
	\arrow["{\id\boxtimes^G \mu}", from=1-1, to=1-3]
	\arrow["{\mu\boxtimes^G\id}"', from=1-1, to=2-1]
	\arrow["\mu", from=1-3, to=2-3]
	\arrow["\mu"', from=2-1, to=2-3]
\end{tikzcd}\end{equation}
Furthermore, if the unit $\iota$ is an isometry, meaning $\iota^*\iota = \id_{H_0^G}$, we say that the algebra object is \emph{normalized}. 
\end{definition}

Let $r: \Rep^G(\A)\to \Rep(\A^G)$ be the functor that restricts a $G$-twisted representation of $\A$ to an honest representation of $\A^G$ under the inclusion $\A^G\subset \A$. We will produce a commutative Frobenius $\mathrm{C}^*$-algebra structure on the object $r(H_0)\in\Rep(\A^G)$. Let $\iota: H_0^G\hookrightarrow r(H_0)$ be the inclusion of $\A^G$-representations. Given $\tilde{I}\in\Jcal_\R$, we define the linear map
\[
\begin{array}{cccc}
    \mu_{\tilde{I}}:  &r(H_0)\boxtimes^G r(H_0) & \to &r(H_0) \\
     & L^G(\pi_0(x)\Omega,\tilde{I})\xi&\mapsto  & \pi_0(x)\xi
\end{array}
\]
for $x\in\A(I)$ and $\xi\in H_0$. Since $G$ is finite, the inclusion $\A^G(I)\subset \A(I)$ is of finite index. Therefore, the same arguments as step 1 of the proof of the implication $(3)\implies (4)$ in \cite[Thm. 2.12]{Gui_2025} show that $\mu_{\tilde{I}}$ is bounded. If $y\in\A^G(I^c)$, it holds that
\begin{align*}
\pi_0(y)\mu_{\tilde{I}} (L^G(\pi_0(x)\Omega, \tilde{I})\xi) &= \pi_0(y)\pi_0(x)\xi = \pi_0(x)\pi_0(y)\xi \\&= \mu_{\tilde{I}}(L^G(\pi_0(x)\Omega, \tilde{I})\pi_0(y)\xi) = \mu_{\tilde{I}}(\pi^{r(H_0)\boxtimes^G r(H_0)}_{I^c}(y)L^G(\pi_0(x)\Omega, \tilde{I})\xi),
\end{align*}
showing that $\mu_{\tilde{I}}$ intertwines the actions of $\A^G(I^c)$. If $\tilde{J}\in\Jcal_\R$ is an interval with $\tilde{J}\subset\tilde{I}$ and $x\in \A(J)$, it then holds that $\mu_{\tilde{J}}(L^G(\pi_0(x)\Omega, \tilde{J})\xi) = \mu_{\tilde{I}}(L^G(\pi_0(x)\Omega, \tilde{I})\xi)$ for all $\xi\in H_0$. Hence, $\mu_{\tilde{I}}$ and $\mu_{\tilde{J}}$ agree. Therefore, $\mu_{\tilde{I}} = \mu_{\tilde{J}}$ for any intervals $\tilde{I}, \tilde{J}\in\Jcal_\R$, and we may write $\mu:=\mu_{\tilde{I}}$, which is a morphism of $\A^G$-representations.

\begin{proposition}\label{prop: GammaFrobeniusAlgebra}
    The object $r(H_0)\in\Rep(\A^G)$, equipped with the unit $\iota: H^G_0\to r(H_0)$ and the multiplication $\mu: r(H_0)\boxtimes^G r(H_0)\to r(H_0)$ is a normalized commutative separable $\mathrm{C}^*$-Frobenius algebra object in $\Rep(\A^G)$.
\end{proposition}
\begin{proof}
    The unitality condition is clear. To argue the associativity condition for $\mu$, that is, the commutativity of diagram \eqref{eq: AssociativityFrobeniusCondition}, let $\tilde{I}\in\Jcal_\R$, $x,y\in\A(I)$ and $\xi\in H_0$. Then, the top-right leg of the diagram \eqref{eq: AssociativityFrobeniusCondition} applied to $L^G(\pi_0(x)\Omega, \tilde{I})L^G(\pi_0(y)\Omega, \tilde{I})\xi$ reads
    \[
    L^G(\pi_0(x)\Omega, \tilde{I})L^G(\pi_0(y)\Omega, \tilde{I})\xi\xmapsto{\id\boxtimes^G\mu} L^G(\pi_0(x)\Omega, \tilde{I})\pi_0(y)\xi\xmapsto{\mu}\pi_0(x)\pi_0(y)\xi = \pi_0(xy)\xi, 
    \]
    and the left-bottom leg applied to the same vector $L^G(\pi_0(x)\Omega, \tilde{I})L^G(\pi_0(y)\Omega, \tilde{I})\xi = L^G(L^G(\pi_0(x)\Omega, \tilde{I})\pi_0(y)\Omega, \tilde{I})\xi$ (using Equation \eqref{eq: JoinLs}) reads
    \[
    L^G(L^G(\pi_0(x)\Omega, \tilde{I})\pi_0(y)\Omega, \tilde{I})\xi\xmapsto{\mu\boxtimes^G\id}
    L^G(\pi_0(xy)\Omega, \tilde{I})\xi\xmapsto{\mu}\pi_0(xy)\xi.\]
    To prove the adjoint commutativity of \eqref{eq: AssociativityFrobeniusCondition}, let $\tilde{I},\tilde{J}\in \Jcal_\R$ be intervals with $\tilde{J}\subset\tilde{I}^{c+}$ and fix $x\in \A(I)$ and $y\in\A(J)$. Consider the diagram
\[\begin{tikzcd}
	{r(H_0)} && {r(H_0)\boxtimes^G r(H_0)} && {r(H_0)} \\
	{r(H_0)\boxtimes^Gr(H_0)} && {r(H_0)\boxtimes^Gr(H_0)\boxtimes^G r(H_0)} && {r(H_0)\boxtimes^G r(H_0)} \\
	{r(H_0)} && {r(H_0)\boxtimes^G r(H_0)} && {r(H_0)}.
	\arrow["{R^G(\pi_0(y)\Omega, \tilde{J})}", from=1-1, to=1-3]
	\arrow["{L^G(\pi_0(x)\Omega, \tilde{I})}"', from=1-1, to=2-1]
	\arrow["\mu", from=1-3, to=1-5]
	\arrow["{L^G(\pi_0(x)\Omega, \tilde{I})}", from=1-3, to=2-3]
	\arrow["{L^G(\pi_0(x)\Omega, \tilde{I})}", from=1-5, to=2-5]
	\arrow["{R^G(\pi_0(y)\Omega, \tilde{J})}"', from=2-1, to=2-3]
	\arrow["\mu"', from=2-1, to=3-1]
	\arrow["{\id\boxtimes^G\mu}"', from=2-3, to=2-5]
	\arrow["{\mu\boxtimes^G\id}", from=2-3, to=3-3]
	\arrow["\mu", from=2-5, to=3-5]
	\arrow["{R^G(\pi_0(y)\Omega, \tilde{J})}"', from=3-1, to=3-3]
	\arrow["\mu"', from=3-3, to=3-5]
\end{tikzcd}\]
By locality of the Connes categorical extension of $\A^G$, the top-left inner diagram commutes adjointly. By naturality and the fact that $(F_1\boxtimes^G F_2)^* = F_1^*\boxtimes^G F_2^*$, the two off-diagonal inner diagrams commute adjointly. We have shown above that the bottom-right diagram commutes, hence the outer diagram commutes, meaning that $\mu\circ R^G(\pi_0(y)\Omega, \tilde{J})\circ \pi_0(x) = \pi_0(x)\circ \mu \circ R^G(\pi_0(y)\Omega, \tilde{J})$ as maps from $H_0$ to $H_0$. Therefore, it holds that $$\hspace{-.4cm}\mu R^G(\pi_0(y)\Omega, \tilde{J}) \pi_0(x)^* = \mu R^G(\pi_0(y)\Omega, \tilde{J}) \pi_0(x^*) = \pi_0(x^*) \mu  R^G(\pi_0(y)\Omega, \tilde{J}) = \pi_0(x)^* \mu  R^G(\pi_0(y)\Omega, \tilde{J})$$
    as maps from $H_0$ to $H_0$. This is exactly the adjoint commutativity of the outer diagram. Since all the other three inner diagrams commute adjointly, by density of fusion products, the bottom-right diagram also commutes adjointly. This is the Frobenius relation for $\mu$. Since $\iota: H_0^G\to r(H_0)$ is an inclusion, it is trivially an isometry.

    To argue the separability of $(r(H_0), \mu, \iota)$, note that $\mu$ is a surjective morphism with bounded right-inverse $(\iota\boxtimes^G\id)$, hence $z:=\mu\mu^*\in \End_{\Rep(\A^G)}(r(H_0))$ is positive and invertible. The map $\Delta:=\mu^*\circ z^{-1}: r(H_0)\to r(H_0)\boxtimes^G r(H_0)$ is a splitting of $\mu$ as an $(r(H_0)-r(H_0))$-bimodule.

    Hence, it remains to argue commutativity of the algebra $(r(H_0), \mu)$. Let $\tilde{I},\tilde{J}\in\Jcal_\R$ be intervals with $\tilde{I}\subset\tilde{J}^{c+}$ and fix $x\in A(I)$, $y\in\A(J)$. Then, using Remark \ref{rk: FormulasLB},
    \begin{align*}
    \mu\circ\beta_{\Gamma, \Gamma}\circ L^G(\pi_0(x)\Omega, \tilde{I})\pi_0(y)\Omega &= \mu L^G(\pi_0(y)\Omega, \tilde{J})\pi_0(x)\Omega = \pi_0(y)\pi_0(x)\Omega\\& = \pi_0(x)\pi_0(y)\Omega = \mu L^G(\pi_0(x)\Omega, \tilde{I})\pi_0(y)\Omega.
    \end{align*}
\end{proof}

When referring to $r(H_0)$ as a $\mathrm{C}^*$-Frobenius algebra object in $\Rep(\A^G)$, we will denote it by $\Gamma$. Recall that we write $\Delta = \mu^*\circ (\mu\mu^*)^{-1}: \Gamma\to \Gamma\boxtimes^G\Gamma$ for the splitting of $\mu$ as a $(\Gamma-\Gamma)$-bimodule witnessing the fact that $\Gamma$ is separable. It is straightforward to check that the diagram
\begin{equation}\label{eq: DeltaAdjointCommutativity}\begin{tikzcd}
	\Gamma && {\Gamma\boxtimes^G\Gamma} \\
	{\Gamma\boxtimes^G\Gamma} && {\Gamma\boxtimes^G\Gamma\boxtimes^G\Gamma}
	\arrow["\Delta", from=1-1, to=1-3]
	\arrow["\Delta"', from=1-1, to=2-1]
	\arrow["{\id\boxtimes^G\Delta}", from=1-3, to=2-3]
	\arrow["{\Delta\boxtimes^G \id}"', from=2-1, to=2-3]
\end{tikzcd}\end{equation}
commutes adjointly. We next introduce the category of unitary modules over $\Gamma$.

\begin{definition}
    A \emph{unitary left module} over $\Gamma$ is an object $H\in\Rep(\A^G)$ together with a morphism $\mu_H:\Gamma\boxtimes^G H\to H$ in $\Rep(\A^G)$ such that $\mu_H\circ (\iota\boxtimes^G\id_{H}) = \id_H$ and the following diagram commutes adjointly
\begin{equation}\label{eq: ModuleCondition}\begin{tikzcd}
	{\Gamma\boxtimes^G\Gamma\boxtimes^G H} && {\Gamma\boxtimes^GH} \\
	{\Gamma\boxtimes^GH} && H.
	\arrow["{\id\boxtimes^G\mu_H}", from=1-1, to=1-3]
	\arrow["{\mu\boxtimes^G\id}"', from=1-1, to=2-1]
	\arrow["{\mu_H}", from=1-3, to=2-3]
	\arrow["{\mu_H}"', from=2-1, to=2-3]
\end{tikzcd}\end{equation}
We call the adjoint commutativity of the diagram above the \emph{module and Frobenius relation} for $(H, \mu_H)$. We write $\M$ for the $\mathrm{W}^*$-category of unitary left $\Gamma$-modules in $\Rep(\A^G)$.
\end{definition}

We may endow $\M$ with a tensor structure as follows. Given $(H,\mu_H),(K,\mu_K)\in \M$, we define their tensor product to be
\begin{equation}\label{eq: TensorModules}
H\boxtimes_\Gamma K:= \Gamma\boxtimes^G_{\Gamma\boxtimes^G \Gamma}(H\boxtimes ^G K),
\end{equation}
where the right-hand side denotes the tensor product of the left $(\Gamma\boxtimes ^G\Gamma)$-module $H\boxtimes ^G K$ with the right $(\Gamma\boxtimes ^G\Gamma)$-module $\Gamma$ in $\Rep(\A^G).$ Note that, since $\Gamma$ is separable, the right-hand side of Equation \eqref{eq: TensorModules} makes sense, even though $\Rep(\A^G)$ is not abelian. With this tensor product, the category $\M$ is a $\mathrm{W}^*$-tensor category with unit $\Gamma$, see \cite[Sec. 4.2.4]{braidedFC}. By definition, there is a canonical morphism $H\boxtimes^G K\to H\boxtimes _\Gamma K$, which we denote $\mu_{H,K}$. The tensor product $H\boxtimes_\Gamma K$ can also be defined by regarding $H$ as a $\Gamma$-bimodule via the right-action $H\boxtimes^G \Gamma\xrightarrow{\beta^{-1}_{\Gamma, H}} \Gamma\boxtimes^G H \xrightarrow{\mu_H} H$, see \cite[Sec. 4.4.7]{braidedFC}. This means that $H\boxtimes_\Gamma K$ is the splitting of the idempotent
\begin{align}\label{eq: Idempotent}
p_{H, K}:H\boxtimes^G K\cong H\boxtimes^G H_0^G\boxtimes^G K&\xrightarrow{\id\boxtimes^G (\Delta\circ \iota)\boxtimes^G \id} H\boxtimes^G \Gamma\boxtimes^G\Gamma\boxtimes^G K\\&\xrightarrow{\beta^{-1}_{\Gamma, H}\boxtimes^G\id}\Gamma\boxtimes^G H\boxtimes^G \Gamma\boxtimes^G K\xrightarrow{\mu_H\boxtimes^G \mu_K} H\boxtimes^G K.\nonumber
\end{align}
Since $\Rep(\A^G)$ is idempotent complete, the object $H\boxtimes_\Gamma K\in \Rep(\A^G)$ exists. In addition, $H\boxtimes_\Gamma K$ is also the coequalizer of $H\boxtimes^G\Gamma\boxtimes^G K\xrightarrow{\id\boxtimes^G \mu_K} H\boxtimes^G K$ and $H\boxtimes^G\Gamma\boxtimes^G K\xrightarrow{\beta^{-1}_{\Gamma, H}\boxtimes^G \id} \Gamma\boxtimes^G H\boxtimes^G K\xrightarrow{\mu_H\boxtimes^G\id} H\boxtimes^G K$, see for example the proof of \cite[Lem. 2.3]{MR3459961}. We stress again that $\Rep(\A^G)$ is not abelian, but since $\Gamma$ is separable, the coequalizer exists, equals the splitting in Equation \eqref{eq: Idempotent} and the morphism $\mu_{H,K}$ is a projection onto the object splitting $p_{H, K}$. The object $H\boxtimes_\Gamma K$ gets equipped with the structure of a unitary left $\Gamma$-module by the morphism $\Gamma\boxtimes^G H\boxtimes ^G K\xrightarrow{\mu_H\boxtimes^G\id} H\boxtimes^G K$. The left unitor on a module $(H, \mu_H)\in\M$ is induced by the map $\mu_H: \Gamma\boxtimes^G H\to H$, and the right unitor is induced by $H\boxtimes^G\Gamma\xrightarrow{\beta_{\Gamma, H}^{-1}}\Gamma\boxtimes^G H\xrightarrow{\mu_H}H$. We will need the following properties of $\mu_{H, K}$. 

\begin{lemma}\label{lemm: technicalLemmasMuHK}
    Let $(H, \mu_H), (K, \mu_K)\in \M$. Then, it holds that
    \begin{enumerate}
        \item the morphism $\mu_{H, K}: H\boxtimes^GK\to H\boxtimes_\Gamma K$ is surjective;
        \item the following diagrams commute
\[\begin{tikzcd}
	{\Gamma\boxtimes^G H\boxtimes^G K} && {H\boxtimes^G \Gamma\boxtimes^G K} \\
	{H\boxtimes^G K} && {H\boxtimes^G K} \\
	& {H\boxtimes_\Gamma K}
	\arrow["{\beta_{\Gamma, H}\boxtimes^G\id}", from=1-1, to=1-3]
	\arrow["{\mu_H\boxtimes^G \id}"', from=1-1, to=2-1]
	\arrow["{\id\boxtimes^G \mu_K}", from=1-3, to=2-3]
	\arrow["{\mu_{H, K}}"', from=2-1, to=3-2]
	\arrow["{\mu_{H, K}}", from=2-3, to=3-2]
\end{tikzcd}\]
\[\begin{tikzcd}
	{\Gamma\boxtimes^G H\boxtimes^G K} && {H\boxtimes^G\Gamma\boxtimes^G K} \\
	{H\boxtimes^G K} & {H\boxtimes_{\Gamma} K } & {H\boxtimes^G K} \\
	{H\boxtimes^G \Gamma\boxtimes^G K} && {\Gamma\boxtimes^G H\boxtimes ^G K}.
	\arrow["{\beta_{\Gamma, H}^*\boxtimes^G\id}", from=1-1, to=1-3]
	\arrow["{\id\boxtimes\mu_K}", from=1-3, to=2-3]
	\arrow["{\mu_H^*\boxtimes^G\id}", from=2-1, to=1-1]
	\arrow["{\mu_{H, K}}"', from=2-1, to=2-2]
	\arrow["{\id\boxtimes^G\mu_K^*}"', from=2-1, to=3-1]
	\arrow["{\mu_{H, K}^*}"', from=2-2, to=2-3]
	\arrow["{\beta_{H, \Gamma}\boxtimes^G\id}"', from=3-1, to=3-3]
	\arrow["{\mu_H\boxtimes^G \id}"', from=3-3, to=2-3]
\end{tikzcd}\]
    \end{enumerate}
\end{lemma}
\begin{proof}
    Part $(\mathrm{i})$ is obvious since $H\boxtimes_\Gamma K$ is defined as the splitting of an idempotent on $H\boxtimes^G K$. The first diagram in part $(\mathrm{ii})$ follows from the following graphical calculus. We read diagrams from bottom to top. We do not include the labels for the objects and morphisms for readability, but make them explicit here. The algebra $\Gamma$ is denoted in black, the module $H$ in red and $K$ in blue. The unit $\iota: H_0^G\to \Gamma$ is denoted by a hollow circle, and the $(\Gamma-\Gamma)$-bimodule morphism $\Delta: \Gamma\to \Gamma\boxtimes^G\Gamma$ by a square. The module actions $\mu_H$, $\mu_K$ and the multiplication $\mu$ are denoted by solid bullets.
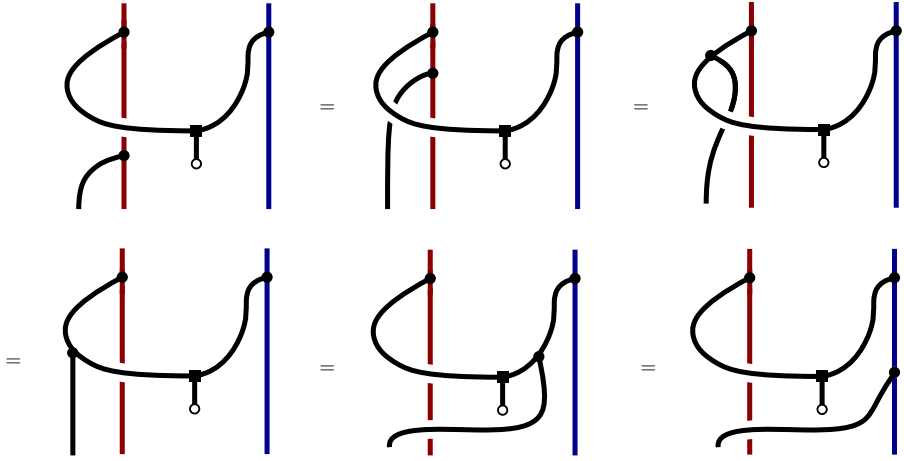
\[
\scalebox{.7}{ 
\begin{tikzpicture}[ipe stylesheet]
  \filldraw[shift={(450.462, 522.636)}, xscale=0.6058, yscale=0.5511, darkred, line width=2.7]
    (0, 0)
     -- (0, -200);
  \draw[white, line width=10.0]
    (433.577, 416.409)
     .. controls (432, 428) and (460, 426) .. (478.6667, 425.6667)
     .. controls (497.3333, 425.3333) and (506.6667, 426.6667) .. (511.9296, 431.0493)
     .. controls (517.1925, 435.432) and (518.385, 442.864) .. (528, 456);
  \filldraw[shift={(279.385, 522.227)}, xscale=0.6058, yscale=0.5511, darkred, line width=2.7]
    (0, 0)
     -- (0, -200);
  \draw[white, line width=10.0]
    (257.577, 416.409)
     .. controls (256, 428) and (284, 426) .. (302.6667, 425.6667)
     .. controls (321.3333, 425.3333) and (330.6667, 426.6667) .. (335.9296, 431.0493)
     .. controls (341.1925, 435.432) and (342.385, 442.864) .. (337.538, 464.909);
  \draw[shift={(427.154, 546.886)}, xscale=0.6058, yscale=0.5511, line width=2.7]
    (0, 0)
     .. controls (0, 28) and (4, 44) .. (8.6667, 57.3333)
     .. controls (13.3333, 70.6667) and (18.6667, 81.3333) .. (22, 91.3333)
     .. controls (25.3333, 101.3333) and (26.6667, 110.6667) .. (25.3333, 119.3333)
     .. controls (24, 128) and (20, 136) .. (4, 144);
  \filldraw[shift={(115.385, 654.205)}, xscale=0.6058, yscale=0.5511, darkred, line width=2.7]
    (0, 0)
     -- (0, -200);
  \draw[shift={(154.154, 585.864)}, xscale=0.6058, yscale=0.5511, white, line width=10.0]
    (0, 0)
     .. controls (-68, 0) and (-82, 6) .. (-93, 13)
     .. controls (-104, 20) and (-112, 28) .. (-114, 40)
     .. controls (-116, 52) and (-112, 68) .. (-64, 96);
  \draw[shift={(256.538, 543.977)}, xscale=0.6058, yscale=0.5511, line width=2.7]
    (0, 0)
     .. controls (0, 36) and (0, 58) .. (1.3333, 75)
     .. controls (2.6667, 92) and (5.3333, 104) .. (12.6667, 114)
     .. controls (20, 124) and (32, 132) .. (40, 132);
  \filldraw[shift={(280.769, 654.205)}, xscale=0.6058, yscale=0.5511, darkred, line width=2.7]
    (0, 0)
     -- (0, -200);
  \draw[shift={(319.538, 585.864)}, xscale=0.6058, yscale=0.5511, white, line width=10.0]
    (0, 0)
     .. controls (-68, 0) and (-82, 6) .. (-93, 13)
     .. controls (-104, 20) and (-112, 28) .. (-114, 40)
     .. controls (-116, 52) and (-112, 68) .. (-64, 96);
  \draw[shift={(319.538, 585.864)}, xscale=0.6058, yscale=0.5511, line width=2.7]
    (0, 0)
     .. controls (-68, 0) and (-82, 6) .. (-93, 13)
     .. controls (-104, 20) and (-112, 28) .. (-114, 40)
     .. controls (-116, 52) and (-112, 68) .. (-64, 96);
  \filldraw[shift={(154.154, 585.864)}, xscale=0.6058, yscale=0.5511, line width=2.7, fill=white]
    (0, 0)
     -- (0, -32);
  \filldraw[shift={(192.923, 654.205)}, xscale=0.6058, yscale=0.5511, draw=darkblue, line width=2.7, fill=darkred]
    (0, 0)
     -- (0, -200);
  \pic[ipe mark scale=5.0]
     at (192.9231, 638.7727) {ipe disk};
  \pic[ipe mark scale=5.0]
     at (154.1538, 585.8636) {ipe square};
  \pic[ipe mark scale=5.0, fill=white]
     at (154.1538, 568.2273) {ipe fdisk};
  \draw[shift={(91.154, 543.977)}, xscale=0.6058, yscale=0.5511, line width=2.7]
    (0, 0)
     .. controls (0, 24) and (6, 32) .. (12, 38)
     .. controls (18, 44) and (24, 48) .. (40, 52);
  \pic[ipe mark scale=5.0]
     at (115.3846, 572.6364) {ipe disk};
  \draw[shift={(154.154, 585.864)}, xscale=0.6058, yscale=0.5511, line width=2.7]
    (0, 0)
     .. controls (-68, 0) and (-82, 6) .. (-93, 13)
     .. controls (-104, 20) and (-112, 28) .. (-114, 40)
     .. controls (-116, 52) and (-112, 68) .. (-64, 96);
  \draw[shift={(154.154, 585.864)}, xscale=0.6058, yscale=0.5511, line width=2.7]
    (0, 0)
     .. controls (24, 4) and (36, 26) .. (41.3333, 41)
     .. controls (46.6667, 56) and (45.3333, 64) .. (45.6667, 73)
     .. controls (46, 82) and (48, 92) .. (64, 96);
  \filldraw[shift={(319.538, 585.864)}, xscale=0.6058, yscale=0.5511, line width=2.7, fill=white]
    (0, 0)
     -- (0, -32);
  \filldraw[shift={(358.308, 654.205)}, xscale=0.6058, yscale=0.5511, draw=darkblue, line width=2.7, fill=darkred]
    (0, 0)
     -- (0, -200);
  \pic[ipe mark scale=5.0]
     at (358.3077, 638.7727) {ipe disk};
  \pic[ipe mark scale=5.0]
     at (319.5385, 585.8636) {ipe square};
  \pic[ipe mark scale=5.0, fill=white]
     at (319.5385, 568.2273) {ipe fdisk};
  \draw[shift={(319.538, 585.864)}, xscale=0.6058, yscale=0.5511, line width=2.7]
    (0, 0)
     .. controls (24, 4) and (36, 26) .. (41.3333, 41)
     .. controls (46.6667, 56) and (45.3333, 64) .. (45.6667, 73)
     .. controls (46, 82) and (48, 92) .. (64, 96);
  \pic[ipe mark scale=5.0]
     at (280.7692, 616.7273) {ipe disk};
  \draw[shift={(115.385, 645.386)}, xscale=0.6058, yscale=0.5511, darkred, line width=2.7]
    (0, 0)
     -- (0, -28);
  \pic[ipe mark scale=5.0]
     at (115.3846, 638.7727) {ipe disk};
  \draw[shift={(280.769, 643.182)}, xscale=0.6058, yscale=0.5511, darkred, line width=2.7]
    (0, 0)
     -- (0, -24);
  \pic[ipe mark scale=5.0]
     at (280.7692, 638.7727) {ipe disk};
  \filldraw[shift={(451.385, 654.909)}, xscale=0.6058, yscale=0.5511, darkred, line width=2.7]
    (0, 0)
     -- (0, -200);
  \draw[shift={(490.154, 586.568)}, xscale=0.6058, yscale=0.5511, white, line width=10.0]
    (0, 0)
     .. controls (-68, 0) and (-82, 6) .. (-93, 13)
     .. controls (-104, 20) and (-112, 28) .. (-114, 40)
     .. controls (-116, 52) and (-112, 68) .. (-64, 96);
  \draw[shift={(490.154, 586.568)}, xscale=0.6058, yscale=0.5511, line width=2.7]
    (0, 0)
     .. controls (-68, 0) and (-82, 6) .. (-93, 13)
     .. controls (-104, 20) and (-112, 28) .. (-114, 40)
     .. controls (-116, 52) and (-112, 68) .. (-64, 96);
  \filldraw[shift={(490.154, 586.568)}, xscale=0.6058, yscale=0.5511, line width=2.7, fill=white]
    (0, 0)
     -- (0, -32);
  \filldraw[shift={(528.923, 654.909)}, xscale=0.6058, yscale=0.5511, draw=darkblue, line width=2.7, fill=darkred]
    (0, 0)
     -- (0, -200);
  \pic[ipe mark scale=5.0]
     at (528.9231, 639.4773) {ipe disk};
  \pic[ipe mark scale=5.0]
     at (490.1538, 586.5682) {ipe square};
  \pic[ipe mark scale=5.0, fill=white]
     at (490.1538, 568.9318) {ipe fdisk};
  \draw[shift={(490.154, 586.568)}, xscale=0.6058, yscale=0.5511, line width=2.7]
    (0, 0)
     .. controls (24, 4) and (36, 26) .. (41.3333, 41)
     .. controls (46.6667, 56) and (45.3333, 64) .. (45.6667, 73)
     .. controls (46, 82) and (48, 92) .. (64, 96);
  \draw[shift={(451.385, 643.886)}, xscale=0.6058, yscale=0.5511, darkred, line width=2.7]
    (0, 0)
     -- (0, -24);
  \pic[ipe mark scale=5.0]
     at (451.3846, 639.4773) {ipe disk};
  \pic[ipe mark scale=5.0]
     at (429.5769, 626.25) {ipe disk};
  \draw[shift={(440.481, 597.223)}, xscale=0.6058, yscale=0.5511, line width=2.7]
    (0, 0)
     .. controls (3.3333, 10) and (4.6667, 19.3333) .. (3.3333, 28)
     .. controls (2, 36.6667) and (-2, 44.6667) .. (-18, 52.6667);
  \filldraw[shift={(114.462, 522.909)}, xscale=0.6058, yscale=0.5511, darkred, line width=2.7]
    (0, 0)
     -- (0, -200);
  \draw[shift={(153.231, 454.568)}, xscale=0.6058, yscale=0.5511, white, line width=10.0]
    (0, 0)
     .. controls (-68, 0) and (-82, 6) .. (-93, 13)
     .. controls (-104, 20) and (-112, 28) .. (-114, 40)
     .. controls (-116, 52) and (-112, 68) .. (-64, 96);
  \draw[shift={(153.231, 454.568)}, xscale=0.6058, yscale=0.5511, line width=2.7]
    (0, 0)
     .. controls (-68, 0) and (-82, 6) .. (-93, 13)
     .. controls (-104, 20) and (-112, 28) .. (-114, 40)
     .. controls (-116, 52) and (-112, 68) .. (-64, 96);
  \filldraw[shift={(153.231, 454.568)}, xscale=0.6058, yscale=0.5511, line width=2.7, fill=white]
    (0, 0)
     -- (0, -32);
  \filldraw[shift={(192, 522.909)}, xscale=0.6058, yscale=0.5511, draw=darkblue, line width=2.7, fill=darkred]
    (0, 0)
     -- (0, -200);
  \pic[ipe mark scale=5.0]
     at (192, 507.4773) {ipe disk};
  \pic[ipe mark scale=5.0]
     at (153.2308, 454.5682) {ipe square};
  \pic[ipe mark scale=5.0, fill=white]
     at (153.2308, 436.9318) {ipe fdisk};
  \draw[shift={(153.231, 454.568)}, xscale=0.6058, yscale=0.5511, line width=2.7]
    (0, 0)
     .. controls (24, 4) and (36, 26) .. (41.3333, 41)
     .. controls (46.6667, 56) and (45.3333, 64) .. (45.6667, 73)
     .. controls (46, 82) and (48, 92) .. (64, 96);
  \draw[shift={(114.462, 511.886)}, xscale=0.6058, yscale=0.5511, darkred, line width=2.7]
    (0, 0)
     -- (0, -24);
  \pic[ipe mark scale=5.0]
     at (114.4615, 507.4773) {ipe disk};
  \pic[ipe mark scale=5.0]
     at (87.9231, 466.9773) {ipe disk};
  \draw[shift={(318.154, 453.886)}, xscale=0.6058, yscale=0.5511, white, line width=10.0]
    (0, 0)
     .. controls (-68, 0) and (-82, 6) .. (-93, 13)
     .. controls (-104, 20) and (-112, 28) .. (-114, 40)
     .. controls (-116, 52) and (-112, 68) .. (-64, 96);
  \draw[shift={(318.154, 453.886)}, xscale=0.6058, yscale=0.5511, line width=2.7]
    (0, 0)
     .. controls (-68, 0) and (-82, 6) .. (-93, 13)
     .. controls (-104, 20) and (-112, 28) .. (-114, 40)
     .. controls (-116, 52) and (-112, 68) .. (-64, 96);
  \filldraw[shift={(318.154, 453.886)}, xscale=0.6058, yscale=0.5511, line width=2.7, fill=white]
    (0, 0)
     -- (0, -32);
  \filldraw[shift={(356.923, 522.227)}, xscale=0.6058, yscale=0.5511, draw=darkblue, line width=2.7, fill=darkred]
    (0, 0)
     -- (0, -200);
  \pic[ipe mark scale=5.0]
     at (356.9231, 506.7955) {ipe disk};
  \pic[ipe mark scale=5.0]
     at (318.1538, 453.8864) {ipe square};
  \pic[ipe mark scale=5.0, fill=white]
     at (318.1538, 436.25) {ipe fdisk};
  \draw[shift={(318.154, 453.886)}, xscale=0.6058, yscale=0.5511, line width=2.7]
    (0, 0)
     .. controls (24, 4) and (36, 26) .. (41.3333, 41)
     .. controls (46.6667, 56) and (45.3333, 64) .. (45.6667, 73)
     .. controls (46, 82) and (48, 92) .. (64, 96);
  \draw[shift={(279.385, 511.205)}, xscale=0.6058, yscale=0.5511, darkred, line width=2.7]
    (0, 0)
     -- (0, -24);
  \pic[ipe mark scale=5.0]
     at (279.3846, 506.7955) {ipe disk};
  \draw[line width=2.7]
    (257.5769, 416.4091)
     .. controls (256, 428) and (284, 426) .. (302.6667, 425.6667)
     .. controls (321.3333, 425.3333) and (330.6667, 426.6667) .. (335.9295, 431.0492)
     .. controls (341.1923, 435.4318) and (342.3846, 442.8636) .. (337.5385, 464.9091);
  \pic[ipe mark scale=5.0]
     at (337.5385, 464.9091) {ipe disk};
  \draw[shift={(489.231, 454.295)}, xscale=0.6058, yscale=0.5511, white, line width=10.0]
    (0, 0)
     .. controls (-68, 0) and (-82, 6) .. (-93, 13)
     .. controls (-104, 20) and (-112, 28) .. (-114, 40)
     .. controls (-116, 52) and (-112, 68) .. (-64, 96);
  \draw[shift={(489.231, 454.295)}, xscale=0.6058, yscale=0.5511, line width=2.7]
    (0, 0)
     .. controls (-68, 0) and (-82, 6) .. (-93, 13)
     .. controls (-104, 20) and (-112, 28) .. (-114, 40)
     .. controls (-116, 52) and (-112, 68) .. (-64, 96);
  \filldraw[shift={(489.231, 454.295)}, xscale=0.6058, yscale=0.5511, line width=2.7, fill=white]
    (0, 0)
     -- (0, -32);
  \filldraw[shift={(528, 522.636)}, xscale=0.6058, yscale=0.5511, draw=darkblue, line width=2.7, fill=darkred]
    (0, 0)
     -- (0, -200);
  \pic[ipe mark scale=5.0]
     at (528, 507.2045) {ipe disk};
  \pic[ipe mark scale=5.0]
     at (489.2308, 454.2955) {ipe square};
  \pic[ipe mark scale=5.0, fill=white]
     at (489.2308, 436.6591) {ipe fdisk};
  \draw[shift={(489.231, 454.295)}, xscale=0.6058, yscale=0.5511, line width=2.7]
    (0, 0)
     .. controls (24, 4) and (36, 26) .. (41.3333, 41)
     .. controls (46.6667, 56) and (45.3333, 64) .. (45.6667, 73)
     .. controls (46, 82) and (48, 92) .. (64, 96);
  \draw[shift={(450.462, 511.614)}, xscale=0.6058, yscale=0.5511, darkred, line width=2.7]
    (0, 0)
     -- (0, -24);
  \pic[ipe mark scale=5.0]
     at (450.4615, 507.2045) {ipe disk};
  \pic[ipe mark scale=5.0, white]
     at (528, 456.5) {ipe disk};
  \pic[ipe mark scale=5.0]
     at (528, 456.5) {ipe disk};
  \node[ipe node]
     at (220, 596) {=};
  \node[ipe node]
     at (388, 596) {=};
  \node[ipe node]
     at (52, 460) {=};
  \node[ipe node]
     at (220, 456) {=};
  \draw[line width=2.7]
    (433.577, 416.409)
     .. controls (432, 428) and (460, 426) .. (478.6667, 425.6667)
     .. controls (497.3333, 425.3333) and (506.6667, 426.6667) .. (511.9296, 431.0493)
     .. controls (517.1925, 435.432) and (518.385, 442.864) .. (528, 456);
  \node[ipe node]
     at (392, 456) {=};
  \filldraw[line width=2.7, fill=white]
    (88, 468)
     -- (88, 412);
\end{tikzpicture}}\]

Let us show the commutativity of the bottom part of the second diagram. The top part follows analogously. Since the module $(K, \mu_K)\in \M$ is unitary by definition, it holds that $\mu_K^* = (\id_\Gamma\boxtimes^G \mu_K)\circ ((\Delta\circ\iota)\boxtimes^G\id_K)$. Hence, the bottom leg of the diagram equals $p_{H,K}$. Therefore, it remains to show that $p_{H, K}^*p_{H,K}^{\phantom{*}} = p_{H, K}^{\phantom{*}}$. Since $p_{H, K}$ is an idempotent, it is enough to argue it is self-adjoint. We do so again by graphical calculus. A morphism in green denotes the adjoint of the corresponding morphism in black. We have
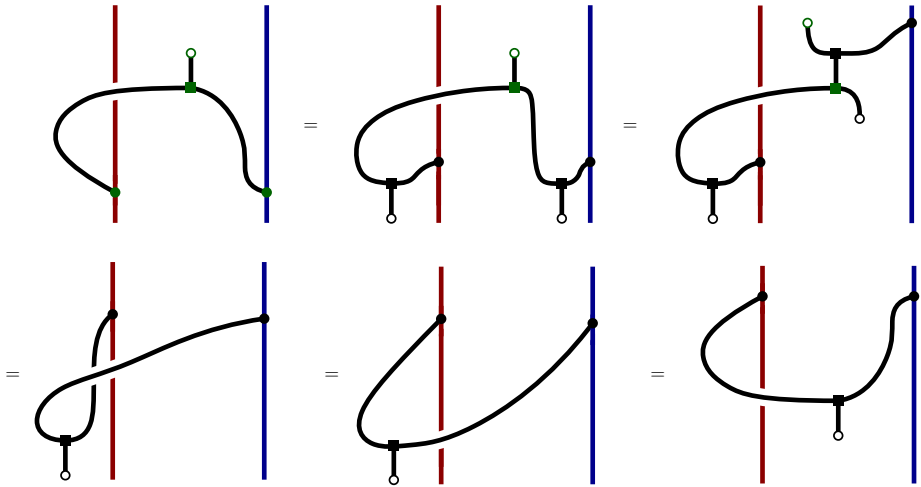
\[
\scalebox{.65}{
\begin{tikzpicture}[ipe stylesheet]
  \draw[shift={(74.577, 492.8)}, xscale=0.6827, yscale=0.6278, line width=2.7]
    (0, 0)
     .. controls (12, 0) and (18, 6) .. (21, 15)
     .. controls (24, 24) and (24, 36) .. (24, 56)
     .. controls (24, 76) and (24, 104) .. (40, 116);
  \filldraw[shift={(101.885, 470.2)}, xscale=0.6827, yscale=-0.6278, darkred, line width=2.7]
    (0, 0)
     -- (0, -200);
  \draw[shift={(74.577, 492.8)}, xscale=0.6827, yscale=0.6278, white, line width=10.0]
    (0, 0)
     .. controls (-16, 0) and (-24, 8) .. (-24, 18)
     .. controls (-24, 28) and (-16, 40) .. (-2.6667, 48)
     .. controls (10.6667, 56) and (29.3333, 60) .. (55.6667, 72)
     .. controls (82, 84) and (116, 104) .. (168, 112);
  \draw[shift={(333.577, 696.2)}, xscale=0.6827, yscale=0.6278, line width=2.7]
    (0, 0)
     .. controls (12, 0) and (14, -4) .. (15, -14.6667)
     .. controls (16, -25.3333) and (16, -42.6667) .. (16.6667, -55.3333)
     .. controls (17.3333, -68) and (18.6667, -76) .. (20.6667, -80.6667)
     .. controls (22.6667, -85.3333) and (25.3333, -86.6667) .. (28.6667, -87.3333)
     .. controls (32, -88) and (36, -88) .. (39.3333, -88)
     .. controls (42.6667, -88) and (45.3333, -88) .. (48, -86.6667)
     .. controls (50.6667, -85.3333) and (53.3333, -82.6667) .. (54.6667, -79.3333)
     .. controls (56, -76) and (56, -72) .. (64, -68);
  \filldraw[shift={(289.885, 618.356)}, xscale=0.6827, yscale=-0.6278, darkred, line width=2.7]
    (0, 0)
     -- (0, -200);
  \draw[shift={(330.846, 696.2)}, xscale=0.6827, yscale=0.6278, white, line width=10.0]
    (0, 0)
     .. controls (-32, 0) and (-60, -6) .. (-78.6667, -11.6667)
     .. controls (-97.3333, -17.3333) and (-106.6667, -22.6667) .. (-114, -28.6667)
     .. controls (-121.3333, -34.6667) and (-126.6667, -41.3333) .. (-128.6667, -50.6667)
     .. controls (-130.6667, -60) and (-129.3333, -72) .. (-124.6667, -78.6667)
     .. controls (-120, -85.3333) and (-112, -86.6667) .. (-104, -87.3333)
     .. controls (-96, -88) and (-88, -88) .. (-83, -84)
     .. controls (-78, -80) and (-76, -72) .. (-60, -68);
  \draw[shift={(147, 696.2)}, xscale=0.6827, yscale=-0.6278, line width=2.7]
    (0, 0)
     .. controls (24, 4) and (36, 26) .. (41.3333, 41)
     .. controls (46.6667, 56) and (45.3333, 64) .. (45.6667, 73)
     .. controls (46, 82) and (48, 92) .. (64, 96);
  \filldraw[shift={(103.308, 618.356)}, xscale=0.6827, yscale=-0.6278, darkred, line width=2.7]
    (0, 0)
     -- (0, -200);
  \draw[shift={(147, 696.2)}, xscale=0.6827, yscale=-0.6278, white, line width=10.0]
    (0, 0)
     .. controls (-68, 0) and (-82, 6) .. (-93, 13)
     .. controls (-104, 20) and (-112, 28) .. (-114, 40)
     .. controls (-116, 52) and (-112, 68) .. (-64, 96);
  \filldraw[shift={(147, 696.2)}, xscale=0.6827, yscale=-0.6278, line width=2.7, fill=white]
    (0, 0)
     -- (0, -32);
  \filldraw[shift={(190.692, 618.356)}, xscale=0.6827, yscale=-0.6278, draw=darkblue, line width=2.7, fill=darkred]
    (0, 0)
     -- (0, -200);
  \pic[ipe mark scale=5.0, darkgreen]
     at (190.6923, 635.9333) {ipe disk};
  \pic[ipe mark scale=5.0, draw=darkgreen, fill=white]
     at (147, 716.2889) {ipe fdisk};
  \draw[shift={(147, 696.2)}, xscale=0.6827, yscale=-0.6278, line width=2.7]
    (0, 0)
     .. controls (-68, 0) and (-82, 6) .. (-93, 13)
     .. controls (-104, 20) and (-112, 28) .. (-114, 40)
     .. controls (-116, 52) and (-112, 68) .. (-64, 96);
  \draw[shift={(103.308, 628.4)}, xscale=0.6827, yscale=-0.6278, darkred, line width=2.7]
    (0, 0)
     -- (0, -28);
  \pic[ipe mark scale=5.0, darkgreen]
     at (103.3077, 635.9333) {ipe disk};
  \pic[ipe mark scale=5.0, darkgreen]
     at (147, 696.2) {ipe square};
  \filldraw[shift={(333.577, 696.2)}, xscale=0.6827, yscale=-0.6278, line width=2.7, fill=white]
    (0, 0)
     -- (0, -32);
  \filldraw[shift={(377.269, 618.356)}, xscale=0.6827, yscale=-0.6278, draw=darkblue, line width=2.7, fill=darkred]
    (0, 0)
     -- (0, -200);
  \pic[ipe mark scale=5.0]
     at (377.2692, 653.5111) {ipe disk};
  \pic[ipe mark scale=5.0, draw=darkgreen, fill=white]
     at (333.5769, 716.2889) {ipe fdisk};
  \draw[shift={(289.885, 628.4)}, xscale=0.6827, yscale=-0.6278, darkred, line width=2.7]
    (0, 0)
     -- (0, -28);
  \filldraw[shift={(262.577, 620.867)}, xscale=0.6827, yscale=-0.6278, line width=2.7, fill=white]
    (0, 0)
     -- (0, -32);
  \pic[ipe mark scale=5.0]
     at (262.5769, 640.9556) {ipe square};
  \pic[ipe mark scale=5.0, fill=white]
     at (262.5769, 620.8667) {ipe fdisk};
  \draw[shift={(330.846, 696.2)}, xscale=0.6827, yscale=0.6278, line width=2.7]
    (0, 0)
     .. controls (-32, 0) and (-60, -6) .. (-78.6667, -11.6667)
     .. controls (-97.3333, -17.3333) and (-106.6667, -22.6667) .. (-114, -28.6667)
     .. controls (-121.3333, -34.6667) and (-126.6667, -41.3333) .. (-128.6667, -50.6667)
     .. controls (-130.6667, -60) and (-129.3333, -72) .. (-124.6667, -78.6667)
     .. controls (-120, -85.3333) and (-112, -86.6667) .. (-104, -87.3333)
     .. controls (-96, -88) and (-88, -88) .. (-83, -84)
     .. controls (-78, -80) and (-76, -72) .. (-60, -68);
  \draw[shift={(289.885, 643.467)}, xscale=0.6827, yscale=-0.6278, darkred, line width=2.7]
    (0, 0)
     -- (0, -28);
  \pic[ipe mark scale=5.0]
     at (289.8846, 653.5111) {ipe disk};
  \filldraw[shift={(360.885, 620.867)}, xscale=0.6827, yscale=-0.6278, line width=2.7, fill=white]
    (0, 0)
     -- (0, -32);
  \pic[ipe mark scale=5.0]
     at (360.8846, 640.9556) {ipe square};
  \pic[ipe mark scale=5.0, fill=white]
     at (360.8846, 620.8667) {ipe fdisk};
  \pic[ipe mark scale=5.0, darkgreen]
     at (333.5769, 696.2) {ipe square};
  \filldraw[shift={(475.308, 618.2)}, xscale=0.6827, yscale=-0.6278, darkred, line width=2.7]
    (0, 0)
     -- (0, -200);
  \draw[shift={(516.269, 696.044)}, xscale=0.6827, yscale=0.6278, white, line width=10.0]
    (0, 0)
     .. controls (-32, 0) and (-60, -6) .. (-78.6667, -11.6667)
     .. controls (-97.3333, -17.3333) and (-106.6667, -22.6667) .. (-114, -28.6667)
     .. controls (-121.3333, -34.6667) and (-126.6667, -41.3333) .. (-128.6667, -50.6667)
     .. controls (-130.6667, -60) and (-129.3333, -72) .. (-124.6667, -78.6667)
     .. controls (-120, -85.3333) and (-112, -86.6667) .. (-104, -87.3333)
     .. controls (-96, -88) and (-88, -88) .. (-83, -84)
     .. controls (-78, -80) and (-76, -72) .. (-60, -68);
  \filldraw[shift={(519, 696.044)}, xscale=0.6827, yscale=-0.6278, line width=2.7, fill=white]
    (0, 0)
     -- (0, -32);
  \filldraw[shift={(562.692, 618.2)}, xscale=0.6827, yscale=-0.6278, draw=darkblue, line width=2.7, fill=darkred]
    (0, 0)
     -- (0, -200);
  \draw[shift={(475.308, 628.244)}, xscale=0.6827, yscale=-0.6278, darkred, line width=2.7]
    (0, 0)
     -- (0, -28);
  \filldraw[shift={(448, 620.711)}, xscale=0.6827, yscale=-0.6278, line width=2.7, fill=white]
    (0, 0)
     -- (0, -32);
  \pic[ipe mark scale=5.0]
     at (448, 640.8) {ipe square};
  \pic[ipe mark scale=5.0, fill=white]
     at (448, 620.7111) {ipe fdisk};
  \draw[shift={(516.269, 696.044)}, xscale=0.6827, yscale=0.6278, line width=2.7]
    (0, 0)
     .. controls (-32, 0) and (-60, -6) .. (-78.6667, -11.6667)
     .. controls (-97.3333, -17.3333) and (-106.6667, -22.6667) .. (-114, -28.6667)
     .. controls (-121.3333, -34.6667) and (-126.6667, -41.3333) .. (-128.6667, -50.6667)
     .. controls (-130.6667, -60) and (-129.3333, -72) .. (-124.6667, -78.6667)
     .. controls (-120, -85.3333) and (-112, -86.6667) .. (-104, -87.3333)
     .. controls (-96, -88) and (-88, -88) .. (-83, -84)
     .. controls (-78, -80) and (-76, -72) .. (-60, -68);
  \draw[shift={(475.308, 643.311)}, xscale=0.6827, yscale=-0.6278, darkred, line width=2.7]
    (0, 0)
     -- (0, -28);
  \pic[ipe mark scale=5.0]
     at (475.3077, 653.3556) {ipe disk};
  \pic[ipe mark scale=5.0, fill=white]
     at (532.6538, 678.4667) {ipe fdisk};
  \pic[ipe mark scale=5.0]
     at (519, 716.1333) {ipe square};
  \draw[line width=2.7]
    (502.6154, 733.7111)
     .. controls (502.6154, 726.1778) and (503.9808, 721.1556) .. (506.484, 718.6444)
     .. controls (508.9872, 716.1333) and (512.6282, 716.1333) .. (515.8141, 716.1333)
     .. controls (519, 716.1333) and (521.7308, 716.1333) .. (526.3526, 716.1111)
     .. controls (530.9744, 716.0889) and (537.4872, 716.0444) .. (542.6859, 719.1944)
     .. controls (547.8846, 722.3444) and (551.7692, 728.6889) .. (562.6923, 733.7111);
  \pic[ipe mark scale=5.0]
     at (562.6923, 733.7111) {ipe disk};
  \draw[shift={(519, 696.044)}, xscale=0.6827, yscale=0.6278, line width=2.7]
    (0, 0)
     .. controls (16, 0) and (20, -12) .. (20, -24);
  \pic[ipe mark scale=5.0, draw=darkgreen, fill=white]
     at (502.6154, 733.7111) {ipe fdisk};
  \pic[ipe mark scale=5.0, darkgreen]
     at (519, 696.0444) {ipe square};
  \filldraw[shift={(189.269, 470.2)}, xscale=0.6827, yscale=-0.6278, draw=darkblue, line width=2.7, fill=darkred]
    (0, 0)
     -- (0, -200);
  \filldraw[shift={(74.577, 472.711)}, xscale=0.6827, yscale=-0.6278, line width=2.7, fill=white]
    (0, 0)
     -- (0, -32);
  \pic[ipe mark scale=5.0]
     at (74.5769, 492.8) {ipe square};
  \pic[ipe mark scale=5.0, fill=white]
     at (74.5769, 472.7111) {ipe fdisk};
  \draw[shift={(101.885, 555.578)}, xscale=0.6827, yscale=-0.6278, darkred, line width=2.7]
    (0, 0)
     -- (0, -28);
  \pic[ipe mark scale=5.0]
     at (101.8846, 565.6222) {ipe disk};
  \pic[ipe mark scale=5.0]
     at (189.2692, 563.1111) {ipe disk};
  \draw[shift={(74.577, 492.8)}, xscale=0.6827, yscale=0.6278, line width=2.7]
    (0, 0)
     .. controls (-16, 0) and (-24, 8) .. (-24, 18)
     .. controls (-24, 28) and (-16, 40) .. (-2.6667, 48)
     .. controls (10.6667, 56) and (29.3333, 60) .. (55.6667, 72)
     .. controls (82, 84) and (116, 104) .. (168, 112);
  \filldraw[shift={(291.308, 467.533)}, xscale=0.6827, yscale=-0.6278, darkred, line width=2.7]
    (0, 0)
     -- (0, -200);
  \filldraw[shift={(378.692, 467.533)}, xscale=0.6827, yscale=-0.6278, draw=darkblue, line width=2.7, fill=darkred]
    (0, 0)
     -- (0, -200);
  \draw[shift={(291.308, 477.578)}, xscale=0.6827, yscale=-0.6278, darkred, line width=2.7]
    (0, 0)
     -- (0, -28);
  \draw[shift={(291.308, 562.956)}, xscale=0.6827, yscale=0.6278, white, line width=10.0]
    (0, 0)
     .. controls (-40, -44) and (-60, -68) .. (-66.6667, -84.6667)
     .. controls (-73.3333, -101.3333) and (-66.6667, -110.6667) .. (-58, -114.6667)
     .. controls (-49.3333, -118.6667) and (-38.6667, -117.3333) .. (-26, -116)
     .. controls (-13.3333, -114.6667) and (1.3333, -113.3333) .. (27.6667, -98.6667)
     .. controls (54, -84) and (92, -56) .. (128, -4);
  \draw[shift={(291.308, 562.956)}, xscale=0.6827, yscale=0.6278, line width=2.7]
    (0, 0)
     .. controls (-40, -44) and (-60, -68) .. (-66.6667, -84.6667)
     .. controls (-73.3333, -101.3333) and (-66.6667, -110.6667) .. (-58, -114.6667)
     .. controls (-49.3333, -118.6667) and (-38.6667, -117.3333) .. (-26, -116)
     .. controls (-13.3333, -114.6667) and (1.3333, -113.3333) .. (27.6667, -98.6667)
     .. controls (54, -84) and (92, -56) .. (128, -4);
  \draw[shift={(291.308, 552.911)}, xscale=0.6827, yscale=-0.6278, darkred, line width=2.7]
    (0, 0)
     -- (0, -28);
  \pic[ipe mark scale=5.0]
     at (291.3077, 562.9556) {ipe disk};
  \draw[shift={(378.692, 547.889)}, xscale=0.6827, yscale=-0.6278, darkblue, line width=2.7]
    (0, 0)
     -- (0, -28);
  \pic[ipe mark scale=5.0]
     at (378.6923, 560.4444) {ipe disk};
  \filldraw[shift={(264, 470.044)}, xscale=0.6827, yscale=-0.6278, line width=2.7, fill=white]
    (0, 0)
     -- (0, -32);
  \pic[ipe mark scale=5.0]
     at (264, 490.1333) {ipe square};
  \pic[ipe mark scale=5.0, fill=white]
     at (264, 470.0444) {ipe fdisk};
  \draw[shift={(520.308, 515.711)}, xscale=0.6827, yscale=0.6278, line width=2.7]
    (0, 0)
     .. controls (24, 4) and (36, 26) .. (41.3333, 41)
     .. controls (46.6667, 56) and (45.3333, 64) .. (45.6667, 73)
     .. controls (46, 82) and (48, 92) .. (64, 96);
  \filldraw[shift={(476.615, 593.556)}, xscale=0.6827, yscale=0.6278, darkred, line width=2.7]
    (0, 0)
     -- (0, -200);
  \draw[shift={(520.308, 515.711)}, xscale=0.6827, yscale=0.6278, white, line width=10.0]
    (0, 0)
     .. controls (-68, 0) and (-82, 6) .. (-93, 13)
     .. controls (-104, 20) and (-112, 28) .. (-114, 40)
     .. controls (-116, 52) and (-112, 68) .. (-64, 96);
  \filldraw[shift={(520.308, 515.711)}, xscale=0.6827, yscale=0.6278, line width=2.7, fill=white]
    (0, 0)
     -- (0, -32);
  \filldraw[shift={(564, 593.556)}, xscale=0.6827, yscale=0.6278, draw=darkblue, line width=2.7, fill=darkred]
    (0, 0)
     -- (0, -200);
  \pic[ipe mark scale=5.0]
     at (564, 575.9778) {ipe disk};
  \pic[ipe mark scale=5.0, fill=white]
     at (520.3077, 495.6222) {ipe fdisk};
  \draw[shift={(520.308, 515.711)}, xscale=0.6827, yscale=0.6278, line width=2.7]
    (0, 0)
     .. controls (-68, 0) and (-82, 6) .. (-93, 13)
     .. controls (-104, 20) and (-112, 28) .. (-114, 40)
     .. controls (-116, 52) and (-112, 68) .. (-64, 96);
  \draw[shift={(476.615, 583.511)}, xscale=0.6827, yscale=0.6278, darkred, line width=2.7]
    (0, 0)
     -- (0, -28);
  \pic[ipe mark scale=5.0]
     at (476.6154, 575.9778) {ipe disk};
  \pic[ipe mark scale=5.0]
     at (520.3077, 515.7111) {ipe square};
  \node[ipe node]
     at (212, 672) {$=$};
  \node[ipe node]
     at (396, 672) {$=$};
  \node[ipe node]
     at (40, 528) {$=$};
  \node[ipe node]
     at (224, 528) {$=$};
  \node[ipe node]
     at (412, 528) {$=$};
\end{tikzpicture}
}
\]
where we use, among others, the adjoint commutativity of the diagram \eqref{eq: DeltaAdjointCommutativity}.

\end{proof}

Given $g\in G$, we continue denoting by $V_g: \Gamma\to \Gamma$ the algebra automorphism induced by $V_g: H_0\to H_0$. For an element $g\in G$ and a module $(H,\mu_H)\in \M$, we say that $(H,\mu_H)$ is a $g$-graded module if the diagram
\[\begin{tikzcd}
	{\Gamma\boxtimes^GH} && {\Gamma\boxtimes^GH} \\
	{\Gamma\boxtimes^GH} && H
	\arrow["{\beta_{H,\Gamma}\circ\beta_{\Gamma, H}}", from=1-1, to=1-3]
	\arrow["{V_{g^{-1}}\boxtimes^G \id}"', from=1-1, to=2-1]
	\arrow["{\mu_H}", from=1-3, to=2-3]
	\arrow["{\mu_H}"', from=2-1, to=2-3]
\end{tikzcd}\]
commutes. We write $\M_g$ for the full subcategory of $\M$ on the $g$-graded modules.
\begin{proposition}\label{Prop: ModGammaIsGGraded}The category $\M$ splits as a direct sum
        \[
        \M\cong\bigoplus_{g\in G}\M_g.
        \]
\end{proposition}
\begin{proof}
    Let $(H, \mu_H)\in\M$. For every $g\in G$, we define $\Pi_g: H\to H$ by
    \begin{align*}
    H\cong H_0^G\boxtimes^G H&\xrightarrow{\iota\boxtimes^G \id}\Gamma\boxtimes ^G H\\ &\xrightarrow{\Delta\boxtimes^G\id}\Gamma\boxtimes^G\Gamma\boxtimes^G H\\& \xrightarrow{\id\boxtimes^G (\beta_{H, \Gamma}\circ\beta_{\Gamma, H})} \Gamma\boxtimes^G\Gamma\boxtimes^G H\\&\xrightarrow{\id\boxtimes^GV_{g}\boxtimes^G\id}\Gamma\boxtimes^G\Gamma\boxtimes^G H\\ &\xrightarrow{\id\boxtimes\mu_H} \Gamma\boxtimes^G H\\ &\xrightarrow{\mu_H} H,
    \end{align*}
    and $\pi_g : = \frac{1}{|G|}\Pi_g$. Then, by the same arguments as \cite[Thm. 3.3]{MR4244264}, $\pi_g$ is a morphism in $\Rep(\A^G)$, the pair $(\pi_g(H),\mu_H|_{\pi_g(H)})$ is a well-defined $g$-graded $\Gamma$-module, and for all $g,h\in G$, it holds that $\pi_g\pi_h = \delta_{g,h}\pi_g$ and $\sum\limits_{g\in G}\pi_g = \id_{H}$. Note that the arguments in that reference can be restricted to the non-super case, and that the relevant hypotheses in \cite[Assumption 3.1]{MR4244264} are satisfied taking, in particular, $\varepsilon: \Gamma\to H_0^G$ to be $\varepsilon(x) = \frac{1}{|G|}\sum\limits_{g\in G} g(x)$ and $i = \Delta\circ \iota$. Hence, $(H, \mu_H) = \bigoplus\limits_{g\in G} (\pi_g(H), \mu_H|_{\pi_g(H)})$, with $(\pi_g(H), \mu_H|_{\pi_g(H)})\in \M_g$ for all $g\in G$. 
    
    It remains to show that, given distinct $g,h\in G$ and $(H, \mu_H)\in\M_g$ and $(K, \mu_K)\in \M_h$, it holds that
    \[
    \Hom_{\M}((H, \mu_H), (K, \mu_K)) = 0.
    \]
    This follows from analogous arguments to \cite[Cor. 3.4]{MR4244264}.
\end{proof}

\begin{remark}
    The following comments are relevant regarding our use of arguments~in~\cite{MR4244264}. 
    \begin{enumerate}
        \item As we have mentioned, the reference deals with $G$-crossed braided supercategories and superalgebras. However, the arguments can be straightforwardly adapted to the non-super case.
        \item McRae defines (in our notation) the product of two $\Gamma$-modules $(H, \mu_H),(K,\mu_K)$ as the coequalizer of $\Gamma\boxtimes^G H\boxtimes^G K\xrightarrow{\mu_H\boxtimes^G\id} H\boxtimes^G K$ and $\Gamma\boxtimes^G H\boxtimes^G K\xrightarrow{\beta_{\Gamma, H}\boxtimes^G\id}H\boxtimes^G\Gamma\boxtimes^G K\xrightarrow{\id\boxtimes^G \mu_K}H\boxtimes^G K$. This coequalizer is canonically equivalent to our $H\boxtimes_\Gamma K$ via the unitary $H\boxtimes^G\Gamma\boxtimes^G K\xrightarrow{\beta_{\Gamma, H}^{-1}\boxtimes^G \id}\Gamma\boxtimes^G H\boxtimes^G K$.
        \item McRae works in an ambient abelian category and assumes that the tensor functors $-\boxtimes^GH$ and $H\boxtimes^G-$ are right exact for all objects $H$. The category $\Rep(\A^G)$ is not abelian, but as we have discussed, $\Gamma$ being separable ensures that $H\boxtimes_\Gamma K$ exists and is a coequalizer. Also, in $\Rep(\A^G)$ all short exact sequences split and the relevant functors in this section will preserve these splittings.
        \item Whilst we work with unitary modules, \cite{MR4244264} does not deal with unitarity. However, the inclusion of the category of unitary $\Gamma$-modules into the category of all $\Gamma$-modules is fully faithful.
    \end{enumerate}
    Hence, we will use the arguments from \cite{MR4244264} without making further reference to these hypotheses.
\end{remark}

We can also define an action $\mathrm{T}$ of $G$ on $\M$. Given $g\in G$ and $(H,\mu_H)\in \M$, we define $\mathrm{T}_g(H, \mu_H) = (H, g^{-1}\mu_H)$, where
\[
g^{-1}\mu_H: \Gamma\boxtimes^G H\xrightarrow{V_{g^{-1}}\boxtimes^G\id} \Gamma\boxtimes ^G H\xrightarrow{\mu_H} H.
\]
The action of $G$ on morphisms is trivial, and it is easy to upgrade the functors $\mathrm{T}_g$ to an action of $G$ on $\M$ compatible with the $G$-grading, see \cite[Sec. 4.2.4]{braidedFC} or \cite[App. A]{MR4244264}. Finally, given modules $(H,\mu_H),(K,\mu_K)\in \M$ such that $(H,\mu_H)$ is homogeneous of degree $g$, the crossed braiding $\mathbb{B}^\Gamma_{(H,\mu_H), (K,\mu_K)}; H\boxtimes_\Gamma K\xrightarrow{\cong} T_g(K)\boxtimes_\Gamma H$ is the unique morphism making the following diagram commute,
\begin{equation}\label{eq: braidingModGamma}\begin{tikzcd}
	{H\boxtimes^G K} &&& {K\boxtimes^GH} \\
	{H\boxtimes_\Gamma K} &&& {\mathrm{T}_g(K)\boxtimes_\Gamma H}.
	\arrow["{\beta_{H,K}}", from=1-1, to=1-4]
	\arrow["{\mu_{H,K}}"', from=1-1, to=2-1]
	\arrow["{\mu_{\mathrm{T}_g(K), H}}", from=1-4, to=2-4]
	\arrow["{\mathbb{B}^\Gamma_{(H,\mu_H), (K,\mu_K)}}"', from=2-1, to=2-4]
\end{tikzcd}\end{equation}
\begin{lemma}\label{lemm: braidingModGamma}
    Fix $g\in G$ and let $(H, \mu_H)\in \M_g$ and $(K, \mu_K)\in \M$. Then, Equation \eqref{eq: braidingModGamma} defines a unique unitary morphism $\B^\Gamma_{(H, \mu_H), (K, \mu_K)}: H\boxtimes_\Gamma K\to \mathrm{T}_g(K)\boxtimes_\Gamma H$ in $\M$. In addition, $\B^\Gamma_{(H, \mu_H), (K, \mu_K)}$ is natural in both entries.
\end{lemma}
\begin{proof}
    All of the statement, except unitarity, follows from the arguments in \cite[App. A]{MR4244264}, in particular from Equations (A.4) and (A.5). Since $\beta_{H, K}$ is unitary, so is $\B^\Gamma_{(H, \mu_H), (K, \mu_K)}$.
\end{proof}

\begin{remark}\label{rk: CommentDeequivariantization}
    It is well-known that, given a faithful action of a finite group $G$ on a conformal net $\A$, there is an inclusion of braided tensor categories $\Rep\,G\hookrightarrow \Loc_{I_0}(\A^G)$ from the braided tensor category of unitary representations of $G$ into the braided tensor category of $I_0$-localized endomorphisms of $\A^G$ for any $I_0\in\Jcal_\mathrm{p}$, see \cite[Rk. 3.3]{muger05} and \cite[Sec. 2.1]{MR1867563}. This induces, via the equivalence $\mathfrak{E}: \Loc_{I_0}(\A^G)\to \Rep(\A^G)$, an inclusion of braided tensor categories $\Rep\,G\hookrightarrow\Rep(\A^G).$ Then, the $\mathrm{C}^*$-Frobenius algebra object $\Gamma$ is isomorphic to the image of the regular representation of $G$ under this inclusion (see \cite[Sec. 3]{MR258394}, \cite[Thm. 3.6]{MR1062748} and \cite[Sec. 2.1 and Thm. 2.8]{MR1867563} for partial statements). The category $\M$ is therefore the $G$-de-equivariantization of $\Rep(\A^G)$ under the inclusion $\Rep\,G\hookrightarrow \Rep(\A^G)$, meaning that it is a $G$-crossed braided tensor category whose $G$-equivariantization is $\Rep(\A^G)$. We do not prove these statements here, as the theory of equivariantizations and de-equivariantizations of braided non-fusion categories is not fully developed in the literature. Instead, we take a different approach similar to that in \cite{MR4244264} and using the theory of $G$-crossed categorical extensions from Section \ref{sec: CategoricalExtensions}.
\end{remark}

We will show that, with the tensor product $\boxtimes^\Gamma$, the grading $\M = \bigoplus_{g\in G}\M_g$, the action $\mathrm{T}$ and the unitaries $\B^\Gamma$, the $\mathrm{W}^*$-category $\M$ is a $G$-crossed braided $\mathrm{W}^*$-tensor category equivalent to $\Rep^G(\A)$. We construct first the underlying functor $\mathfrak{M}: \M\to \Rep^G(\A)$, inspired by \cite[Thm. 2.10]{Gui_2025}. We need the following lemma, which is analogous to \cite[Lem. 1.5]{Gui_2025} for our choice of braiding.

\begin{lemma}\label{lemm: L(xi, I-1)}
    Let $\mathcal{B}$ be a conformal net and let $H, K\in \Rep(\mathcal{B})$ be representations. Then, for any $\tilde{I}\in\Jcal_\R$ and any $\xi\in H(I)$, the following diagram commutes,
\[\begin{tikzcd}
	K && {H\boxtimes K} \\
	{H\boxtimes K} && {K\boxtimes  H},
	\arrow["{L(\xi, \tilde{I})}", from=1-1, to=1-3]
	\arrow["{L(\xi, \tilde{I}-1)}"', from=1-1, to=2-1]
	\arrow["{\B_{H, K}}", from=1-3, to=2-3]
	\arrow["{\B_{K, H}}", from=2-3, to=2-1]
\end{tikzcd}\]
where $L$ and $R$ are the operators in the Connes categorical extension of $\mathcal{B}.$
\end{lemma}
\begin{proof}
    We use the properties of the (non-crossed) Connes categorical extension of $\mathcal{B}$, see Definition \ref{def: CategoricalExtension} and Theorem \ref{Thm: ConnesCategoricalExtension} for $\A = \mathcal{B}$ and $G$ the trivial group. Let $\eta\in K((\tilde{I}-1)^{c+}) = K(\tilde{I}^{c-})$. Then, using locality and compatibility with the braiding of the operators $L$ and $R$,
    \begin{align*}
        L(\xi, \tilde{I}-1)\eta = R(\eta, \tilde{I}^{c-})\xi = \B_{K, H}\circ L(\eta, \tilde{I}^{c-})\xi = \B_{K, H}\circ R(\xi, \tilde{I})\eta = \B_{K, H}\circ\B_{H, K}\circ L(\xi, \tilde{I})\eta.
    \end{align*}
    Since $K(\tilde{I}^{c-})\subset K$ is dense, the claim follows.
\end{proof}

\begin{proposition}\label{Prop: ExtensionFunctor}
    Fix $g\in G$ and let $\big((H, \pi_H),\mu_H\big)\in \M_g$ be a $g$-graded module over $\Gamma$. Then, there exists a unique $g$-twisted representation $\pi^{\mathfrak{M}(H)}$ of $\A$ on $H$ such that, for all $\tilde{I}\in \Jcal_\R$ and all $x\in \A_\R(\tilde{I})$, it holds that
    \begin{equation}\label{eq: ExtensionDefinition}
    \pi^{\mathfrak{M}(H)}_{\tilde{I}}(x): H\xrightarrow{L^G(\pi_0(x)\Omega,\tilde{I})}\Gamma\boxtimes^G H\xrightarrow{\mu_H} H.  
    \end{equation}
\end{proposition} 
\begin{proof}
    Fix $g\in G$ and let $(H, \mu_H)\in \M_g$. Let $\tilde{I},\tilde{J}\in\Jcal_\R$ be intervals with $\tilde{J}\subset\tilde{I}^{c+}$, and $\xi\in \Gamma(I)$, $\eta\in H(J)$. Consider the following diagram,
\begin{equation}\label{eq: AdjointCommutativityConstruction FunctorM}\begin{tikzcd}
	\Gamma && {\Gamma\boxtimes^G H} && H \\
	{\Gamma\boxtimes^G\Gamma} && {\Gamma\boxtimes^G\Gamma\boxtimes^G H} && {\Gamma\boxtimes^G H} \\
	\Gamma && {\Gamma\boxtimes^G H} && H.
	\arrow["{R^G(\eta, \tilde{J})}", from=1-1, to=1-3]
	\arrow["{L^G(\xi, \tilde{I})}"', from=1-1, to=2-1]
	\arrow["{\mu_H}", from=1-3, to=1-5]
	\arrow["{L^G(\xi, I)}", from=1-3, to=2-3]
	\arrow["{L^G(\xi, I)}", from=1-5, to=2-5]
	\arrow["{R^G(\eta, \tilde{J})}", from=2-1, to=2-3]
	\arrow["\mu"', from=2-1, to=3-1]
	\arrow["{\id\boxtimes^G\mu_H}", from=2-3, to=2-5]
	\arrow["{\mu\boxtimes^G\id}", from=2-3, to=3-3]
	\arrow["{\mu_H}", from=2-5, to=3-5]
	\arrow["{R^G(\eta, \tilde{J})}", from=3-1, to=3-3]
	\arrow["{\mu_H}", from=3-3, to=3-5]
\end{tikzcd}
\end{equation}
    By locality and naturality of the Connes categorical extension of $\A^G$, together with the fact that $(F_1\boxtimes^G F_2)^* = F_1^*\boxtimes^G F_2^*$ and the module and Frobenius relations for $\mu_H$, each of the small diagrams commutes adjointly. Note that $\mu_H$ is surjective, since $H\cong H_0^G\boxtimes^G H\xrightarrow{\iota\boxtimes^G \id}\Gamma \boxtimes^G H\xrightarrow{\mu_H} H$ equals the identity morphism on $H$. Thus, if we choose $\eta$ so that $R(\eta, \tilde{J}): H_0^G\to \Gamma$ is unitary, then $R(\eta, \tilde{J}):\Gamma\to \Gamma\boxtimes^G H$ is unitary by \cite[Lem. 1.4]{Gui_2025} and $\mu_HR(\eta, \tilde{J})$ is surjective. Note that such $\eta$ always exists, as $\A^G(J^c)$ is a type $\mathrm{III}$ factor and $H_0^G$ and $\Gamma$ are nonzero modules over it. Hence, by \cite[Lem. 2.9]{Gui_2025}, there is a unique representation $\pi^{\mathfrak{M}(H)}_{\tilde{I}}$ of $\A_\R(\tilde{I})$ on $H$ such that Equation \eqref{eq: ExtensionDefinition} holds.

     Isotony is clear. Let us show that the constructed representation is $g$-twisted. Let $\tilde{I}\in\Jcal_\R$ and $x\in \A_\R(\tilde{I}) = \A(I)  = \A_\R(\tilde{I} - 1)$. Then, we need to show the commutativity of the outer diagram in 
\[\begin{tikzcd}
	H &&&&& {\Gamma\boxtimes^GH} \\
	&&& {\Gamma\boxtimes^GH} \\
	\\
	{\Gamma\boxtimes^G H} &&&&& H.
	\arrow["{L^G(\pi_0(x)\Omega, \tilde{I}-1)}", from=1-1, to=1-6]
	\arrow["{L^G(\pi_0(x)\Omega, \tilde{I})}"'{pos=0.8}, from=1-1, to=2-4]
	\arrow["{L^G(\pi_0(g^{-1}x)\Omega, \tilde{I})}"', from=1-1, to=4-1]
	\arrow["{\mu_H}", from=1-6, to=4-6]
	\arrow["{\beta_{H, \Gamma}\circ \beta_{\Gamma,H}}"'{pos=0.3}, from=2-4, to=1-6]
	\arrow["{V_{g^{-1}}\boxtimes^G \id}", from=2-4, to=4-1]
	\arrow["{\mu_H}"', from=4-1, to=4-6]
\end{tikzcd}\]
The top triangle commutes by Lemma \ref{lemm: L(xi, I-1)} for $\mathcal{B} = \A^G$, the left triangle commutes trivially (using that $V_{g}\Omega  = \Omega$) and the lower quadrilateral commutes by the hypothesis that $(H, \mu_H)$ is a $g$-graded module.
\end{proof}

By Proposition \ref{Prop: ExtensionFunctor}, we can define a $\mathrm{W}^*$-functor $$\mathfrak{M}: \M\to \Rep^G(\A)$$ by linearly extending the assignment obtained by sending a $g$-graded module $(H,\mu_H)\in \M_g$ to $\mathfrak{M}(H,\mu_H):=(H,\pi^{\mathfrak{M}(H)})\in \Rep^g(\A)$ and a morphism $F$ in $\M$ to the same map of Hilbert spaces. Indeed, fix $g\in G$ and let $F:(H, \mu_H)\to (K, \mu_K)$ be a morphism in $\M_g$. We show that $F$ intertwines the $g$-twisted actions $\pi^{\mathfrak{M}(H)}$ and $\pi^{\mathfrak{M}(K)}$ on $H$ and $K$ respectively. Let $\tilde{I}\in\Jcal_\R$, $x\in \A_\R(\tilde{I})$ and $\xi\in H$. Then, we can compute
\begin{align*}
    \pi_{\tilde{I}}^{\mathfrak{M}(K)}(x)\circ F(\xi) &= \mu_K\circ L^G(\pi_0(x)\Omega, \tilde{I}) F(\xi)\\ & =  \mu_K\circ (\id\boxtimes^G F)\circ L^G(\pi_0(x)\Omega, \tilde{I})\xi \\& = F\circ \mu_H\circ L^G(\pi_0(x)\Omega, \tilde{I})\\ & = F\circ\pi_{\tilde{I}}^{\mathfrak{M}(H)}(x)(\xi).
\end{align*}

\begin{proposition}\label{prop: MEquivalenceOfGCats}
    The functor $\mathfrak{M}: \M\to \Rep^G(\A)$ is an equivalence of $\mathrm{W}^*$-categories. In addition, it respects the $G$-grading and satisfies
    \[
\mathfrak{M}\circ \mathrm{T}_g
 =  T_g\circ \mathfrak{M}    \]
for all $g\in G$, making it into an equivalence of $G$-graded $\mathrm{W}^*$-categories with a $G$-action.
\end{proposition}
\begin{proof}
    Let us first show that $\mathfrak{M}$ is essentially surjective. Fix $g\in G$ and let $H^g = (H, \pi^H)\in \Rep^g(\A)$ be a $g$-twisted representation of $\A$. Let $r(H)\in \Rep(\A^G)$ denote the restriction of $H$ to an $\A^G$-representation. We next provide $r(H)$ with the structure of a left $\Gamma$-module. Fix $\tilde{I}\in\Jcal_\R$. We define the linear map
\begin{equation}\label{eq: InverseToM}
\begin{array}{cccc}
 \mu_{H, \tilde{I}}: &\Gamma\boxtimes^G r(H)   &\to & r(H)  \\
     & L^G(\pi_0(x)\Omega,\tilde{I})\eta&\mapsto &\pi_{\tilde{I}}^H(x)(\eta). 
\end{array}
\end{equation}
We note that $\mu_{H, \tilde{I}}$ is bounded by the same arguments as step 1 of the second part of the proof of \cite[Thm. 2.10]{Gui_2025}. If $\tilde{J}\in \Jcal_\R$ is an interval with $\tilde{J}\subset\tilde{I}$, then $\mu_{H, \tilde{I}} = \mu_{H, \tilde{J}}$. Hence, $\mu_{H, \tilde{I}}  = \mu_{H, \tilde{J}}$ for any pair of intervals $\tilde{I},\tilde{J}\in\Jcal_\R$. We denote $\mu_{H}:=\mu_{H,\tilde{I}}$ for any $\tilde{I}\in\Jcal_\R$. We next show that $\mu_H$ is a morphism of $\A^G$-representations. Let $\tilde{I}\in \Jcal_\R$ and $y\in \A^G(I)$. Then, for any $x\in \A(I)$ and any $\eta\in r(H)$, it holds that $\pi^{\Gamma\boxtimes^Gr(H)}_I(y)\circ L^G(\pi_0(x)\Omega, \tilde{I}) = L^G(\pi_0(yx)\Omega, \tilde{I})$ by \cite[Prop. 2.3]{MR4339084}, and hence
\begin{align*}
    \mu_{H}\circ \pi^{\Gamma\boxtimes^Gr(H)}_I(y)\circ L^G(\pi_0(x)\Omega, \tilde{I})\eta & = \mu_H\big(L^G(\pi_0(yx)\Omega, \tilde{I})\eta\big) \\& = \pi^H_{\tilde{I}}(yx)(\eta)\\& = \pi_{\tilde{I}}^H(y)\circ \mu_H\big(L^G(\pi_0(x)\Omega, \tilde{I})\eta\big).
\end{align*}
Let us show that $(H, \mu_H)$ is indeed a unitary module over $\Gamma$. To check unitality, we compute, for $\eta\in H$,
\[
\mu_H\circ (\iota\boxtimes^G\id_{H})\eta = \mu_H(L^G(\Omega, \tilde{I})\eta) = \pi^H_{\tilde{I}}(1)(\eta) = \eta
\]
for any $\tilde{I}\in\Jcal_\R$. We next check the associativity and Frobenius relation. Fix $\tilde{I}\in \Jcal_\R$ and let $x, y\in \A(I)$ and $\eta\in H$. Then, we compute
\begin{align*}
\mu_H\circ(\id\boxtimes^G \mu_H)(L^G(\pi_0(x)\Omega, \tilde{I})L^G(\pi_0(y)\Omega, \tilde{I})\eta)&= \mu_H(L^G(\pi_0(x)\Omega, \tilde{I})\pi_{\tilde{I}}^H(y)(\eta))\\ & = \pi_{\tilde{I}}^H(xy)(\eta),
\end{align*}
and using Equation \eqref{eq: JoinLs},
\begin{align*}
\mu_H\circ(\mu\boxtimes^G\id)(L^G(\pi_0(x)\Omega, \tilde{I})L^G(\pi_0(y)\Omega, \tilde{I})\eta) & = \mu_H\circ(\mu\boxtimes^G\id)(L^G(L^G(\pi_0(x)\Omega, \tilde{I})\pi_0(y)\Omega,\tilde{I})\eta)\\ & = \mu_H(L^G(\pi_0(xy)\Omega, \tilde{I})\eta)\\ & = \pi_{\tilde{I}}^H(xy)\eta,
\end{align*}
which shows the commutativity of \eqref{eq: ModuleCondition}. Recall the diagram \eqref{eq: AdjointCommutativityConstruction FunctorM}. We have just argued that the bottom-right inner diagram commutes, and we know that the other three commute adjointly. Hence, the outer diagram commutes. Namely, if $\tilde{J}\in\Jcal_\R$ is an interval such that $\tilde{J}\subset\tilde{I}^{c+}$ and $\eta\in H(J)$, we have $\mu_H\circ R^G(\eta, \tilde{J})\circ \pi_0(x) = \pi_{\tilde{I}}^H(x)\circ \mu_H\circ R^G(\eta, \tilde{J})$ as maps $H_0\to H$. It therefore holds that $\mu_H\circ R^G(\eta, \tilde{J})\circ \pi_0(x^*) = \pi_{\tilde{I}}^H(x^*)\circ \mu_H\circ R^G(\eta, \tilde{J})$ as maps $H_0\to H$, and since both $\pi_{0, I}$ and $\pi^H_{\tilde{I}}$ are $*$-actions, we find $\mu_H\circ R^G(\eta, \tilde{J})\circ \pi_0(x)^* = \pi_{\tilde{I}}^H(x)^*\circ \mu_H\circ R^G(\eta, \tilde{J})$ as maps $H_0\to H$. This is exactly the adjoint commutativity of the outer diagram in \eqref{eq: AdjointCommutativityConstruction FunctorM} for $\xi = \pi_0(x)\Omega$. By density of the fusion product, the bottom-right inner diagram in \eqref{eq: AdjointCommutativityConstruction FunctorM} commutes adjointly. This is the Frobenius relation for $\mu_H$.

It is clear that $\mathfrak{M}(r(H), \mu_H) = (H, \pi^H)$, hence proving essential surjectivity of $\mathfrak{M}$. 

To show that $\mathfrak{M}$ is fully faithful, we fix $H,K\in \Rep(\A^G)$ and $F\in \Hom_{\Rep(\A^G)}(H, K)$. For every $x\in \A_\R(\tilde{I})$, we have
\[
F\circ \pi^{\mathfrak{M}(H)}_{\tilde{I}}(x)\eta = F\circ\mu_H\circ L^G(\pi_0(x)\Omega, \tilde{I})\eta
\]
for any $\eta\in H$. Similarly, 
\[
\pi^{\mathfrak{M}(K)}_{\tilde{I}}(x)\circ F(\eta) = \mu_K\circ L^G(\pi_0(x)\Omega, \tilde{I})\circ F(\eta) = \mu_K\circ (\id_\Gamma\boxtimes F)\circ L^G(\pi_0(x)\Omega, \tilde{I})\eta.
\]
Therefore, the morphism $F$ intertwines the $\Gamma$-actions if and only if it intertwines the $\A_\R(\tilde{I})$-actions for all $\tilde{I}\in\Jcal_\R$. This shows that $$\Hom_{\M}((H, \mu_H), (K, \mu_K)) = \Hom_{\Rep^G(\A)}(\mathfrak{M}(H, \mu_H), \mathfrak{M}(K, \mu_K)).$$

That $\mathfrak{M}$ preserves the $G$-gradings follows from Proposition \ref{Prop: ExtensionFunctor}. 
For the compatibility with the $G$-action, we claim that $\mathfrak{M}\circ \mathrm{T}_g(H,\mu_H)  = T_g\circ\mathfrak{M}(H, \mu_H)$ for all modules $(H, \mu_H)\in \M$ and all $g\in G$. Indeed, this equality is equivalent to the commutativity of 
\[\begin{tikzcd}
	H && {\Gamma\boxtimes^GH} && {\Gamma\boxtimes^GH} \\
	\\
	{\Gamma\boxtimes^G H} &&&& H,
	\arrow["{L^G(\pi_0(x)\Omega, \tilde{I})}", from=1-1, to=1-3]
	\arrow["{L^G(\pi_0(g^{-1}x)\Omega, \tilde{I})}"', from=1-1, to=3-1]
	\arrow["{V_{g^{-1}}\boxtimes^G \id_H}", from=1-3, to=1-5]
	\arrow["{\mu_H}", from=1-5, to=3-5]
	\arrow["{\mu_H}"', from=3-1, to=3-5]
\end{tikzcd}\]
but since $L^G(\pi_0(g^{-1}x)\Omega, \tilde{I}) = L^G(V_{g^{-1}}\circ\pi_0(x)\Omega, \tilde{I}) = (V_{g^{-1}}\boxtimes^G\id_H)\circ L^G(\pi_0(x)\Omega, \tilde{I})$, the claim follows. Hence, the identity natural transformation provides the compatibility with the $G$-actions, meaning
\[
\mathfrak{M}\circ \mathrm{T}_g = T_g\circ\mathfrak{M}\]
for all $g\in G$.
\end{proof}

We will next show that $\M$ is naturally a $G$-crossed braided $\mathrm{W}^*$-tensor category equivalent to $\Rep^G(\A)$ via (an upgrade to a tensor functor of) the functor $\mathfrak{M}$. To do so, we first need the following lemma, which is analogous to \cite[Lem. 5.8]{Gui_2025}. Recall that we write $\iota: H_0^G\hookrightarrow \Gamma$ for the inclusion of $\A^G$-representations.

\begin{lemma}
    Let $(H, \mu_H)\in \M$. For any interval $\tilde{I}\in\Jcal_\R$, it holds that $\mathfrak{M}(H, \mu_H)_+(\tilde{I}) = H(I) = \mathfrak{M}(H, \mu_H)_-(\tilde{I})$.
\end{lemma}
\begin{proof}
    The fact that $\mathfrak{M}(H, \mu_H)_+(\tilde{I}) \subset H(I)$, follows directly from the equality of $\A^G$-representations $r\circ\mathfrak{M}(H) = H$. To prove this equality, consider $\tilde{I}\in\Jcal_\R$ and $x\in \A^G(I)$. Then, 
    \[\hspace{-.1cm}
    \pi^{\mathfrak{M}(H)}_{\tilde{I}} (x) =\mu_H\circ  L^G(\pi_0(x)\Omega, \tilde{I}) = \mu_H\circ \pi_{\tilde{I}}^{\Gamma\boxtimes^G H}(x)\circ L^G(\Omega, \tilde{I}) = \pi_{\tilde{I}}^H(x)\circ \mu_H\circ L^G(\Omega, \tilde{I}) = \pi^H_{\tilde{I}}(x),
    \]
    using \cite[Prop. 2.3]{MR4339084}.
    
   For the other inclusion, let $\eta\in H(I)$. We have that $\mu_HR^G(\eta,\tilde{I})\iota\Omega = \mu_H(\iota\boxtimes^G \id)R^G(\eta,\tilde{I})\Omega = \eta$. Moreover, $\mu_H R^G(\eta, \tilde{I}): H_0 = \Gamma\to \Gamma\boxtimes^G H\to H$ intertwines the actions $\pi_0$ and $\pi^{\mathfrak{M}(H)}_{\tilde{I}^{c+}}$ of $\A_\R(\tilde{I}^{c+})$, since for each $\xi\in \Gamma(I^c)$,
    \[
    \mu_HR^G(\eta, \tilde{I})\mu L^G(\xi, \tilde{I}^{c+}) = \mu_H L^G(\xi, \tilde{I}^{c+})\mu_H R^G(\eta, \tilde{I}) 
    \]
    when acting on $\Gamma$, by Diagram \eqref{eq: AdjointCommutativityConstruction FunctorM}. Hence, $\eta\in \mathfrak{M}(H, \mu_H)_+(\tilde{I})$ and $\mathfrak{M}(H, \mu_H)_+(\tilde{I}) = H(I)$. We can similarly prove that $\mathfrak{M}(H, \mu_H)_-(\tilde{I}) = H(I)$.
\end{proof}

We may now introduce the following notation. Let $H^g\in \Rep^g(\A)$ and $K^h\in \Rep^h(\A)$ be representations twisted by $g,h\in G$ respectively. Let $\mu_H: \Gamma\boxtimes^Gr(H)\to r(H)$ be the morphism $L^G(\pi_0(x)\Omega, \tilde{I})\xi\mapsto \pi_{\tilde{I}}^H(x)(\xi)$ defined in the proof of Proposition \ref{prop: MEquivalenceOfGCats} (see Equation \eqref{eq: InverseToM}). We similarly define $\mu_K$. Then, $(H, \mu_H),(K, \mu_K)\in \M$ and $\mathfrak{M}(H, \mu_H) = H$ and $\mathfrak{M}(K, \mu_K) = K$. Given $\tilde{I}\in \Jcal_\R$ and $\xi\in H^g_+(\tilde{I})$, we define the maps of Hilbert spaces
\begin{align*}
L^\Gamma(\xi, \tilde{I}):&\ K\xrightarrow{L^G(\xi, \tilde{I})} H\boxtimes ^G K \xrightarrow{\mu_{H, K}} H\boxtimes_\Gamma K\\
R^\Gamma(\xi, \tilde{I}):&\ K\xrightarrow{R^G(\xi, \tilde{I})} K\boxtimes ^G H \xrightarrow{\mu_{K, H}} K\boxtimes_\Gamma H.
\end{align*}

\begin{proposition}\label{prop: MisEquivalenceOFG-X-BraidedCats}
    Let $\A$ be a conformal net being acted on faithfully by a finite group $G$. Let $\Gamma\in \Rep(\A^G)$ be the vacuum representation $H_0$ of $\A$ seen as an $\A^G$-representation, with its canonical commutative Frobenius $\mathrm{C}^*$-algebra structure from Proposition \ref{prop: GammaFrobeniusAlgebra}. Let $\M$ be the $\mathrm{W}^*$-category of unitary left modules over $\Gamma$. Then,
    \begin{enumerate}
        \item the category $\M$, with the tensor structure $\boxtimes_\Gamma$, the grading $\M = \bigoplus_{g\in G}\M_g$, the $G$-action $\mathrm{T}$, and the unitaries $\B^\Gamma$, is a $G$-crossed braided $\mathrm{W}^*$-tensor category;
        \item the equivalence of $\mathrm{W^*}$-categories $\mathfrak{M}$ can be extended to an equivalence
    \[
    (\mathfrak{M}, \Phi): \M\to \Rep^G(\A)
    \]
    of $G$-crossed braided $\mathrm{W}^*$-tensor categories. The data witnessing the compatibility with the $G$-action is trivial. 
    \end{enumerate}
\end{proposition}
\begin{proof}
    By Proposition \ref{prop: MEquivalenceOfGCats}, we can identify $\M$ with $\Rep^G(\A)$ via $\mathfrak{M}$ as $G$-graded categories with $G$-action. By Theorem \ref{thm: UniquenessCategoricalExtensions}, it is enough to show that, under this identification, $(\boxtimes_\Gamma, \B^\Gamma, L^\Gamma, R^\Gamma)$ is a $G$-crossed categorical extension of $\A$. We check the hypotheses of Definition \ref{def: CategoricalExtension}. 
    
    Locality follows from the adjoint commutativity of the outer diagram in 
\[\begin{tikzcd}
	R && {R\boxtimes^G K} && {R\boxtimes_\Gamma K} \\
	{H\boxtimes^GR} && {H\boxtimes^GR\boxtimes^G K} && {H\boxtimes^G(R\boxtimes_\Gamma K)} \\
	{H\boxtimes_\Gamma R} && {(H\boxtimes_\Gamma R)\boxtimes^G K} && {H\boxtimes_\Gamma R\boxtimes_\Gamma K},
	\arrow["{R^G(\eta, \tilde{J})}", from=1-1, to=1-3]
	\arrow["{L^G(\xi, \tilde{I})}"', from=1-1, to=2-1]
	\arrow["{\mu_{R, K}}", from=1-3, to=1-5]
	\arrow["{L^G(\xi, \tilde{I})}", from=1-3, to=2-3]
	\arrow["{L^G(\xi, \tilde{I})}", from=1-5, to=2-5]
	\arrow["{R^G(\eta, \tilde{J})}", from=2-1, to=2-3]
	\arrow["{\mu_{H, R}}"', from=2-1, to=3-1]
	\arrow["{\id\boxtimes^G\mu_{R, K}}", from=2-3, to=2-5]
	\arrow["{\mu_{H, R}\boxtimes^G\id}", from=2-3, to=3-3]
	\arrow["{\mu_{H, R\boxtimes_\Gamma K}}", from=2-5, to=3-5]
	\arrow["{R^G(\eta, \tilde{J})}"', from=3-1, to=3-3]
	\arrow["{\mu_{H\boxtimes_\Gamma T, K}}"', from=3-3, to=3-5]
\end{tikzcd}\]
    where $\tilde{I}, \tilde{J}\in\Jcal_\R$ are intervals such that $\tilde{J}\subset \tilde{I}^{c+}$ and we take twisted representations $H^g\in \Rep^g(\A)$, $K^h\in \Rep^h(\A)$ and $R^k\in \Rep^k(\A)$, and vectors $\xi\in H^g_+(\tilde{I})$ and $\eta\in K^h_-(\tilde{J})$. The top-left diagram commutes adjointly by locality of the Connes categorical extension of $\A^G$, see Theorem \ref{Thm: ConnesCategoricalExtension} or \cite[Thm. 3.4]{Gui21}. The off-diagonal diagrams commute adjointly by naturality of $L^G(\xi, \tilde{I})$ and $R^G(\eta, \tilde{J})$. For the adjoint commutativity of the bottom-right diagram, see \cite[Diag. (5.5)]{Gui_2025} and \cite[Sec. 3.2 and 3.4]{MR4418715} and diagrams (3.10) and (3.11) in the last reference. Note that we can apply Gui's results by Lemma \ref{lemm: technicalLemmasMuHK}, which implies that \cite[Eqs. (5.2), (5.3)]{Gui_2025} hold.

    Using similar techniques to the proof of Lemma \ref{lemm: technicalLemmasMuHK}, it is easy to see that, after dropping the unitors, $\mu_{\Gamma, H} = \mu_H$ and $\mu_{H, \Gamma} = \mu_H\circ \beta^{-1}_{\Gamma, H}$. Hence, by letting either $H$ or $K$ be $\Gamma$, locality implies that $L^\Gamma(\xi, \tilde{I})\in \Hom_{\A_\R(\tilde{I}^{c+})}(R, H\boxtimes_\Gamma R)$ and $R^\Gamma(\eta, \tilde{J})\in \Hom_{\A_\R(\tilde{J}^{c-})}(R, R\boxtimes_\Gamma K)$.
    
    Compatibility with the $G$-grading has been argued in Proposition \ref{prop: MEquivalenceOfGCats}. Isotony of $L^\Gamma$ and $R^\Gamma$ follows from isotony of $L^G$ and $R^G$. Given $H^g\in \Rep^g(\A)$, $\tilde{I}\in\Jcal_\R$ and $\xi\in H^g_+(\tilde{I})$ and $\eta\in H^g_-(\tilde{I})$, naturality of $L^\Gamma(\xi, \tilde{I})$ and $R^\Gamma(\eta, \tilde{I})$ acting on $K^h\in \Rep^h(\A)$ follow from the naturality of $\mu_{H,K}$ together with the naturality of $L^G(\xi, \tilde{I})$ and $R^G(\eta, \tilde{I})$ respectively. Unitality of $L^\Gamma(\xi, \tilde{I})$ can be argued as follows,
    \[
    \mu_{H, \Gamma} L^G(\xi, \tilde{I})\Omega = \mu_H\beta_{ \Gamma, H}^{-1} L^G(\xi, \tilde{I})\Omega = \mu_H \beta_{\Gamma, H}^{-1} R^G(\Omega, \tilde{I}^{c+})\xi = \mu_H  L^G(\Omega, \tilde{I}^{c+})\xi = \mu_H R^G(\xi, \tilde{I})\Omega = \xi.
    \]
   Similarly, $R^{\Gamma}(\eta, \tilde{I})\Omega = \eta.$

    Both the Reeh-Schlieder property and density of fusion products follow from the fact that $\mu_{H,K}$ is surjective by Lemma \ref{lemm: technicalLemmasMuHK}. Finally, for the braiding property, let $\tilde{I}\in\Jcal_\R$ and $H^g\in\Rep^g(\A)$, $K^h\in\Rep^h(\A)$. Then, for every $\xi\in H^g_-(\tilde{I})$, we need to show the commutativity of the outer diagram in
\[\begin{tikzcd}
	{K^h} &&&& {T_g(H^g)\boxtimes_\Gamma K^h} \\
	\\
	&& {T_g(H^g)\boxtimes^G K^h} \\
	& {K^h\boxtimes^G H^g} & {K^h\boxtimes^G T_g(H^g)} \\
	\\
	{K^h\boxtimes_\Gamma H^g} && {T_g(K^h\boxtimes_\Gamma H^g)} && {T_g(K^h)\boxtimes_\Gamma T_g(H^g)}
	\arrow["{L^\Gamma(\Gamma_g\xi, \tilde{I})}", from=1-1, to=1-5]
	\arrow["{L^G(\Gamma_g\xi, \tilde{I})}", from=1-1, to=3-3]
	\arrow["{R^G(\xi, \tilde{I})}", from=1-1, to=4-2]
	\arrow["{R^\Gamma(\xi, \tilde{I})}"', from=1-1, to=6-1]
	\arrow["{\B^\Gamma_{T_g(H), K}}", from=1-5, to=6-5]
	\arrow["{\mu_{T_g(H), K}}", from=3-3, to=1-5]
	\arrow["{\beta_{T_g(H), K}}", from=3-3, to=4-3]
	\arrow["{\mu_{K, H}}", from=4-2, to=6-1]
	\arrow["\id", from=4-3, to=4-2]
	\arrow["{\mu_{K, T_g(H)}}", from=4-3, to=6-5]
	\arrow["{\Gamma_g}"', from=6-1, to=6-3]
	\arrow["\cong", from=6-5, to=6-3]
\end{tikzcd}\]
The top and left triangles commute by definition. The right quadrilateral also commutes by definition. The inner-most quadrilateral commutes by the fact that the Connes categorical extension of $\A^G$ is indeed a categorical extension, see \cite[Thm. 3.4]{Gui21} or \cite[Prop. 3.26]{GcrossedbraidedRep}. The bottom diagram commutes trivially.

Hence, we have argued that $(\boxtimes_\Gamma, \B^\Gamma, L^\Gamma, R^\Gamma)$ is a $G$-crossed categorical extension of $\A$. By Theorem \ref{thm: UniquenessCategoricalExtensions}, $\M$ is a $G$-crossed braided $\mathrm{W}^*$-tensor category and there exists a unitary family of isomorphisms $\Phi_{H, K}:H^g\boxtimes K^h\xrightarrow{\cong} H^g\boxtimes_\Gamma K^h$ inducing (together with $\mathfrak{M}$) an equivalence of $G$-crossed braided $\mathrm{W}^*$-tensor categories $\M\cong \Rep^G(\A)$.
\end{proof}

Let $\M^G$ denote the $G$-equivariantization of $\M$, and let us write
\[
\mathfrak{D}: \Rep(\A^G)\to \M^G
\]
for the functor sending $H\in \Rep(\A^G)$ to the unitary $\Gamma$-module $(\Gamma\boxtimes^G H, \mu\boxtimes^G\id_H)\in \M$ together with the family of isomorphisms $u_g = V_{g}\boxtimes^G\id_H$ for all $g\in G$. For a morphism $F\in \Hom_{\Rep(\A^G)}(H, K)$, we write $\mathfrak{D}(F):=\id_\Gamma\boxtimes^G F\in \Hom_{\M}(\mathfrak{D}(H), \mathfrak{D}(K))$, which is trivially compatible with the $G$-equivariant data. We also write $\mathfrak{M}^G: \M^G\to (\Rep^G(\A))^G$ for the equivalence of braided $\mathrm{W}^*$-tensor categories induced by $\mathfrak{M}.$ We therefore obtain a $\mathrm{W}^*$-functor
    \begin{equation}\label{eq: InductionFunctor}
        \Rep(\A^G)\xrightarrow{\mathfrak{D}}\M^G\xrightarrow{\mathfrak{M}^G} (\Rep^G(\A))^G,
    \end{equation}
which we will show produces an equivalence of braided $\mathrm{W}^*$-tensor categories. Given $((H, \pi_H), u)\in (\Rep^G(\A))^G$, we write $H^G : = \{\xi\in H\,|\, u_g\xi  = \xi\ \text{ for all $g \in G$}\}$ for the fixed-points Hilbert space of $H$ under $u: G\to U(H)$. Note that $\pi^H$ restricts to an honest action of $\A^G$ on $H^G$, which we continue denoting by $\pi_H$. In addition, if $F: (H, u)\to (K, v)$ is a morphism in $(\Rep^G(\A))^G$, it restricts to a morphism $F|_{H^G}:H^G\to K^G$ in $\Rep(\A^G)$. We~write
\[
\mathfrak{R}: (\Rep^G(\A))^G\to \Rep(\A^G)
\]
for the $\mathrm{W}^*$-functor sending $((H, \pi^H), u)\in (\Rep^G(\A))^G$ to $(H^G, \pi^H)$ and $F: (H, u)\to (K, v) $ to $F|_{H^G}:H^G\to K^G$.

\begin{remark}\label{rk: NonHomogeneous}
    The underlying object in $\Rep^G(\A)$ of $((H, \pi_H), u)$ will usually be a direct sum of homogeneous objects each belonging to a different $G$-component, even if $((H, \pi_H), u)$ is simple in $(\Rep^G(\A))^G$. However, we still use the notation $(H, \pi^H)$, understanding that the Hilbert space $H$ splits as a direct sum of subspaces, on each of which $\pi^H$ restricts to a twisted representation, possibly with a different twist $g\in G$ for every subspace.
\end{remark}

We claim that $\mathfrak{R}$ produces an inverse functor to the functor $\mathfrak{M}^G\circ\mathfrak{D}$ in Equation \eqref{eq: InductionFunctor}. Before proving this, we need some preliminaries.

Recall that, given $I_0\in \Jcal_{\mathrm{p}}$, by Theorem \ref{Thm: RepIsMugerForG}, there is an equivalence of $G$-crossed braided $\mathrm{W}^*$-tensor categories $\mathfrak{E}: G-\Loc_{I_0}(\A)\to \Rep^G(\A)$, which induces an equivalence between the equivariantizations $\mathfrak{E}^G: (G-\Loc_{I_0}(\A))^G\to (\Rep^G(\A))^G$. Let us denote by $(G-\Loc_{I_0}(\A))^G_0$ the full subcategory of $(G-\Loc_{I_0}(\A))^G$ on objects whose $G$-equivariance data is the identity. The inclusion $(G-\Loc_{I_0}(\A))^G_0\hookrightarrow (G-\Loc_{I_0}(\A))^G$ is an equivalence, by \cite[Rk. 3.7.2]{muger05}. We continue writing $\mathfrak{E}^G$ for the equivalence 
\[
\mathfrak{E}^G: (G-\Loc_{I_0}(\A))^G_0\hookrightarrow (G-\Loc_{I_0}(\A))^G\xrightarrow{\mathfrak{E}^G}(\Rep^G(\A))^G.
\]
Nonzero objects of $(G-\Loc_{I_0}(\A))^G_0$ can be constructed as follows. Fix $g\in G$ and denote by $Z_G(g)$ the centralizer of $g$ in $G$ and by $C_g$ the conjugacy class of $g$. For every $h\in C_g$, we fix an element $h_g\in G$ conjugating $g$ to $h$. Fix $\rho_g\in (g-\Loc_{I_0}(\A))^{Z_G(g)}_0$, that is, $\rho_g$ is a $g$-localized-in-$I_0$ endomorphism of $\A$ which is fixed by the action of $Z_G(g)$ on the nose. We write $\rho_h:=\gamma_{h_g}(\rho_g)\in (h-\Loc_{I_0}(\A))^{Z_G(h)}_0$ for $h\in C_g$. Note that, since $\rho_g$ is $Z_G(g)$-invariant, $\rho_h$ is independent of the chosen $h_g$. Then, we can construct $\rho: = \bigoplus\limits_{h\in C_g}\rho_h\in (G-\Loc_{I_0}(\A))^G_0$, and all objects of $(G-\Loc_{I_0}(\A))^G_0$ are direct sums of objects constructed as such. The image of $\rho$ under $\mathfrak{E}^G$ has underlying $G$-twisted representation $\bigoplus\limits_{h\in C_g}(H_0, \pi_0\circ \rho_h)$ and the following $G$-equivariant data $u$. Given $k,h\in G$ and $\xi\in (H_0, \pi_0\circ \rho_h)$, then $u_k\xi : = V_k\xi\in (H_0, \pi_0\circ \rho_{khk^{-1}}).$

\begin{lemma}\label{eq: Rdoesnotkill}
    Let $((H, \pi^H), u)\in (\Rep^G(\A))^G$ be a nonzero object. Then, $\mathfrak{R}((H, \pi^H), u)\in \Rep(\A^G)$ is nonzero.
\end{lemma}
\begin{proof}
    Let $I_0\in \Jcal_\mathrm{p}$ be an interval. By the discussion immediately above the lemma, it is enough to show the statement when $((H, \pi^H), u) = \mathfrak{E}^G(\oplus_{h\in C_g}\rho_h)$ for some $g\in G$ and $\rho_g\in (g-\Loc_{I_0}(\A))^{Z_G(g)}_0$. In this case, the underlying Hilbert space of $\mathfrak{R}\circ\mathfrak{E}^G(\oplus_{h\in C_g}\rho_h)$ is $\big(\bigoplus_{h\in C_g} H_0\big)^G$, where the action of $k\in G$ sends $\xi\in H_0$ in the $h$-component to $V_k\xi$ in the $(khk^{-1})$-component. Hence, the map
    \[
    \begin{array}{ccc}
        H_0^{Z_G(g)} &\to & \big(\bigoplus_{h\in C_g} H_0\big)^G  \\
        \xi &\mapsto & (V_{h_g}\xi )_h 
    \end{array}
    \]
    is an isomorphism. In particular, $\mathfrak{R}\circ\mathfrak{E}^G(\oplus_{h\in C_g}\rho_h)$ is nonzero, as $\mathbb{C}\cdot \Omega\subset H_0^{Z_G(g)}$.
\end{proof}

We also need the following relation between $\A$ and $\A^G$. Let $I\in \Jcal$ and fix $x\in \A(I)$. Then, the vector $\pi_0(x)\Omega\in r(H_0)$ is trivially $I$-bounded as witnessed by the $\A^G(I^c)$-equivariant map $Z^G(\pi_0(x)\Omega,\tilde{I}) = \pi_0(x): H^G_0\to H_0$. Given another $y\in\A(I)$, we can consider the composition
\[
H_0^G\xrightarrow{Z^G(\pi_0(x)\Omega, \tilde{I})}H_0\xrightarrow{Z^G(\pi_0(y)\Omega, \tilde{I})^*}H_0^G,
\]
which is $\A^G(I^c)$-equivariant, hence an element of $\A^G(I)$ by Haag duality.

\begin{lemma}\label{lemm: ProjectOntoAG}
    For any $I\in \Jcal$ and $x,y\in \A(I)$, we have
    \[
    Z^G(\pi_0(y)\Omega, \tilde{I})^*Z^G(\pi_0(x)\Omega, \tilde{I}) = \pi_0^G\Big(\frac{1}{|G|}\sum\limits_{g\in G} g(y^*x)\Big)\in B(H_0^G).
    \]
\end{lemma}
\begin{proof}
    It is clear that $\frac{1}{|G|}\sum\limits_{g\in G} g(y^*x)\in \A^G(I)$. Let $\xi, \xi'\in H_0^G$. We compute
    \begin{align*}
        \langle\pi_0^G\Big(\frac{1}{|G|}\sum\limits_{g\in G}g(y^*x)\Big)(\xi), \xi'\rangle & =  \langle \pi_0\Big(\frac{1}{|G|}\sum\limits_{g\in G}g(y^*x)\Big)(\xi), \xi'\rangle \\ & = \frac{1}{|G|}\sum\limits_{g\in G} \langle \pi_0(g(y^*x))(\xi), \xi'\rangle \\ &=\frac{1}{|G|}\sum_{g\in G}\langle V_g \pi_0(y^*)\pi_0(x)(\xi), V_g\xi'\rangle\\ & = \frac{1}{|G|}\sum_{g\in G}\langle \pi_0(y^*)\pi_0(x)(\xi), \xi'\rangle \\ & = \langle  \pi_0(x)(\xi), \pi_0(y)(\xi')\rangle \\& = \langle Z^G(\pi_0(x)\Omega,\tilde{I} )\xi, Z^G(\pi_0(y)\Omega,\tilde{I} )\xi'\rangle.
    \end{align*}
\end{proof}

We can now prove the main result of this section.

\begin{theorem}\label{thm: RepFixedPointsIsEquivTwistedReps}
    Let $\A$ be a conformal net being acted on faithfully by a finite group $G$. Then, the functors
\[\begin{tikzcd}
	{\Rep(\A^G)} && {(\Rep^G(\A))^G}
	\arrow["{\mathfrak{M}^G\circ\mathfrak{D}}", shift left=2, from=1-1, to=1-3]
	\arrow["{\mathfrak{R}}", shift left=2, from=1-3, to=1-1]
\end{tikzcd}\]
provide an equivalence of braided $\mathrm{W}^*$-tensor categories.
\end{theorem}
\begin{proof}
   By the same arguments as \cite[Thm. 2.11]{MR4244264} the functor $\mathfrak{D}: \Rep(\A^G)\to \M^G$ is a functor of braided $\mathrm{W}^*$-tensor categories when equipped with the tensor data
   \begin{align*}
   \Psi_{H, K}^{-1}: \Gamma\boxtimes^G H\boxtimes^G K\cong \Gamma\boxtimes^G H\boxtimes^G H_0^G\boxtimes^G K&\xrightarrow{\id\boxtimes^G\iota\boxtimes^G\id} \Gamma\boxtimes^G H\boxtimes^G\Gamma\boxtimes^G K\\&\xrightarrow{\mu_{\Gamma\boxtimes^G H, \Gamma\boxtimes^G K}}(\Gamma\boxtimes^G H)\boxtimes_\Gamma(\Gamma\boxtimes^G K).
   \end{align*}
Hence, the composition $\Rep(\A^G)\xrightarrow{\mathfrak{D}}\M^G\xrightarrow{\mathfrak{M}^G}(\Rep^G(\A))^G$ is a braided $\mathrm{W}^*$-tensor functor. It remains to show that $\mathfrak{R}$ is an inverse to $\mathfrak{M}^G\circ \mathfrak{D}$.

Let $(H, \pi_H)\in \Rep(\A^G)$. Then, $(\Gamma\boxtimes^G H)^G \cong \Gamma^G\boxtimes ^G H = H_0^G\boxtimes^G H\xrightarrow{\text{unitor}} H$ provides a unitary isomorphism $\mathfrak{R}\circ \mathfrak{M}^G\circ\mathfrak{D}(H)\cong H$ which is natural in $H$. Hence, we have $\mathfrak{R}\circ \mathfrak{M}^G\circ\mathfrak{D}\cong \text{Id}_{\Rep(\A^G)}$. 

We show next that $\mathfrak{M}^G\circ \mathfrak{D}\circ \mathfrak{R}\cong \text{Id}_{(\Rep^G(\A))^G}$. Let $I_0\in \Jcal_\mathrm{p}$. Since $\mathfrak{E}^G: (G-\Loc_{I_0}(\A))^G_0\to (\Rep^G(\A))^G$ is an equivalence, it is enough to construct, for every $\sigma\in (G-\Loc_{I_0}(\A))^G_0$, a unitary equivalence $\mathfrak{M}^G\circ \mathfrak{D}\circ \mathfrak{R}\circ \mathfrak{E}^G(\sigma)\cong \mathfrak{E}^G(\sigma)$ natural in $\sigma$. Fix $g\in G$ and $\rho_g\in (g-\Loc_{I_0}(\A))^{Z_G(g)}_0$. For every $h\in C_g$, let $h_g\in G$ be an element conjugating $g$ to $h$, and write $\rho_h:=\gamma_{h_g}(\rho_g)\in (h-\Loc_{I_0}(\A))^{Z_G(h)}_0$. We define $\rho:=\bigoplus\limits_{h\in C_g} \rho_h\in (G-\Loc_{I_0}(\A))^G_0$. Any object of $(G-\Loc_{I_0}(\A))^G_0$ is isomorphic to a direct sum of objects constructed as $\rho$. The underlying Hilbert space of $\mathfrak{M}\circ\mathfrak{D}\circ \mathfrak{R}\circ\mathfrak{E}^G(\rho)\in (\Rep^G(\A))^G$ is
\[
\Gamma\boxtimes^G (\bigoplus\limits_{h\in C_g} H_0)^G
\]
with $G$-equivariance data $\{V_g\boxtimes^G\id\}_{g\in G}$ and action of $x\in \A_\R(\tilde{I})$ given by
\[
\Gamma\boxtimes^G (\bigoplus\limits_{h\in C_g} H_0)^G\xrightarrow{L^G(\pi_0(x)\Omega, \tilde{I})}\Gamma\boxtimes^G \Gamma\boxtimes^G (\bigoplus\limits_{h\in C_g} H_0)^G\xrightarrow{\mu\boxtimes^G\id}\Gamma\boxtimes^G (\bigoplus\limits_{h\in C_g} H_0)^G
\]
for all $\tilde{I}\in \Jcal_\R$. Fix $\tilde{I}\in\Jcal_\R$. We claim that the linear map
\[
\begin{array}{cccc}
    \Phi_{\rho, \tilde{I}}: & \Gamma\boxtimes^G (\bigoplus\limits_{h\in C_g} H_0)^G &\to & \bigoplus\limits_{h\in C_g} H_0 \\
     & L^G(\pi_0(x)\Omega, \tilde{I})(\xi_h)_h&\mapsto &\pi^{\mathfrak{E}^G(\rho)}_{\tilde{I}}(x)(\xi_h)_h = (\pi_0(h^{\epsilon(\tilde{I})}\rho_hx)\xi_h)_h 
\end{array}
\]
provides a unitary equivalence in $(\Rep^G(\A))^G$ between $\mathfrak{M}^G\circ\mathfrak{D}\circ\mathfrak{R}\circ\mathfrak{E}^G(\rho)$ and $\mathfrak{E}^G(\rho)$. The map is bounded by the same arguments as the map in Equation \eqref{eq: InverseToM}. Similarly, by the same arguments as those below Equation \eqref{eq: InverseToM}, $\Phi_{\rho, \tilde{I}}$ is independent of $\tilde{I}$, and we write $\Phi_{\rho} := \Phi_{\rho, \tilde{I}}$ (cf. the proof of Proposition \ref{prop: MEquivalenceOfGCats}). We first argue that $\Phi_\rho$ is $G$-equivariant. Let $(\xi_h)_h\in (\bigoplus\limits_{h\in C_g} H_0)^G$. This means that, for every $h\in C_g$ and any element $k\in G$, it holds that $V_k\xi_h = \xi_{khk^{-1}}$. Then, for any $k\in G$ and $x\in \A(I)$ we have
\begin{align*}
\Phi_\rho\big((V_k\boxtimes^G\id)\circ L^G(\pi_0(x)\Omega,\tilde{I})(\xi_h)_h\big)
& = \Phi_\rho\big(L^G(\pi_0(kx)\Omega,\tilde{I})(\xi_h)\big)\\
& = \pi^{\mathfrak{E}^G(\rho)}_{\tilde{I}}(kx)(\xi_h)_h\\
& = (\pi_0(h^{\epsilon(\tilde{I})}\rho_h kx)\xi_h)_h\\
& = (\pi_0(kh^{\epsilon(\tilde{I})}k^{-1}\rho_{khk^{-1}}kx)\xi_{khk^{-1}})_{khk^{-1}}\\
& = (\pi_0(kh^{\epsilon(\tilde{I})}\rho_h x)V_k\xi_h)_{khk^{-1}}\\
& = V_k(\pi_0(h^{\epsilon(\tilde{I})}\rho_h x)\xi_h)_h\\
& = V_k\circ \Phi_\rho\big(L^G(\pi_0(x)\Omega,\tilde{I})(\xi_h)_h\big),
\end{align*}
as needed. We next show that $\Phi_\rho$ is a morphism of twisted $\A$-representations. Let $\tilde{I}\in\Jcal_\R$ be an interval and fix $x,y\in \A_\R(\tilde{I})$. Then, denoting $\xi:=(\xi_h)_h$,
\begin{align*}
\Phi_\rho\big(\pi_{\tilde{I}}^{\mathfrak{M}(\mathfrak{D}\circ\mathfrak{R}\circ\mathfrak{E}^G(\rho))}(y)\circ L^G(\pi_0(x)\Omega, \tilde{I})\xi\big)& = \Phi_\rho\big((\mu\boxtimes^G\id)\circ L^G(\pi_0(y)\Omega, \tilde{I}) L^G(\pi_0(x)\Omega, \tilde{I})\xi\big)\\ & = \Phi_\rho\big((\mu\boxtimes^G\id)\circ L^G(L^G(\pi_0(y)\Omega, \tilde{I})\pi_0(x)\Omega, \tilde{I})\xi\big)\\ & = \Phi_\rho\big(L^G(\mu L^G(\pi_0(y)\Omega, \tilde{I})\pi_0(x)\Omega, \tilde{I})\xi\big)
\\& = \Phi_{\rho}(L^G(\pi_0(yx)\Omega, \tilde{I}),\xi)
\\ & = \pi_{\tilde{I}}^{\mathfrak{E}^G(\rho)}(yx)\xi\\ & = \pi_{\tilde{I}}^{\mathfrak{E}^G(\rho)}(y)\circ \Phi_\rho\big(L^G(\pi_0(x)\Omega, \tilde{I})\xi\big),
\end{align*}
     where we have used Equation \eqref{eq: JoinLs}. It remains to show that $\Phi_\rho$ is an isometry with dense image. Let us first argue it is an isometry. Let $\xi = (\xi_h)_h,\ \xi' = (\xi'_h)_h\in (\bigoplus\limits_{h\in C_g} H_0)^G$ and $x,x'\in \A(I)$. Then, using Lemma \ref{lemm: ProjectOntoAG},
    \begin{align*}
         \langle L^G(\pi_0(x)\Omega, \tilde{I})\xi,  L^G(\pi_0(x')\Omega, \tilde{I})\xi' \rangle &= \langle \pi_0(x)\Omega\otimes \xi,  \pi_0(x')\Omega\otimes\xi'\rangle \\ & = \langle\pi_{\tilde{I}}^{\mathfrak{E}^G}(Z^G(\pi_0(x')\Omega, \tilde{I})^*Z^G(\pi_0(x)\Omega, \tilde{I}))(\xi), \xi'\rangle\\ & = \frac{1}{|G|}\sum\limits_{k\in G}\langle \pi_{\tilde{I}}^{\mathfrak{E}^G(\rho)}(k(x'^*x))(\xi), \xi'\rangle \\ & = \frac{1}{|G|}\sum\limits_{k\in G}\langle \pi_{\tilde{I}}^{\mathfrak{E}^G(\rho)}(kx)(\xi), \pi_{\tilde{I}}^{\mathfrak{E}^G(\rho)}(kx')(\xi')\rangle\\ & = \frac{1}{|G|}\sum\limits_{k\in G}\sum\limits_{h\in C_g}\langle \pi_0(h^{\epsilon(\tilde{I})} \rho_h k (x))(\xi_h), \pi_0(h^{\epsilon(\tilde{I})} \rho_h k (x'))(\xi_{h}')\rangle.
     \end{align*}
     Note that $\pi_0(h^{\epsilon(\tilde{I})} \rho_h k (x))(\xi_h) = \pi_0(h^{\epsilon(\tilde{I})} k\rho_{k^{-1}hk}  (x))V_k(\xi_{k^{-1}hk}) = V_k\circ \pi_0((k^{-1}hk)^{\epsilon(\tilde{I})}\rho_{k^{-1}hk}(x))(\xi_{k^{-1}hk})$. Therefore, using that the $V_k$'s are unitary, the sum above equals
     \begin{align*}
          \langle L^G(\pi_0(x)\Omega,& \tilde{I})\xi,  L^G(\pi_0(x')\Omega, \tilde{I})\xi' \rangle \\&= \frac{1}{|G|}\sum\limits_{k\in G}\sum\limits_{h\in C_g}\langle \pi_0(h^{\epsilon(\tilde{I})} \rho_h k (x))(\xi_h), \pi_0(h^{\epsilon(\tilde{I})} \rho_h k (x'))(\xi_{h}')\rangle\\ & = \frac{1}{|G|}\sum_{k\in G}\sum_{h\in C_g}\langle \pi_0((k^{-1}hk)^{\epsilon(\tilde{I})}\rho_{k^{-1}hk}(x))(\xi_{k^{-1}hk}), \pi_0((k^{-1}hk)^{\epsilon(\tilde{I})}\rho_{k^{-1}hk}(x'))(\xi'_{k^{-1}hk})\rangle  \\ & = \frac{1}{|G|}\sum\limits_{k\in G}\langle \pi_{\tilde{I}}^{\mathfrak{E}^G(\rho)}(x)(\xi), \pi_{\tilde{I}}^{\mathfrak{E}^G(\rho)}(x')(\xi')\rangle \\ & = \langle \pi_{\tilde{I}}^{\mathfrak{E}^G(\rho)}(x)(\xi), \pi_{\tilde{I}}^{\mathfrak{E}^G(\rho)}(x')(\xi')\rangle\\ & = \langle \Phi_{\rho}\big(L^G(\pi_0(x)\Omega, \tilde{I})\xi\big), \Phi_{\rho}\big( L^G(\pi_0(x')\Omega, \tilde{I})\xi'\big)\rangle,
     \end{align*}
     and $\Phi_\rho$ is an isometry. It remains to show that $\Phi_\rho$ has dense image. 
     
     The following argument is extracted from the proof of \cite[Thm. 4.17]{MR4244264}.  Let us denote for the rest of the proof $H_\rho:=\mathfrak{E}^G(\rho)$. Let $\hat{G}$ be a set whose elements are a choice of irreducible representation of $G$ for every isomorphism class. We may break, as a $G$-representation, $H_0\cong \bigoplus\limits_{\alpha\in \hat{G}} H_\alpha$, where $H_\alpha$ is the direct sum of the subspaces of $H_0$ transforming as the irreducible representation $\alpha\in\hat{G}$. Note that, if $\1\in \hat{G}$ denotes the trivial representation, $H_{\1} = H^G_0$. Then, as a $G$-representation,
     \[
     \Gamma\boxtimes^G H_\rho^G = \bigoplus\limits_{\alpha\in \hat{G}}(H_\alpha\boxtimes^G H_\rho^G).
     \]
     Similarly, $H_\rho$ can be split as a $G$-representation as $H_\rho=\bigoplus\limits_{\alpha\in \hat{G}}(H_\rho)_\alpha$, where $(H_\rho)_\alpha$ is the sum of all the subspaces of $H_\rho$ transforming under the representation $\alpha$ for every $\alpha\in \hat{G}$, with $(H_\rho)_{\1} = H_\rho^G$. Since we have proven that $\Phi_\rho: \Gamma\boxtimes^G H_\rho^G\to H_\rho$ is $G$-equivariant, it respects the grading given by $\hat{G}$. In particular, the restriction of $\Phi_\rho$ to the trivial component $H_0^G\boxtimes^G H_\rho^G\to H_\rho^G$ is a unitary isomorphism, by inspecting the definition of $\Phi_\rho$. Therefore, the cokernel of $\Phi_\rho$ is an object of $(\Rep^G(\A))^G$ whose $G$-invariant subspace is zero. By Lemma \ref{eq: Rdoesnotkill}, we deduce that the cokernel of $\Phi_\rho$ is zero. Hence, $\Phi_\rho$ is a unitary. Since $\Phi_\rho$ is clearly natural in $\mathfrak{E}^G(\rho)$, the claim follows.
\end{proof}

\subsection{Compatibility with the balance}
We finish this paper by upgrading the equivalence of braided $\mathrm{W}^*$-tensor categories in Theorem \ref{thm: RepFixedPointsIsEquivTwistedReps} to an equivalence of balanced $\mathrm{W}^*$-tensor categories $\Rep(\A^G)\cong (\Rep^G(\A))^G$. We denote by $\vartheta$ the canonical balance on $\Rep(\A^G)$, so that we reserve the notation $\theta$ for the $G$-crossed balance on $\Rep^G(\A)$ and $\theta^G$ for the balance on $(\Rep^G(\A))^G$. Recall that, on a representation $H\in\Rep(\A^G)$, the unitary $\vartheta_H:H\to H$ is given by the action $e^{-2\pi i L_0}$ of $\widetilde{\exp}(-2\pi i L_0)\in\widetilde{\Mob}$. 

\begin{theorem}\label{Thm: EquivariantizationCompatibleBalance} 
    Let $\A$ be a conformal net being acted on faithfully by a finite group $G$. Then, the functors
\[\begin{tikzcd}
	{\Rep(\A^G)} && {(\Rep^G(\A))^G}
	\arrow["{\mathfrak{M}^G\circ\mathfrak{D}}", shift left=2, from=1-1, to=1-3]
	\arrow["{\mathfrak{R}}", shift left=2, from=1-3, to=1-1]
\end{tikzcd}\]
provide an equivalence of balanced $\mathrm{W}^*$-tensor categories.
\end{theorem}
\begin{proof}
    It is enough to show that $\mathfrak{R}: (\Rep^G(\A))^G\to \Rep(\A^G)$ preserves the balance. Let $X:=(\bigoplus\limits_{g\in G}H^g, u)\in (\Rep^G(\A))^G$ be an object, with $H^g\in\Rep^g(\A)$ a (possibly zero) $g$-twisted representation. Then, for every $g,k\in  G$, the unitary $u_k: \bigoplus\limits_{h\in G}T_k(H^h)\to \bigoplus\limits_{h\in G}H^h$ restricts to a unitary
    \[
    u_k^g: T_k(H^g)\to H^{gkg^{-1}}.
    \]
 The balance $\theta^G_{X}$ on $(\bigoplus\limits_{g\in G}H^g, u)$ is
    \[
    \bigoplus\limits_{g\in G}H^g\xrightarrow{\bigoplus_g e^{-2\pi i L_0}} \bigoplus\limits_{g\in G}T_g(H^g)\xrightarrow{\bigoplus_g u_g^g} \bigoplus\limits_{g\in G}H^g,
    \]
    where we abuse the notation to denote by $e^{-2\pi i L_0}$ the action of $\widetilde{\exp}(-2\pi i L_0)\in\widetilde{\Mob}$ on each of the $H^g$. Now, we have
    \[
    \mathfrak{R}(\theta^G_{X}) = \mathfrak{R}(\bigoplus_{g\in G}u_g^g)\circ \mathfrak{R}(\bigoplus_{g\in G} e^{-2\pi i L_0}) = \mathfrak{R}(\bigoplus_{g\in G} e^{-2\pi i L_0}),
    \]
    where we use that the action of the $u_g^g$'s on $\mathfrak{R}(\bigoplus_{g\in G}H^g)$ is trivial by definition. Since the projective $\Diff^+(S^1)$-action on $H_0^G$ is the restriction of the projective $\Diff^+(S^1)$-action on $H_0$, the action of $\widetilde{\exp}(-2\pi i L_0)\in \widetilde{\Mob}$ on $\mathfrak{R}(\bigoplus_{g\in G}H^g)$ is the restriction of the action of $\widetilde{\exp}(-2\pi i L_0)$ on $\bigoplus_{g\in G}H^g$. Therefore, 
\[\vartheta_{\mathfrak{R}(X)} = \mathfrak{R}(\bigoplus\limits_{g\in G}e^{-2\pi i L_0}) = \mathfrak{R}(\theta^G_X),\]
    as needed.
\end{proof}

\appendix
\section{Bi-involutive structures}
\label{Appendix}
In \cite[App. A]{GcrossedbraidedRep}, we constructed an involutive $\mathrm{W}^*$-tensor structure on the $\mathrm{W}^*$-tensor category $\Rep^{\Aut(\A)}(\A)$. By restricting, one also obtains an involutive $\mathrm{W}^*$-tensor structure on $\Rep^G(\A)$, and on $\Rep(\A^G)$. The category $(\Rep^G(\A))^G$ also comes equipped with an involutive $\mathrm{W}^*$-tensor structure by the fact that it is an equivariantization of an involutive $\mathrm{W}^*$-tensor category and $G$ acts by $\mathrm{W}^*$-tensor functors, see below. We note that, in \cite{henriques2024completewcategories}, what we refer to as involutive $\mathrm{W}^*$-tensor categories are called bi-involutive $\mathrm{W}^*$-tensor categories.

An involutive structure on a $\mathrm{W}^*$-tensor category $\mathcal{C}$ consists of an anti-linear $\mathrm{W}^*$-functor $\overline{\,\cdot\,}:\mathcal{C}\to \mathcal{C}$, equipped with a unitary anti-tensor structure $\nu_{X, Y}: \overline{X}\otimes \overline{Y}\to \overline{X\otimes Y}$ natural in $X,Y\in\mathcal{C}$ and compatible with the associator, a unitary isomorphism $r: \1\to \overline{\1}$ and a natural family of unitary isomorphisms $\varphi_X: X\to \overline{\overline{X}}$. The obvious compatibilities between these pieces of data are spelled out, for example, in \cite[Def. 2.1]{MR2861112} or \cite[App. A]{GcrossedbraidedRep}. An involutive $\mathrm{W}^*$-functor between two involutive $\mathrm{W}^*$-tensor categories consists of a $\mathrm{W}^*$-tensor functor $F$ together with a unitary natural transformation $\chi: \overline{F(-)}\cong F(\overline{\,-\,})$. The family $\chi$ satisfies the obvious compatibility conditions with the tensor structure data of $F$, see \cite[App. A]{GcrossedbraidedRep}. Natural transformations between involutive $\mathrm{W}^*$-tensor functors are also defined in the obvious way. We write $\underline{\Aut_{\otimes}^{\overline{\,\cdot\,}}}(\mathcal{C})$ for the tensor category of involutive $\mathrm{W}^*$-tensor automorphisms of $\mathcal{C}$ and unitary involutive natural transformations.

Let $\mathcal{C}$ be an involutive $\mathrm{W}^*$-tensor category and $G$ be a discrete group acting on the underlying $\mathrm{W}^*$-tensor category of $\mathcal{C}$ by a tensor functor $T:\underline{G}\to \underline{\Aut_\otimes(\mathcal{C})}$. Recall that $T$ consists of tensor functors $(T_g, s_g): \mathcal{C}\to \mathcal{C}$ for every $g\in G$ together with unitary natural transformations $n_{g,h}: T_gT_h\cong T_{gh}$ for all $g, h\in G$. A \emph{lift of $T$ to an action by involutive $\mathrm{W}^*$-tensor functors} consists of a unitary family of natural isomorphisms $\chi_g: \overline{T_g(-)}\cong T_g(\overline{\,-\,})$ upgrading $(T_g,s_g)$ to an involutive $\mathrm{W}^*$-tensor functor for every $g\in G$ and making the following diagrams commute 
\begin{equation}\label{eq: CompatibilityBraidingInvolution}\begin{tikzcd}
	{\overline{X\otimes Y}} && {\overline{T_g(Y)\otimes X}} & {\overline{X}\otimes \overline{T_g(Y)}} \\
	{\overline{Y}\otimes \overline{X}} & {T_e(\overline{Y})\otimes \overline{X}} & {T_{g^{-1}}\circ T_g(\overline{Y})\otimes \overline{X}} & {\overline{X}\otimes T_g(\overline{Y})}
	\arrow["{\overline{\beta_{X, Y}}}", from=1-1, to=1-3]
	\arrow["{\nu_{X, Y}}"', from=1-1, to=2-1]
	\arrow["{\nu_{T_g(Y), X}}", from=1-3, to=1-4]
	\arrow["{\id\otimes \chi_g(Y)}", from=1-4, to=2-4]
	\arrow["\cong"', from=2-1, to=2-2]
	\arrow["\cong"', from=2-2, to=2-3]
	\arrow["{\beta_{\overline{X}, T_g(\overline{Y})}}", from=2-4, to=2-3]
\end{tikzcd}\end{equation}
for all $g\in G$, $X\in \mathcal{C}_g$, and $Y\in\mathcal{C}$, and
\begin{equation}\begin{tikzcd}\label{eq: CompatibilityNChi}
	{\overline{T_gT_h(X)}} && {\overline{T_{gh}(X)}} \\
	{T_g(\overline{T_h(X)})} && {T_{gh}(\overline{X})} \\
	& {T_gT_h(\overline{X})}
	\arrow["{\overline{n_{g,h}(X)}}", from=1-1, to=1-3]
	\arrow["{\chi_g(T_h(X))}"', from=1-1, to=2-1]
	\arrow["{\chi_{gh}(X)}", from=1-3, to=2-3]
	\arrow["{T_g(\chi_h(X))}"', from=2-1, to=3-2]
	\arrow["{n_{g,h}(\overline{X})}"', from=3-2, to=2-3]
\end{tikzcd}\end{equation}
for all $g,h\in G$ and $X\in \mathcal{C}.$ In particular, the families $\chi_g$ upgrade $T$ to a tensor functor $T: \underline{G}\to \underline{\Aut_{\otimes}^{\overline{\,\cdot\,}}}(\mathcal{C})$.

A \emph{$G$-crossed braided involutive $\mathrm{W}^*$-tensor category} consists of a $G$-crossed braided $\mathrm{W}^*$-tensor category $\mathcal{C}$, together with a lift of the $G$-action to an action by involutive $\mathrm{W}^*$-tensor functors and such that $\overline{\,\cdot\,}:\mathcal{C}\to \mathcal{C}$ restricts to equivalences $\mathcal{C}_g\to \mathcal{C}_{g^{-1}}$ for all $g\in G$. An \emph{involutive $G$-crossed balanced $\mathrm{W}^*$-tensor category} is an involutive $G$-crossed braided $\mathrm{W}^*$-tensor category $\mathcal{C}$ together with a $G$-crossed balance $\theta$ on its underlying $G$-crossed braided $\mathrm{W}^*$-tensor category such that the diagram
\begin{equation}\begin{tikzcd}\label{eq: RibbonCondition}
	{\overline{T_g(X)}} && {\overline{X}} \\
	&& {T_e(\overline{X})} \\
	{T_g(\overline{X})} && {T_{g^{-1}}\circ T_g(\overline{X})}
	\arrow["{\overline{\theta_X^*}}", from=1-1, to=1-3]
	\arrow["{\chi_g(X)}"', from=1-1, to=3-1]
	\arrow["\cong"', from=2-3, to=1-3]
	\arrow["{\theta_{T_g(\overline{X})}}"', from=3-1, to=3-3]
	\arrow["{n_{g^{-1}, g}(\overline{X})}"', from=3-3, to=2-3]
\end{tikzcd}\end{equation}
commutes for all $g\in G$ and $X\in\mathcal{C}_g$. The commutativity of Diagram \eqref{eq: RibbonCondition} is a generalization to the context of involutive $\mathrm{W}^*$-tensor categories of the usual ribbon condition in rigid semisimple $G$-crossed balanced tensor categories, see \cite[Rk. A.4]{GcrossedbraidedRep} and \cite[Def. 2.3 and Eq. (2.4.a)]{MR2674592}. Functors between involutive $G$-crossed braided and $G$-crossed balanced $\mathrm{W}^*$-tensor categories are defined in the obvious way.

In Section \ref{Sec: G-X-balanced}, we have argued how the equivariantization of a $G$-crossed balanced $\mathrm{W}^*$-tensor category is canonically balanced. We generalize this statement here to the context of involutive $\mathrm{W}^*$-tensor categories.

\begin{lemma}\label{lemm: EquivariantizationIsInvolutive}
    Let $G$ be a discrete group and $(\mathcal{C}, \nu, \varphi, r)$ be an involutive $\mathrm{W}^*$-tensor category. Let $G$ act on $\mathcal{C}$ by a tensor functor $T: \underline{G}\to \underline{\Aut}_{\otimes }^{\overline{\,\cdot\,}}(\mathcal{C})$. Then, the equivariantization $\mathcal{C}^G$ has the canonical structure of an involutive $\mathrm{W}^*$-tensor category. If $\mathcal{C}$ is a $G$-crossed braided (balanced) involutive $\mathrm{W}^*$-tensor category, then $\mathcal{C}^G$ is canonically a braided (resp. balanced) involutive $\mathrm{W}^*$-tensor category.
\end{lemma} 
\begin{proof}
    Given an object $(X,u)\in \mathcal{C}^G$, we define $\overline{(X,u)}:=(\overline{X}, \overline{u})$, where $\overline{u}_g: T_g(\overline{X})\xrightarrow{\chi_g(X)^{-1}}\overline{T_g(X)}\xrightarrow{\overline{u_g}}\overline{X}$. It is easy to see that this produces a well-defined functor $\mathcal{C}^G\to \mathcal{C}^G$, using the compatibility \eqref{eq: CompatibilityNChi} and naturality of the unitaries $\chi$.
    
    The unit of $\mathcal{C}^G$ is given by $(\1,\{i_g^{-1}\}_{g\in G})$, and it is straightforward to check that $r:\1\to \overline{\1}$ provides a valid morphism $(\1,\{i_g^{-1}\}_{g\in G})\cong \overline{(\1,\{i_g^{-1}\}_{g\in G})}$. Similarly, the isomorphisms
    \[
    \varphi_{(X,u)} = \varphi_X \hspace{1cm} \nu_{(X,u),(Y,v)} = \nu_{X,Y}
    \]
    for $(X,u),(Y,v)\in\mathcal{C}^G$ provide the rest of the necessary structure for $\mathcal{C}^G$ to be an involutive $\mathrm{W}^*$-tensor category.

    If $\mathcal{C}$ is a $G$-crossed braided involutive $\mathrm{W}^*$-tensor category, we show that $\mathcal{C}^G$ is a braided involutive $\mathrm{W}^*$-tensor category. Fix objects $(X, u), (Y, v)\in\mathcal{C}^G$. Let $X = \bigoplus\limits_{g\in G} X_g$, with $X_g\in \mathcal{C}_g$. For every $g,h\in G$, recall that we write $u^h_g$ for the restriction of $u_g: T_g(X)\xrightarrow{\cong } X$ to $T_g(X_h)\xrightarrow{\cong} X_{ghg^{-1}}$, and similarly for $Y$ and $v$. The specialization of Equation \eqref{eq: CompatibilityBraidingInvolution} to $(X, u), (Y, v)\in \mathcal{C}^G$ is
    \begin{equation}\label{eq SpecialCompatilityBRaidingInolution}
        \nu^G_{(X, u), (Y, v)} = \beta^G_{\overline{(X, u)}, \overline{(Y, v)}}\circ \nu^G_{(Y, v), (X, u)}\circ \overline{\beta^G_{(X, u), (Y, v)}}.
    \end{equation}
    Given $g,h\in G$, the component of Equation \eqref{eq SpecialCompatilityBRaidingInolution} on $\overline{X_g\otimes Y_h}$ is the commutativity of the outer diagram in
\[\begin{tikzcd}
	{\overline{X_g\otimes Y_h}} && {\overline{T_g(Y_h)\otimes X_g}} && {\overline{Y_h\otimes X_g}} & {\overline{X_g}\otimes \overline{Y_h}} \\
	&& {\overline{X_g}\otimes \overline{T_g(Y_h)}} &&& {T_{g^{-1}}(\overline{Y_h})\otimes \overline{X_g}} \\
	&& {\overline{X_g}\otimes T_g(\overline{Y_h)}} \\
	&& {T_{g^{-1}}\circ T_g(\overline{Y_h})\otimes \overline{X_g}} \\
	{\overline{Y_h}\otimes \overline{X_g}} &&&&& {\overline{T_{g^{-1}}(Y_h)}\otimes\overline{X_g}}.
	\arrow["{\overline{\beta_{X_g, Y_h}}}", from=1-1, to=1-3]
	\arrow["{\nu_{X_g, Y_h}}"', from=1-1, to=5-1]
	\arrow["{\overline{v_g\otimes\id}}", from=1-3, to=1-5]
	\arrow["{\nu_{T_g(Y_h), X_g}}", from=1-3, to=2-3]
	\arrow["{\nu_{Y_h, X_g}}", from=1-5, to=1-6]
	\arrow["{\beta_{\overline{X_g}, \overline{Y_h}}}", from=1-6, to=2-6]
	\arrow["{\id\otimes \overline{v_g}}"{description}, from=2-3, to=1-6]
	\arrow["{\id\otimes \chi_g(Y_h)}", from=2-3, to=3-3]
	\arrow["{\chi_{g^{-1}}^{-1}(Y_h)\otimes \id}", from=2-6, to=5-6]
	\arrow["{\beta_{\overline{X_g}, T_g(\overline{Y_h})}}", from=3-3, to=4-3]
	\arrow["\cong"', from=4-3, to=5-1]
	\arrow["{T_{g^{-1}}(\overline{v_g})\otimes \id}"{description}, from=4-3, to=5-6]
	\arrow["{\overline{v_{g^{-1}}}\otimes\id}", from=5-6, to=5-1]
\end{tikzcd}\]
    The leftmost diagram commutes by Equation \eqref{eq: CompatibilityBraidingInvolution} and the lowest diagram commutes by hypothesis on $v$. The top-right diagram commutes by naturality of $\nu$, and the remaining diagram commutes by naturality of $\beta$ and $\chi$. Hence, the claim follows.

    Assume that $\mathcal{C}$ has, in addition, a $G$-crossed balance $\theta$. Then, we need to show that, for every object $(X,u)\in \mathcal{C}^G$, it holds that
    \begin{equation}
    \label{eq: RibbonConditionUncrossed}
    \theta^G_{\overline{(X,u)}} = \overline{(\theta^G_{(X,u)})^*},
    \end{equation}
    which is the specialization of Equation \eqref{eq: RibbonCondition} when the trivial group is acting. Here, $\theta^G$ is the induced balance on $\mathcal{C}^G$. Given $g\in G$, the $g$-component of Equation \eqref{eq: RibbonConditionUncrossed} is exactly the commutativity of the outer diagram in 
\[\hspace{-.6cm}\begin{tikzcd}[column sep=large]
	{\overline{T_{g^{-1}}(X_{g^{-1}})}} &&&&& {\overline{X_{g^{-1}}}} \\
	& {T_{g^{-1}}(\overline{X_{g^{-1}}})} && {T_e(\overline{X_{g^{-1}}})} \\
	\\
	& {T_g\circ T_{g^{-1}}(\overline{X_{g^{-1}}})} && {\overline{T_e(X_{g^{-1}})}} \\
	& {T_g(\overline{T_{g^{-1}}(X_{g^{-1}})})} && {\overline{T_g\circ T_{g^{-1}}(X_{g^{-1}})}} \\
	\\
	{\overline{X_{g^{-1}}}} & {T_g(\overline{X_{g^{-1}}})} &&&& {\overline{T_g(X_{g^{-1}})}}
	\arrow["{\overline{\theta_{X_{g^{-1}}}^*}}", from=1-1, to=1-6]
	\arrow["{\chi_{g^{-1}}(X_{g^{-1}})}", from=1-1, to=2-2]
	\arrow["{\theta_{T_{g^{-1}}(\overline{X_{g^{-1}}})}}"', from=2-2, to=4-2]
	\arrow["\cong", from=2-4, to=1-6]
	\arrow["{n_{g,g^{-1}}(\overline{X_{g^{-1}}})}", from=4-2, to=2-4]
	\arrow["{\overline{u_e^{g^{-1}}}}"', from=4-4, to=1-6]
	\arrow["{\chi_{e}(X_{g^{-1}})}", from=4-4, to=2-4]
	\arrow["{T_g(\chi_{g^{-1}}(X_{g^{-1}}))}"', from=5-2, to=4-2]
	\arrow["{T_g(\overline{u_{g^{-1}}^{g^{-1}}})}"', from=5-2, to=7-2]
	\arrow["{\overline{n_{g, g^{-1}}(X_{g^{-1}})}}"', from=5-4, to=4-4]
	\arrow["{\chi_{g}(T_{g^{-1}}(X_{g^{-1}}))}", from=5-4, to=5-2]
	\arrow["{\overline{T_g(u_{g^{-1}}^{g^{-1}})}}"', from=5-4, to=7-6]
	\arrow["{\overline{(u_{g^{-1}}^{g^{-1}})^*}}", from=7-1, to=1-1]
	\arrow["{\theta_{\overline{X_{g^{-1}}}}}"', from=7-1, to=7-2]
	\arrow["{\chi_g(\overline{X_{g^{-1}}})^{-1}}"', from=7-2, to=7-6]
	\arrow["{\overline{u_g^{g^{-1}}}}"', from=7-6, to=1-6]
\end{tikzcd}\]
We now argue the commutativity of all the inner diagrams. The upper diagram is the ribbon condition \eqref{eq: RibbonCondition}. The left-most diagram commutes by the naturality of the $G$-crossed balance $\theta$. The central diagram commutes by \eqref{eq: CompatibilityNChi}. The right-most diagram commutes by hypothesis on $u$, and the unique triangle commutes trivially. The lower diagram commutes by naturality of the family $\chi$. Hence, condition \eqref{eq: RibbonConditionUncrossed} is satisfied, and the claim follows.
\end{proof}

Let $\A$ be a conformal net. We now recall from \cite[App. A]{GcrossedbraidedRep} how to endow the $\mathrm{W}^*$-tensor category $\Rep^{\Aut(\A)}(\A)$ with the structure of an involutive $\mathrm{W}^*$-balanced tensor category. If $G$ is a discrete group acting on $\A$ by a group homomorphism $G\to \Aut(\A)$, we can pull back the balanced involutive $\mathrm{W}^*$-tensor structure on $\Rep^{\Aut(\A)}(\A)$ under this homomorphism to obtain a $G$-crossed balanced involutive $\mathrm{W}^*$-tensor structure on $\Rep^G(\A)$.

Recall that we fix the point $\mathrm{p}=1\in S^1$, and that we define $\widetilde{S^1_+} := (0,\nicefrac{1}{2})\in \Jcal_\R$, and $S^1_+:=q(\widetilde{S^1_+})$, the upper semi-circle. The von Neumann algebra $\A(S^1_+)$ acts on $H_0$ by definition, and the vacuum vector $\Omega\in H_0$ is cyclic for the action of $\A(S^1_+)$, by the Reeh-Schlieder Theorem. One can therefore define an anti-linear unbounded operator $S: \A(S^1_+)\Omega\to \A(S^1_+)\Omega$ such that, for every $x\in \A(S^1_+)$,
\[
S\pi_{0, S_+^1}(x)(\Omega) = \pi_{0, S_+^1}(x^*)(\Omega), 
\]
following Tomita-Takesaki theory. The operator $S$ is preclosed and we continue denoting its closure by $S$. Let $S = \mathrm{J}\circ \Delta^{1/2}$ be the polar decomposition of $S$. Then, the anti-unitary map $\mathrm{J}: H_0\to H_0$ is called the modular conjugation and satisfies $\mathrm{J}\Omega = \Omega$ and $\mathrm{J}^2 = \id_{H_0}$. By the Bisognano-Wichmann Theorem \cite{bgl93}, it also holds that, for every $I\in\Jcal$ and $x\in \A(I)$,
\[
\mathrm{J}\circ \pi_{0, I}(x)\circ \mathrm{J}\in \Hom_{\A(\overline{I})^c}(H_0,H_0)\cong \A(\overline{I}),
\]
where $\overline{I} := \{\overline{z}\,|\, z\in I\}\in \Jcal$ is the image of $I$ under complex conjugation in $S^1$. Hence, we obtain a von Neumann algebra isomorphism
\[
\begin{array}{cccc}
     j: &\A(I) &\to & \A(\overline{I})^{\text{op}} \\
     & x&\mapsto & j(x),
\end{array}
\]
where $j(x)$ is the unique element in $\A(\overline{I})$ such that $\pi_{0, \overline{I}}(j(x)) = \mathrm{J}\circ \pi_{0, I}(x^*)\circ \mathrm{J}$. The isomorphisms $j$ are compatible with inclusions of intervals and commute with automorphisms of~$\A$.

Fix an automorphism $\varphi\in\Aut(\A)$ and let $(H^\varphi, \pi^H)\in \Rep^\varphi(\A)$ be a $\varphi$-twisted representation of $\A$. We define a new representation $\overline{(H^\varphi, \pi^H)}:=(\overline{H^\varphi}, \overline{\pi^H})\in \Rep^{\varphi^{-1}}(\A)$. Here, $\overline{H^\varphi}$ denotes the complex conjugate Hilbert space of $H^\varphi$ and $\overline{\pi^H}$ is defined as follows. Let $\tilde{I}\in \Jcal_{\R}$ and $x\in \A_\R(\tilde{I})$. Then, we write
\begin{equation}\label{eq: defInvolutionRepGA}
\overline{\pi^H}_{\tilde{I}}(x)(\overline{\xi}) = \overline{\pi^H_{-\tilde{I}}(j(x)^*)(\xi)}.
\end{equation}
for all $\overline{\xi}\in \overline{H^\varphi}$. It is easy to see that Equation \eqref{eq: defInvolutionRepGA} indeed defines a $\varphi^{-1}$-twisted action of $\A$ on $\overline{H^\varphi}$. It is also clear that, if $K^\varphi\in \Rep^\varphi(\A)$ is a twisted representation and $F: H^\varphi\to K^\varphi$ is a morphism of $\varphi$-twisted representations, the complex conjugate map $\overline{F}: \overline{H^\varphi}\to \overline{K^\varphi}$ intertwines the $\varphi^{-1}$-twisted actions $\overline{\pi^H}$ and $\overline{\pi^K}$ of $\A$. We write $\overline{\,\cdot\,}: \Rep^{\Aut(\A)}(\A)\to \Rep^{\Aut(\A)}(\A)$ for the anti-linear functor that sends a twisted representation $(H^\varphi, \pi^H)\in \Rep^\varphi(\A)$ to $(\overline{H^\varphi}, \overline{\pi^H})$, and a morphism of twisted representations $F: H^\varphi\to K^\varphi$ to the complex conjugate morphism $\overline{F}: \overline{H^\varphi}\to \overline{K^\varphi}$.

The functor $\overline{\,\cdot\,}$ is upgraded to an involutive $\mathrm{W}^*$-tensor structure on $\Rep^{\Aut(\A)}(\A)$ as follows. The compatibility with the unit is given by the unitary $i:(H_0, \pi_0)\xrightarrow{\cong} \overline{(H_0,\pi_0)}$, $\xi\mapsto \overline{\mathrm{J}\xi}$, and the trivialization of the square of $\overline{\,\cdot\,}$ is defined to be the family $\phi_{H}: (H^\varphi,\pi^H)\xrightarrow{\cong}\overline{\overline{(H^\varphi, \pi^H)}}$, $\xi\mapsto \overline{\overline{\xi}}$, for $H\in\Rep^{\Aut(\A)}(\A)$.

Let us next introduce the anti-tensor structure $\nu$. Given an interval $\tilde{I}\in \Jcal_\R$ and a twisted representation $H^\varphi\in \Rep^\varphi(\A)$, there are canonical isomorphisms $H_-^\varphi(\tilde{I})\cong \overline{H^\varphi}_+(-\tilde{I})$ and $H_+^\varphi(\tilde{I})\cong \overline{H^\varphi}_-(-\tilde{I})$ given by $\xi\mapsto \overline{\xi}$. In addition, for $\xi\in H_{\pm}^\varphi(\tilde{I})$, we have
\[
Z^{\mp}(\overline{\xi}, -\tilde{I}) = \overline{Z^{\pm}(\xi, \tilde{I})}\circ i.
\]
Therefore, given another twisted representation $K^\mu\in \Rep^\mu(\A)$, the map $\overline{H^\varphi}_+(\tilde{I})\otimes \overline{K^\mu}\to \overline{K^\mu\otimes H^\varphi_-(-\tilde{I})}$ given by $\overline{\xi}\otimes\overline{\eta}\mapsto \overline{\eta\otimes\xi}$ defines a unitary isomorphism
\[
\nu^{\tilde{I}}_{H, K}: \overline{H^\varphi}_+(\tilde{I})\boxtimes \overline{K^\mu}\to \overline{K^\mu\boxtimes H^\varphi_-(-\tilde{I})},
\]
see \cite[Prop. A.9]{GcrossedbraidedRep}. The anti-tensor structure $\nu$ of $\overline{\,\cdot\,}$ is then defined, on the twisted representations $H^\varphi$ and $K^\mu$,~as
\[
\nu_{H, K}: \overline{H^\varphi}\boxtimes \overline{K^\mu} = \overline{H^\varphi}_+(\widetilde{S^1_-})\boxtimes \overline{K^\mu}\xrightarrow{\nu_{H, K}^{\widetilde{S^1_-}}}\overline{K^\mu\boxtimes H_-^\varphi(-\widetilde{S_-^1})} = \overline{K^\mu\boxtimes H_-^\varphi(\widetilde{S_+^1})} = \overline{K^\mu\boxtimes H^\varphi}.
\]
The lift of the action $T$ of $\Aut(\A)$ on $\Rep^{\Aut(\A)}(\A)$ to an action by involutive $\mathrm{W}^*$-tensor automorphisms is given trivially by
\[
T_\mu(\overline{\,\cdot\,}) = \overline{T_\mu(-)}
\]
for all $\mu\in\Aut(\A)$. The data presented here upgrades $\Rep^{\Aut(\A)}(\A)$ to an involutive $\Aut(\A)$-crossed balanced $\mathrm{W}^*$-tensor category.

\begin{theorem}(\cite[Thm. A.15]{GcrossedbraidedRep})
     Let $\A$ be a conformal net. The $\Aut(\A)$-crossed balanced $\mathrm{W}^*$-tensor category $\Rep^{\Aut(\A)}(\A)$ is naturally an $\Aut(\A)$-crossed balanced involutive $\mathrm{W}^*$-tensor category.
\end{theorem}

If $G$ is a discrete group acting on $\A$ by a group homomorphism $\Phi: G\to \Aut(\A)$, we obtain that $\Rep^G(\A)$ is canonically a $G$-crossed balanced involutive $\mathrm{W}^*$-tensor category.

\begin{corollary}
    Let $\A$ be a conformal net being acted on by a discrete group $G$. Then, the $G$-crossed balanced $\mathrm{W}^*$-tensor category $\Rep^{G}(\A)$ is naturally a $G$-crossed balanced involutive $\mathrm{W}^*$-tensor category.  
\end{corollary}

 We are now ready to show that the equivalence of balanced $\mathrm{W}^*$-tensor categories $\Rep(\A^G)\cong (\Rep^G(\A))^G$ in Theorem \ref{Thm: EquivariantizationCompatibleBalance}, for $\A$ being acted on faithfully by a finite group $G$, can be upgraded to an equivalence of balanced involutive $\mathrm{W}^*$-tensor categories. We will need the following two lemmas.

 \begin{lemma}\label{lemm: compatibilityLRnu}
    Fix automorphisms $\varphi, \mu\in \Aut(\A)$ and an interval $\tilde{I}\in \Jcal_\R$. Then, for any twisted representations $H^\varphi\in\Rep^\varphi(\A)$ and $K^\mu\in\Rep^\mu(\A)$, and any $\xi\in H^\varphi_-(\tilde{I})$, it holds that
    \[
    \nu_{H, K}\circ L(\overline{\xi}, -\tilde{I}) = \overline{R(\xi, \tilde{I})}
    \]
    when acting on $\overline{K^\mu}$. Moreover, if $\tilde{J}\in\Jcal_\R$ is an interval such that $\tilde{I}\subset\tilde{J}^{c+}$ and $\eta\in K^\mu_+(\tilde{J})$, it holds that
    \[
    \nu_{H, K}\circ L(\overline{\xi}, -\tilde{I})\overline{\eta} = \overline{L(\eta, \tilde{J})\xi}.
    \]
 \end{lemma}
 \begin{proof}
     We prove the first statement. Let $\alpha_{-\tilde{I}}$ be a path in $\R$ from $-\tilde{I}$ to $\widetilde{S^1_-}$. We have to show the commutativity of the outer diagram in 
\[\begin{tikzcd}
	{\overline{K^\mu}} && {\overline{H^\varphi}_{+}(-\tilde{I})\boxtimes \overline{K^\mu}} && {\overline{H^\varphi}_{+}(-)\boxtimes \overline{K^\mu}} \\
	&& {\overline{K^\mu\boxtimes H^\varphi_-(\tilde{I})}} && {\overline{K^\mu\boxtimes H^\varphi_-(+)}}.
	\arrow["{Z^+(\overline{\xi}, -\tilde{I})}", from=1-1, to=1-3]
	\arrow["{\overline{Z^-(\xi, \tilde{I})}}"', from=1-1, to=2-3]
	\arrow["{\alpha_{-\tilde{I}}^\bullet}", from=1-3, to=1-5]
	\arrow["{\nu^{-\tilde{I}}_{H, K}}", from=1-3, to=2-3]
	\arrow["{\nu_{H, K}^{\widetilde{S^1_-}}}", from=1-5, to=2-5]
	\arrow["{\overline{(-\alpha_{-\tilde{I}})^\bullet}}", from=2-3, to=2-5]
\end{tikzcd}\]
The right square commutes by \cite[Lem. A.10]{GcrossedbraidedRep}, and the left triangle commutes trivially. For the second statement, let $\tilde{J}\in\Jcal_\R$ be an interval such that $\tilde{I}\subset\tilde{J}^{c+}$, and let $\eta\in K^\mu_+(\tilde{J})$. Then, 
\[
\nu_{H, K}\circ L(\overline{\xi}, -\tilde{I}) \overline{\eta}= \overline{R(\xi, \tilde{I})\eta} = \overline{R(\xi, \tilde{I})L(\eta, \tilde{J})\Omega} = \overline{L(\eta, \tilde{J})R(\xi, \tilde{I})\Omega} = \overline{L(\eta, \tilde{J})\xi}.
\]
 \end{proof}

 \begin{lemma}\label{lemm: compatibilityLJ}
     Given intervals $\tilde{I}, \tilde{J}\in \Jcal_\R$ and any $x\in\A(I)$ and $y\in \A(J)$, it holds that
\[
\mathrm{J}\mu L^G(\pi_0(x)\Omega, \tilde{I})\pi_0(y)\Omega = \mu L^G(\pi_0(j(x)^*)\Omega, -\tilde{I})\pi_0(j(y)^*)\Omega.
\]
\begin{proof}
We compute
\begin{align*}
    \mathrm{J}\mu L^G(\pi_0(x)\Omega, \tilde{I})\pi_0(y)\Omega & = \mathrm{J}\pi_0(x)\pi_0(y)\Omega\\ & = \mathrm{J}\pi_0(x)\mathrm{J}\mathrm{J}\pi_0(y)\mathrm{J}\Omega\\ & = \pi_0(j(x)^*)\pi_0(j(y)^*)\Omega\\ & = \mu L^G(\pi_0(j(x)^*)\Omega, -\tilde{I})\pi_0(j(y)^*)\Omega,
\end{align*}
as needed.
\end{proof}
 \end{lemma}

We write $\upsilon$ for the anti-tensor structure of the involution on $\Rep(\A^G)$, so that we can reserve $\nu$ for the anti-tensor structure of the involution on $\Rep^G(\A)$ and $\nu^G$ for the anti-tensor structure on the equivariantization $(\Rep^G(\A))^G$. We can now prove the main result of the Appendix.

\begin{theorem}\label{thm: Fixed-PointsEquivalenceInvolutive}
     Let $\A$ be a conformal net being acted on faithfully by a finite group $G$. Then, the functors
\[\begin{tikzcd}
	{\Rep(\A^G)} && {(\Rep^G(\A))^G}
	\arrow["{\mathfrak{M}^G\circ\mathfrak{D}}", shift left=2, from=1-1, to=1-3]
	\arrow["{\mathfrak{R}}", shift left=2, from=1-3, to=1-1]
\end{tikzcd}\]
provide an equivalence of balanced involutive $\mathrm{W}^*$-tensor categories.
\end{theorem}
\begin{proof}
    In Theorem \ref{Thm: EquivariantizationCompatibleBalance}, we have proved that $\mathfrak{M}^G\circ\mathfrak{D}$ and $\mathfrak{R}$ provide an equivalence of balanced $\mathrm{W}^*$-tensor categories. We will upgrade the composition $\mathfrak{M}^G\circ \mathfrak{D}$ to an involutive $\mathrm{W}^*$-tensor functor. Let us denote $\mathfrak{X}:=\mathfrak{M}^G\circ \mathfrak{D}$.
    
    Let $H = (H, \pi^H)\in \Rep(\A^G)$. Then, $\overline{\mathfrak{X}(H)}$ has underlying Hilbert space $\overline{\Gamma\boxtimes^G H}$ with twisted $\A$-action $\overline{\pi^{\mathfrak{X}(H)}}$ given~by
    \[
    \overline{\pi^{\mathfrak{X}(H)}}_{\tilde{I}}(x)(\overline{L^G(\pi_0(y)\Omega, \tilde{J})\xi}) = \overline{(\mu\boxtimes^G\id)\circ L^G(\pi_0(j(x)^*)\Omega, -\tilde{I})L^G(\pi_0(y)\Omega, \tilde{J})\xi}
    \]
    for all $\tilde{I}, \tilde{J}\in\Jcal_\R$, $x\in \A_\R(\tilde{I})$, $y\in\A(J)$ and $\xi\in H$. The $G$-equivariance data is given, for all $g\in G$ and $\xi\in \Gamma\boxtimes^G H$, by
    \[
    T_g(\overline{\Gamma\boxtimes^G H}) = \overline{T_g(\Gamma\boxtimes^G H)}\xrightarrow{\overline{V_g\boxtimes^G\id}} \overline{\Gamma\boxtimes^G H}.
    \]
    On the other hand $\mathfrak{X}(\overline{H})$ has underlying Hilbert space $\Gamma\boxtimes^G\overline{H}$ with twisted $\A$-action given~by
    \[
    \pi^{\mathfrak{X}(\overline{H})}_{\tilde{I}}(x)(L^G(\pi_0(y)\Omega, \tilde{J})\overline{\xi}) = (\mu\boxtimes^G\id)\circ L^G(\pi_0(x)\Omega,\tilde{I}) L^G(\pi_0(y)\Omega, \tilde{J})\overline{\xi},
    \]
    for all $\tilde{I}, \tilde{J}\in\Jcal_\R$, $x\in \A_\R(\tilde{I})$, $y\in\A(J)$ and $\xi\in H$. The $G$-equivariance data is given by $T_g(\Gamma\boxtimes^G\overline{H})\xrightarrow{V_g\boxtimes^G\id}\Gamma\boxtimes^G\overline{H}$ for all $g\in G$. We define the unitary isomorphism of Hilbert~spaces
    \[
    \begin{array}{cccc}
        \chi_{H}^{\phantom{.}}: & \overline{\Gamma\boxtimes^G H}   &\to & \Gamma\boxtimes^G\overline{H} \\
    &  \overline{L^G(\pi_0(x)\Omega, \tilde{I})\xi}   &\mapsto  & L^G(\pi_0(j(x)^*)\Omega,-\tilde{I})\overline{\xi} 
    \end{array}
    \]
    for any $\tilde{I}\in\Jcal_\R$, and where $x\in\A(I)$ and $\xi\in H$. The unitary is well-defined and independent of the choice of $\tilde{I}$ as it equals the composition
    \[
    \overline{\Gamma\boxtimes^G H}\xrightarrow{\overline{\beta_{\Gamma, H}}}\overline{H\boxtimes^G\Gamma}\xrightarrow{\upsilon_{\Gamma, H}^{-1}}\overline{\Gamma}\boxtimes^G\overline{H}\xrightarrow{i^{-1}\boxtimes^G\id}\Gamma\boxtimes^G\overline{H}.
    \]
     We will show that $\chi_H^{\phantom{.}}$ is a morphism in $(\Rep^G(\A))^G$. Let $\tilde{I}, \tilde{J}\in\Jcal_\R$ be intervals such that $-\tilde{I}\subset\tilde{J}^{c-}$ and such that there exists $\tilde{O}\in\Jcal_\R$ with $-\tilde{I}\cup\tilde{J}\subset\tilde{O}$. Then, for every $x\in\A_\R(\tilde{I})$, $y\in\A_\R(\tilde{J})$ and $\xi\in H$, it holds that
    \begin{align*}
        \chi_{H}^{\phantom{.}}\circ \overline{\pi^{\mathfrak{X}(H)}}_{\tilde{I}}(x)(\overline{L^G(\pi_0(y)\Omega, \tilde{J})\xi}) &  =\chi_{H}^{\phantom{.}}(\overline{(\mu\boxtimes^G\id)\circ L^G(\pi_0(j(x)^*)\Omega, -\tilde{I
    })L^G(\pi_0(y)\Omega, \tilde{J})\xi})\\ & = \chi_{H}^{\phantom{.}}(\overline{L^G(\mu L^G(\pi_0(j(x)^*)\Omega, -\tilde{I})\pi_0(y)\Omega, \tilde{O})\xi})\\ & = L^G(\mathrm{J}\mu L^G(\pi_0(j(x)^*)\Omega, -\tilde{I})\pi_0(y)\Omega, -\tilde{O})\overline{\xi}
    \\& = L^G(\mu L^G(\pi_0(x)\Omega, \tilde{I})\pi_0(j(y)^*)\Omega, -\tilde{O})\overline{\xi},\\ &= (\mu\boxtimes^G\id)\circ L^G(\pi_0(x)\Omega, \tilde{I})\circ L^G(\pi_0(j(y)^*)\Omega, -\tilde{J})\overline{\xi})\\& = \pi_{\tilde{I}}^{\mathfrak{X}(\overline{H})}(x)\circ \chi_{H}^{\phantom{.}}(\overline{L^G(\pi_0(y)\Omega, \tilde{J})\xi})
    \end{align*}
    using Lemma \ref{lemm: compatibilityLJ}. Hence, $\chi_H^{\phantom{.}}$ is a morphism of twisted $\A$-representations. Now, given $g\in G$,
    \begin{align*}
        \chi_H\circ (\overline{V_g\boxtimes^G\id})\overline{L^G(\pi_0(y)\Omega, \tilde{J})\xi} & = \chi_H^{\phantom{.}}(\overline{L^G(V_g\pi_0(y)\Omega, \tilde{J})}\xi)\\ & = L^G(\mathrm{J}\, V_g\,\pi_0(y)\Omega, -\tilde{J})\overline{\xi}\\ & =L^G(V_g\,\mathrm{J}\,\pi_0(y)\Omega, -\tilde{J})\overline{\xi}
       \\& 
        = (V_g\boxtimes^G \id)\circ \chi_{H}^{\phantom{.}}(\overline{L^G(\pi_0(y)\Omega, \tilde{J})\xi}),
    \end{align*}
    which shows that $\chi_{H}^{\phantom{.}}$ is compatible with the $G$-equivariance data. It is clear that $\chi_H^{\phantom{.}}$ is natural in $H$. Hence, it remains to show that $\chi$ upgrades $\mathfrak{X} = \mathfrak{M}^G\circ \mathfrak{D}$ to an involutive $\mathrm{W}^*$-tensor functor. We only show compatibility with the tensorator, the rest of the conditions, which can be found below \cite[Def. A.1]{GcrossedbraidedRep} follow analogously. We have to show the commutativity of the following diagram for all $H,~ K\in ~\Rep(\A^G),$
\[\begin{tikzcd}
	{\overline{\Gamma\boxtimes^G H}\boxtimes \overline{\Gamma\boxtimes^G K}} && {(\Gamma\boxtimes^G\overline{H})\boxtimes(\Gamma\boxtimes^G\overline{K})} && {(\Gamma\boxtimes^G\overline{H})\boxtimes_\Gamma(\Gamma\boxtimes^G\overline{K})} \\
	{\overline{(\Gamma\boxtimes^G K)\boxtimes (\Gamma\boxtimes^G H)}} &&&& {\Gamma\boxtimes^G\overline{H}\boxtimes^G\Gamma\boxtimes^G\overline{K}} \\
	{\overline{(\Gamma\boxtimes^G K)\boxtimes_\Gamma (\Gamma\boxtimes^G H)}} &&&& {\Gamma\boxtimes^G\overline{H}\boxtimes^G H_0^G\boxtimes^G\overline{K}} \\
	{\overline{\Gamma\boxtimes^G K\boxtimes^G \Gamma\boxtimes^G H}} &&&& {\Gamma\boxtimes^G\overline{H}\boxtimes^G\overline{K}} \\
	{\overline{\Gamma\boxtimes^G K\boxtimes^G H_0^G\boxtimes^G H}} && {\overline{\Gamma\boxtimes^G K\boxtimes^G H}} && {\Gamma\boxtimes^G\overline{K\boxtimes^G H}},
	\arrow["{\chi_H^{\phantom{.}}\boxtimes\chi_{K}^{\phantom{.}}}", from=1-1, to=1-3]
	\arrow["{\nu_{\Gamma\boxtimes^GH, \Gamma\boxtimes^G K}}"', from=1-1, to=2-1]
	\arrow["{\Phi_{\Gamma\boxtimes^G \overline{H}, \Gamma\boxtimes^G\overline{K}}}"', from=1-5, to=1-3]
	\arrow["{\mu_{\Gamma\boxtimes^G\overline{H}, \Gamma\boxtimes^G\overline{K}}}"', from=2-5, to=1-5]
	\arrow["{\overline{\Phi_{\Gamma\boxtimes^GK, \Gamma\boxtimes^G H}}}", from=3-1, to=2-1]
	\arrow["{\id\boxtimes^G\iota\boxtimes^G\id}"', from=3-5, to=2-5]
	\arrow["{\overline{\mu_{\Gamma\boxtimes^GK, \Gamma\boxtimes^GH}}}", from=4-1, to=3-1]
	\arrow["\cong"', from=4-5, to=3-5]
	\arrow["{\id\boxtimes^G\upsilon_{H, K}}", from=4-5, to=5-5]
	\arrow["{\overline{\id\boxtimes^G\iota\boxtimes^G\id}}", from=5-1, to=4-1]
	\arrow["\cong", from=5-3, to=5-1]
	\arrow["{\chi^{\phantom{.}}_{K\boxtimes^G H}}"', from=5-3, to=5-5]
\end{tikzcd}\]
where we have already substituted the tensorator for $\mathfrak{M}^G\circ\mathfrak{D}$, following Proposition \ref{prop: MEquivalenceOfGCats} and Theorem \ref{thm: RepFixedPointsIsEquivTwistedReps}. Let $\tilde{I}, \tilde{J}\in\Jcal_\R$ be intervals with $\tilde{I}\subset\tilde{J}^{c+}$, $x\in \A(J)$, $\xi\in H(I)$ and $\eta\in K(J)$. Then, $\overline{L^G(\pi_0(x)\Omega, \tilde{J})L^G(\eta, \tilde{J})\xi}\in \overline{\Gamma\boxtimes^GK\boxtimes^G H}$ is a vector of the representation at the bottom centre of the diagram. Let us compute the action of the left leg of the diagram on this vector. We have,
\[
\begin{alignedat}{3}
\overline{L^G(\pi_0(x)\Omega, \tilde{J})L^G(\eta, \tilde{J})\xi}
&\xmapsto{\makebox[8em][c]{}}
&\qquad&
\overline{L^G(L^G(\pi_0(x)\Omega, \tilde{J})\eta, \tilde{J})L^G(\Omega, \tilde{I})\xi}\\
&\xmapsto{\makebox[8em][c]{$\overline{\mu_{\Gamma\boxtimes^GK,\,\Gamma\boxtimes^GH}}$}}
&&
\overline{L^\Gamma(L^G(\pi_0(x)\Omega, \tilde{J})\eta, \tilde{J})L^G(\Omega, \tilde{I})\xi}\\
&\xmapsto{\makebox[8em][c]{$\overline{\Phi_{\Gamma\boxtimes^GK,\,\Gamma\boxtimes^GH}}$}}
&&
\overline{L(L^G(\pi_0(x)\Omega, \tilde{J})\eta, \tilde{J})L^G(\Omega, \tilde{I})\xi}\\
&\xmapsto{\makebox[8em][c]{$\nu_{\Gamma\boxtimes^GH,\,\Gamma\boxtimes^GK}^{-1}$}}
&&
L(\overline{L^G(\Omega, \tilde{I})\xi}, -\tilde{I})
\overline{L^G(\pi_0(x)\Omega, \tilde{J})\eta}\\
&\xmapsto{\makebox[8em][c]{$\chi_H^{\phantom{.}}\boxtimes\chi_K^{\phantom{.}}$}}
&&
L(L^G(\Omega,-\tilde{I})\overline{\xi}, -\tilde{I})
L^G(\pi_0(j(x)^*)\Omega, -\tilde{J})
\overline{\eta},
\end{alignedat}
\]
where we use Lemma \ref{lemm: compatibilityLRnu}. Similarly, the image of $\overline{L^G(\pi_0(x)\Omega, \tilde{J})L^G(\eta, \tilde{J})\xi}$ under the right leg of the diagram is
\[
\begin{alignedat}{3}\hspace{-.3cm}
\overline{L^G(\pi_0(x)\Omega, \tilde{J})L^G(\eta, \tilde{J})\xi}
&\xmapsto{\makebox[8em][c]{$\chi^{\phantom{.}}_{H\boxtimes^G K}$}}
&\qquad&
L^G(\pi_0(j(x)^*)\Omega, -\tilde{J})\overline{L^G(\eta, \tilde{J})\xi}\\
&\xmapsto{\makebox[8em][c]{$\mathrm{id}\boxtimes^G\upsilon_{H, K}^{-1}$}}
&&
L^G(\pi_0(j(x)^*)\Omega, \tilde{J})L^G(\overline{\xi}, -\tilde{I})\overline{\eta}\\
&\xmapsto{\makebox[8em][c]{}}
&&
L^G(\pi_0(j(x)^*)\Omega, -\tilde{J})L^G(\overline{\xi}, -\tilde{I})L^G(\Omega,-\tilde{I})\overline{\eta}\\
&\xmapsto{\makebox[8em][c]{$\mu_{\Gamma\boxtimes^G\overline{H},\,\Gamma\boxtimes^G\overline{K}}$}}
&&
\mu_{\Gamma\boxtimes^G\overline{H},\,\Gamma\boxtimes^G\overline{K}}
\Bigl(
L^G(\pi_0(j(x)^*)\Omega, -\tilde{J})
L^G(\overline{\xi}, -\tilde{I})
L^G(\Omega,-\tilde{I})
\overline{\eta}
\Bigr)
\\
&=&\qquad&
\mu_{\Gamma\boxtimes^G\overline{H},\,\Gamma\boxtimes^G\overline{K}}
\Bigl(
L^G(\Omega, -\tilde{I})
L^G(\overline{\xi}, -\tilde{I})
L^G(\pi_0(j(x)^*)\Omega, -\tilde{J})
\overline{\eta}
\Bigr)\\
&=&&
L^\Gamma(L^G(\Omega, -\tilde{I})\overline{\xi}, -\tilde{I})
L^G(\pi_0(j(x)^*)\Omega, -\tilde{J})
\overline{\eta}\\
&\xmapsto{\makebox[8em][c]{$\Phi_{\Gamma\boxtimes^G\overline{H},\,\Gamma\boxtimes^G\overline{K}}$}}
&&
L(L^G(\Omega, -\tilde{I})\overline{\xi}, -\tilde{I})
L^G(\pi_0(j(x)^*)\Omega, -\tilde{J})
\overline{\eta},
\end{alignedat}
\]
where we use that $\mu_{\Gamma\boxtimes^G\overline{H}, \Gamma\boxtimes^G\overline{K}}$ equalizes the actions of $\Gamma$ on the left and on the right. Indeed, $L^G(\Omega, -\tilde{I})
L^G(\overline{\xi}, -\tilde{I})
L^G(\pi_0(j(x)^*)\Omega, -\tilde{J})
\overline{\eta}$ and $L^G(\pi_0(j(x)^*)\Omega, -\tilde{J})
L^G(\overline{\xi}, -\tilde{I})
L^G(\Omega,-\tilde{I})
\overline{\eta}$ are the images of $L^G(\Omega, -\tilde{I})L^G(\overline{\xi}, -\tilde{I})L^G(\pi_0(j(x)^*), -\tilde{J})L^G(\Omega, \tilde{J})\overline{\eta}\in \Gamma\boxtimes^G\overline{H}\boxtimes^G\Gamma\boxtimes^G\Gamma\boxtimes^G\overline{K}$ under $\id_{\Gamma\boxtimes^G\overline{H}}\boxtimes^G\mu\boxtimes^G\id_{\overline{K}}$ and $(\mu\boxtimes^G\id_{\overline{H}
\boxtimes^G\Gamma\boxtimes^G\overline{K}})\circ (\beta^{-1}_{\Gamma, \Gamma\boxtimes^G\overline{H}}\boxtimes^G\id_{\Gamma\boxtimes^G\overline{K}})$ respectively.
\end{proof}

\bibliographystyle{halpha-abbrv}
\bibliography{FixedPoints}
\end{document}